\documentclass[a4paper,11pt,reqno]{amsart}

\usepackage[utf8]{inputenc}
\usepackage[T1]{fontenc}
\usepackage{amsfonts,amssymb}
\usepackage{graphicx}
\usepackage{array}
\usepackage{enumerate}
\usepackage{url}
\usepackage{tikz}
\usetikzlibrary{arrows}
\usepackage{xargs}
\usepackage{stmaryrd}
\usepackage{hyperref}
\usepackage{mathrsfs} 

\DeclareFontFamily{U}{mathb}{}
\DeclareFontShape{U}{mathb}{m}{n}{ <5> <6> <7> <8> <9> <10> <12> gen * mathb <11> mathb10}{}
\DeclareSymbolFont{mathb}{U}{mathb}{m}{n}
\DeclareMathSymbol{\Asterisk}     {2}{mathb}{"06}

\usepackage{comment}

\usepackage{empheq}
\numberwithin{equation}{section}

\usepackage{amsthm}
\theoremstyle{plain}

\newtheorem{theorem}{Theorem}[section]
\newtheorem{lemma}[theorem]{Lemma}
\newtheorem{corollary}[theorem]{Corollary}

\newtheorem{proposition}[theorem]{Proposition}
\newtheorem{observation}[theorem]{Observation}

\theoremstyle{definition}
\newtheorem{definition}[theorem]{Definition}
\newtheorem{convention}[theorem]{Convention}
\newtheorem{example}[theorem]{Example}
\theoremstyle{remark}
\newtheorem{remark}[theorem]{Remark}

\newcommand{\R}{\mathbb{R}}
\newcommand{\C}{\mathbb{C}}
\renewcommand{\H}{\mathbb{H}}
\newcommand{\Q}{\mathbb{Q}}
\newcommand{\Z}{\mathbb{Z}}
\newcommand{\N}{\mathbb{N}}
\newcommand{\A}{\mathbb{A}}
\newcommand{\J}{\mathbb{J}}
\newcommand{\F}{\mathbb{F}}

\newcommand{\calO}{\mathcal{O}}

\newcommand{\defeq}{\mathrel{\mathop{:}}=}
\newcommand{\eqdef}{=\mathrel{\mathop{:}}}
\newcommand{\abs}[1]{\lvert #1 \rvert}
\newcommand{\norm}[1]{\lVert #1 \rVert}
\newcommand{\gen}[1]{\langle #1 \rangle}

\newcommand{\id}{\operatorname{id}}
\newcommand{\bigast}{\mathop{\Asterisk}}
\newcommand{\Hom}{\operatorname{Hom}}
\newcommand{\cato}{\textsc{cat}(\oldstylenums{0})}
\newcommand{\into}{\hookrightarrow}
\newcommand{\GL}{\operatorname{GL}}
\newcommand{\SL}{\operatorname{SL}}

\newcommand{\bfG}{\mathbf{G}}
\newcommand{\bfH}{\mathbf{H}}
\newcommand{\bfP}{\mathbf{P}}
\newcommand{\bfQ}{\mathbf{Q}}
\newcommand{\bfR}{\mathbf{R}}

\newcommand{\bfS}{\mathbf{S}}
\newcommand{\bfT}{\mathbf{T}}
\newcommand{\bfU}{\mathbf{U}}
\newcommand{\bfL}{\mathbf{L}}

\newcommand{\bfZ}{\mathbf{Z}}
\newcommand{\calK}{\mathcal{K}}
\newcommand{\calV}{\mathcal{V}}
\newcommand{\calC}{\mathcal{C}}

\newcommand{\Stab}{\operatorname{Stab}}
\newcommand{\Fix}{\operatorname{Fix}}
\newcommandx{\X}[1][1={}]{X^*_{#1}}
\newcommandx{\coX}[1][1={}]{X_*^{#1}}

\newcommand{\TGod}{\bar{\mathbf{T}}}
\newcommand{\PGod}{\bar{\mathbf{P}}}
\newcommand{\OmegaGod}{\bar{\Omega}}

\newcommand{\Ch}{\mathcal{C}}
\newcommand{\Vt}{\mathcal{V}}

\newcommand{\VR}{\operatorname{VR}}
\newcommand{\FP}{\textit{FP}}
\newcommand{\Cay}{\operatorname{Cay}}
\newcommand{\Ner}{\operatorname{Ner}}
\newcommand{\Gal}{\operatorname{Gal}}
\newcommand{\Isom}{\operatorname{Isom}}
\newcommand{\Aut}{\operatorname{Aut}}

\newcommand{\pr}{\operatorname{pr}}
\newcommand{\rk}{\operatorname{rk}}

\newcommand{\typ}{\operatorname{typ}}

\renewcommand{\min}{\operatorname{min}}
\newcommand{\fin}{^{\text{fin}}}
\renewcommand{\inf}{^{\text{inf}}}
\newcommand{\infimum}{\operatorname{inf}}

\renewcommand{\setminus}{\smallsetminus}

\newcommand{\chr}{\operatorname{char}}

\DeclareMathOperator{\NC}{NCo}

\newcommand{\newcomment}[4]{%
\newcounter{#2counter}
\expandafter\newcommand\csname #1\endcsname[1]{%
\refstepcounter{#2counter}%
{\color{#4}(#3\arabic{#2counter})}\marginpar{\scriptsize\raggedright\textbf{\color{#4}(#2 \arabic{#2counter}):} ##1}%
}}
\reversemarginpar

\definecolor{darkgreen}{rgb}{0,0.6,0}

\begin{document}
\title[Finiteness properties of arithmetic approximate lattices]{Higher finiteness properties of\\arithmetic approximate lattices:\\The Rank Theorem for number fields}
\date{\today}


\author[T.~Hartnick]{Tobias Hartnick}
\address{Karlsruher Institut für Technologie, D-76128 Karlsruhe, Germany}
\email{tobias.hartnick@kit.edu}

\author[S.~Witzel]{Stefan Witzel}
\address{Mathematisches Institut, JLU Gießen, Arndtstr.\ 2, D-35392 Gießen, Germany}
\thanks{S.W.\ was supported through the DFG projects WI 4079/2, WI 4079/6, and a Feodor Lynen Scholarship of the Humboldt Foundation.}
\email{stefan.witzel@math.uni-giessen.de}

\begin{abstract}
We introduce geometric and homological  finiteness properties for countable approximate groups via coarse geometry and then study these finiteness properties for $S$-arithmetic reductive approximate groups. For $S$-arithmetic approximate groups without infinite places we show that the finiteness length is finite and compute this finiteness length explicitly. In the simple case it is one less than the sum of the local ranks. This extends the Rank Theorem of Bux, Köhl and the second author from positive characteristic to characteristic zero. Our proof is based on a geometric version of their proof, but except for some input from reduction theory it is characteristic free. This indicates that the apparent differences between arithmetic groups in characteristic zero and positive characteristic concerning finiteness properties are entirely due to the presence of infinite places.
\end{abstract}

\maketitle

\section{Introdution}

\subsection{A characteristic free Rank Theorem}
Recall that a group $\Gamma$ is said to be \emph{of type $F$} (respectively \emph{of type $F_n$} for some $n \in \mathbb N$) if it admits a compact classifying space (respectively a classifying space with compact $n$-skeleton). It is said to be \emph{of type $FP$} (respectively \emph{of type $FP_n$} for some $n \in \mathbb N$) if the trivial $\Z\Gamma$-module $\Z$ admits a finite resolution (respectively a finite partial resolution of length $n$) by finitely generated projective modules. We say that $\Gamma$ has \emph{finiteness length} (respectively, \emph{homological finiteness length}) $\ell$ if it is of type $F_\ell$, but not of type $F_{\ell+1}$ (respectively type $\FP_\ell$, but not $\FP_{\ell+1}$). If it is of type $F_\ell$ (or $\FP_\ell$) for all $\ell$ we say that the (homological) finiteness length is infinite.

Examples of groups with finite finiteness length arise naturally as $S$-arithmetic subgroups of isotropic reductive groups over global fields in positive characteristic. The precise finiteness lengths of these groups had long been predicted conjecturally and were ultimately established by Bux, Köhl and the second author in their Rank Theorem \cite{bux13} extending previous work notably by Abels, Abramenko~\cite{AbelsAbramenko,Abramenko}, Bux--Wortmann~\cite{BuxWortman07,BuxWortman11} (see also \cite{Gandini12}), and the second author \cite{Witzel14}.

At first sight it seems that the Rank Theorem does not have a counterpart in characteristic zero. Namely, it is a classical result by Borel and Serre \cite[Théorème~6.2]{BorelSerre76} (see also \cite[Section~VII.2]{Brown89}) that an $S$-arithmetic subgroup of a reductive group over a global field of characteristic zero (a number field) is virtually of type $F$.

On closer inspection, however, it becomes clear that the distinction between finite or infinite finiteness length is about whether the natural space on which $\Gamma$ acts is a product of Euclidean buildings or whether it has a factor that is a symmetric space. In other words, whether or not the set $S$ contains an infinite (i.e.\ Archimedean) place. Now a global field of positive characteristic cannot have an infinite place while the definition of an $S$-arithmetic group in characteristic zero requires $S$ to contain all of the infinite places, of which there is at least one.

It is the goal of the present article to explain that the Rank Theorem generalizes to characteristic zero provided one is willing to leave the class of $S$-arithmetic groups and work in the more general setting of \emph{$S$-arithmetic approximate groups}, which are approximate lattices in the sense of \cite{BH}. In this wider context, there exist $S$-arithmetic approximate groups of characteristic zero with and without infinite places. Anticipating terminology explained below, we prove the following characteristic free version of the Rank Theorem:

\begin{theorem}[Rank Theorem]\label{MainTheorem}
Let $k$ be a global field of arbitrary characteristic and let $S$ be a finite, non-empty set of finite places of $k$. Let $\bfG$ be a non-commutative, absolutely almost simple $k$-isotropic $k$-group. Then the approximate group $\bfG(\calO_S)$ has finiteness length and homological finiteness length $d-1$, where $d = \sum_{s \in S} \rk_{k_s} \bfG$.
\end{theorem}

Thus, the dichotomy between the results of Borel--Serre and Bux--Köhl--Witzel is really about the presence of infinite places and not about the characteristic of the underlying global field.

As we will explain below, the positive characteristic case of Theorem \ref{MainTheorem} follows immediately from \cite{bux13}, since in this case every $S$-arithmetic approximate group is commensurable to an $S$-arithmetic group; thus our contribution is to prove the characteristic zero part of the theorem. Nevertheless, we prefer to state the above characteristic free formulation, since the geometric parts of the proof are completely independent of the characteristic, following closely the proof in \cite{bux13}. The only input into the proof that currently still depends on the characteristic is adelic reduction theory, for which no characteristic free version seems to be available. In the positive characteristic case \cite{bux13} uses Harder's version of reduction theory \cite{harder69}, in characteristic zero the appropriate form of reduction theory is proved by Godement \cite{godement64}.

As in the positive characteristic case, the Rank Theorem determines the finiteness properties of $S$-arithmetic approximate subgroups in arbitrary reductive groups:

\begin{corollary}\label{cor:reductive}
Let $k$ and $S$ be as above. Let $\bfG$ be a reductive $k$-isotropic $k$-group. Then there exists a finite extension $\ell/k$, non-commutative and absolutely almost simple groups $\bfH_1, \dots, \bfH_k$  and a $k$-isogeny $R_{\ell/k} \bfH \to \mathscr{D}\bfG^0$, where
$\bfH \defeq \bfH_1 \times \dots \times \bfH_k$. Moreover, if $T$ denotes the set of places of $\ell$ above $S$ and $d_i \defeq \sum_{s \in T} \rk_{\ell_s} \bfH_i$ for $1 \le i \le k$, then the approximate group $\bfG(\calO_S)$ has finiteness length and homological finiteness length $d-1$, where
 $d = \infimum_i d_i$.
\end{corollary}

\subsection{Killing primes and $S$-arithmetic approximate groups}

We now explain the setting of Theorem \ref{MainTheorem} and Corollary \ref{cor:reductive} in more detail, starting with the notion of an $S$-arithmetic approximate group. To motivate this notion, we recall the classical construction of \emph{killing a prime} of an $S$-arithmetic group. Consider for example the $S$-arithmetic group
\begin{equation}\label{CheapExample}
\Gamma := \mathrm{SL}_n({\Z[1/p]}) < \mathrm{SL}_n(\R) \times \mathrm{SL}_n(\Q_p),
\end{equation}
which is an irreducible lattice under the diagonal embedding. If we project $\Gamma$ onto the first factor $\mathrm{SL}_n(\R)$, we obtain a dense subgroup. If, on the other hand, we first intersect $\Gamma$ with a compact open subgroup $K$ of $\mathrm{SL}_n(\Q_p)$ and then project the intersection to the first factor, then we obtain a lattice: If we choose $K\defeq \mathrm{SL}_n(\Z_p)$ then the resulting lattice is $\mathrm{SL}_n(\Z)$ and any other choice of $K$ will provide a commensurable subgroup. We say that this lattice arises from $\Gamma$ by killing the prime $p$. 
With this terminology, every $S$-arithmetic group arises from a group over a global field $k$ by killing all the primes of $k$ which are not contained in $S$.

The example \eqref{CheapExample} illustrates an asymmetry: namely, it is not possible to intersect $\Gamma$ with a compact open subgroup of $\SL_n(\R)$ and then project the result to obtain a lattice in $\SL_n(\Q_p)$, simply because $\SL_n(\R)$ is connected and non-compact, so it does not contain compact open subgroups. Thus, in the context of $S$-arithmetic groups, infinite places cannot be killed in the same way as finite ones. 

This asymmetry is removed naturally in the context of $S$-arithmetic approximate groups; in particular, by killing infinite places we will be able to produce $S$-arithmetic approximate groups without infinite places even in characteristic zero. 

Formally, an \emph{approximate subgroup} of a group $G$ is a subset $\Lambda$ which is symmetric, contains the identity and satisfies $\Lambda \cdot \Lambda \subset \Lambda \cdot F$ for some finite subset $F \subset G$ (cf.\ \cite{Tao}). A discrete approximate subgroup of a locally compact group $G$ is called a \emph{uniform approximate lattice} if there exists a compact subset $K \subset G$ such that $G = \Lambda K$ \cite{BH}. For the notion of a non-uniform approximate lattices there are several competing definitions. For the purposes of this article, we follow Hrushovski \cite{Hrushovski} and call a discrete approximate subgroup $\Lambda \subset G$ an \emph{approximate lattice} if there exists a subset $L \subset G$ of finite Haar measure such that $G = \Lambda L$.

Example of approximate lattices arise from lattices $\Gamma$ in products $G \times H$ of locally compact groups as follows. One chooses a compact subset $W \subset H$ of non-empty interior, forms the intersection $\Gamma \cap (G \times W)$ and projects to $G$. The outcome $\Lambda = \Lambda(G, H, \Gamma, W)$ is an approximate lattice in $G$; it is uniform or non-uniform according to whether $\Gamma$ has the corresponding property and up to commensurability does not depend on the choice of $W$. This construction of approximate lattices, also known as \emph{cut-and-project construction}, was first introduced by Meyer \cite{Meyer} in the abelian case and later generalized to non-abelian groups in \cite{BHP1}.

For example, if $\Gamma$ is as in \eqref{CheapExample} and $B$ is some compact identity neighborhood in $\mathrm{SL}_n(\R)$, then we can kill the infinite place by passing to the approximate lattice $\Lambda = \Lambda(\mathrm{SL}_n(\Q_p), \mathrm{SL}_2(\R), \Gamma, B)$ and obtain a non-uniform approximate lattice in $\mathrm{SL}_n(\Q_p)$. 

In general, we refer to an approximate lattice obtained from an $S$-arithmetic group by killing finitely many (finite or infinite) places as an \emph{$S$-arithmetic approximate group}. It was established recently by Hrushovski \cite{Hrushovski} (generalizing previous work of Machado \cite{Machado}) that every approximate lattice in a product of noncompact, adjoint simple groups over local fields is commensurable to a product of lattices and $S$-arithmetic approximate lattices $\Lambda_i$ in subproducts. Moreover, the factors $\Lambda_i$ can be chosen to be irreducible in the sense that they are not commensurable to a product or, equivalently, do not project discretely onto any proper quotient. In particular, if $\Lambda$ is an irreducible approximate lattice in such a product, then it is either $S$-arithmetic or commensurable to a group. If, moreover, $\Lambda$ is non-uniform and the underlying local fields are of characteristic zero, then the group case cannot occur, and hence arithmeticity even holds in rank one. This allows us to give a stronger formulation of the Rank Theorem in characteristic zero:
\begin{corollary}
If $k$ is a number field, then Theorem \ref{MainTheorem} and Corollary \ref{cor:reductive} hold for any irreducible non-uniform approximate lattice.
\end{corollary}

\subsection{Finiteness properties via coarse geometry}
Having clarified the notion of an ($S$-arithmetic) approximate subgroup, we now need to define finiteness properties for such objects, as the definitions given at the beginning of the article to not make sense for them. Our starting point is an observation due to Alonso  \cite{alonso94} that being of type $F_n$ as well as being of type $FP_n$ are coarse invariants among countable groups for all $n \in \mathbb N$. More precisely, a countable group $\Gamma$ is of type $F_n$ (respectively, of type $FP_n$) if and only if $\Gamma$ is coarsely $n$-connected (respectively coarsely $n$-acyclic) with respect to some (hence any) proper left-invariant metric.

As pointed out in \cite{BH, CHT}, if $\Lambda$ is a countable approximate subgroup of some group $G$ then all proper left-invariant metrics on the group $\Lambda^\infty$ generated by $\Lambda$ restrict to coarsely equivalent metrics on $\Lambda$. This allows us to make the following definition:

\begin{definition}
  A countable approximate subgroup $\Lambda$ is said to be of \emph{type $F_n$} (respectively \emph{type $FP_n$}) if it is coarsely $n$-connected (respectively coarsely $n$-acyclic) with respect to the restriction of some (hence any) proper left-invariant metric on $\Lambda^\infty$.
\end{definition}

By the aforementioned result of Alonso, this definition is compatible with the classical definition for groups. The finiteness property $F_1$ for countable approximate groups admits many different characterizations (e.g.\ the existence of a proper and cocompact quasi-isometric quasi-action on a connected graph) and corresponds precisely to \emph{geometric finite generation} of $\Lambda$ in the sense of \cite{CHT}; it implies in particular, that every finite index subset of $\Lambda$ generates a finitely-generated group. The higher finiteness properties for countable approximate groups are much less understood, and in particular it is not quite clear in which sense $F_2$ can be seen as a version of finite presentation.

\subsection{Comparison to Bux-Köhl-Witzel}
As mentioned before, our proof of the Rank Theorem (Theorem \ref{MainTheorem}) is largely parallel to the proof in the positive characteristic case \cite{bux13}. However, there are two main differences. 

The first difference was already mentioned above and concerns adelic reduction theory. 
The main theorems of adelic reduction theory in positive characteristic were established by Harder \cite{harder69} and recast in geometric terms in \cite{bux13}. Here we extract the analogues of Harder's theorems for number fields from work of Godement \cite{godement64} and redo the translation to geometry including infinite places and approximate groups. We also use the opportunity to elaborate on the role of rescaled Busemann functions in reduction theory.

We would like to point out that the reduction theory we develop in Sections~\ref{sec:adelic_reduction} and~\ref{sec:s-arithmetic_reduction} differs from the classical treatment as in \cite{Borel63} or \cite{PlatonovRapinchuk} in one regard: the fundamental role is played not by Siegel sets but by their saturations, namely orbits under the unipotent rational points of the parabolic. Rather than considering a single Siegel set whose translates cover (an existence statement) and that meets only finitely many translates (a rough uniqueness statement), we consider saturations of Siegel sets of different levels and show that their translates cover if they are large enough (existence statement) and that translates of small enough saturations meet only when they meet at infinity (uniqueness statement). This is in line with \cite{harder69} and \cite{bux13}. For instance, in the well known example of $\SL_2(\R)$ the saturated Siegel sets are horoballs in the upper halfplane $\H^2$ centered at $\infty$, the translates under $\SL_2(\Z)$ of a large enough horoball cover $\H^2$ while the translates of a small enough horoball are disjoint.

The second difference to \cite{bux13} is more subtle. Since we deal with approximate groups rather than groups, all arguments involving (cocompact) group actions need to be replaced by geometric arguments throughout. At several occasion this actually renders the picture clearer when coarse equivalences can seamlessly be used instead of a group theoretic argument that takes a short justification.

\subsection{A general descent principle}
Both the cut-and-project construction of Meyer in the theory of approximate groups and the passage form adelic to $S$-adic reduction theory in the theory of $S$-arithmetic group are based on very similar descent principles, but these are usually formulated slightly differently in the two fields. The following simple, but very general descent principle (Proposition~\ref{prop:descent} in the text) simultaneously encompasses standard arguments from approximate group theory and from reduction theory and may be of independent interest :
 \begin{proposition}\label{prop:descentintro}
Let $G$ and $H$ be topological groups and $\Gamma < G \times H$ a discrete subgroup that maps injectively to $G$ and densely to $H$. Let $W \subseteq H$ be a compact symmetric identity neighborhood and let
\[
\Lambda \defeq \pi_G((G \times W) \cap \Gamma)\text{.}
\]
Then for all compact subsets $I, K \subseteq H$ there exist finite sets $E,F \subseteq G$ such that for all subsets $\Pi \subseteq G$,  
\[
\Lambda F \Pi \cap E \Lambda \Pi \supseteq \pi_G(\Gamma (\Pi\times K) \cap (G \times I))\text{.}
\]
\end{proposition}

\subsection{Future directions and open problems}

Theorem \ref{MainTheorem} conclusively describes the finiteness properties of $S$-arithmetic approximate subgroups of reductive groups in the case where $S$ contains no infinite places. On the other hand, the aforementioned theorem of Borel and Serre covers the case where $S$ contains all infinite places. It applies even in the approximate case, since we will see below that if $S$ contains all infinite places, then every $S$-arithmetic approximate group is commensurable to an $S$-arithmetic group.

Concerning the determination of finiteness properties of $S$-arithmetic approximate subgroups of reductive groups, the remaining open question is therefore what happens if $S$ contains some but not all infinite places. We expect that such approximate groups will be of type $F_\infty$, putting them on the Borel--Serre side of the divide. In other words, \emph{any} symmetric space factor should suffice to guarantee good finiteness properties. 

Beyond the reductive case results on finiteness properties of $S$-arithmetic (approximate) groups are restricted to the solvable case and are fragmentary \cite{Abels87,Bux04,Schesler,Witzel13}.

Besides finiteness properties there remain several important open problems. In the arithmetic case (i.e.\ if $S$ consists of all infinite places) Leuzinger--Young \cite{LeuzingerYoung21} have established that the homological filling function changes from polynomial to exponential in the critical degree $d$. Thus while there is no \emph{qualitative} change (i.e.\ affecting finiteness properties) visible in this case, there is a \emph{quantitative} change (i.e.\ affecting the filling functions). One might expect that this quantitative change of behavior also appears in the general case of $S$-arithmetic approximate groups as soon as $S$ contains a single Archimedean place. However, even if $S$ is assumed to contain \emph{all} Archimedean places (and hence may be assumed to be a group) no such result is currently known --- not even in the presence of a single finite place. A related, but potentially simpler, question, is whether approximate lattices in products of symmetric spaces and Euclidean buildings are always undistorted. This is currently only known for $S$-arithmetic groups due to Lubotzky, Mozes and Raghunathan \cite{LMR}. The question of undistortedness is open for approximate lattices, where it was asked by Machado, as well as for non-uniform lattices on exotic buildings, though no examples of the latter are currently known. The proof in \cite{LMR} makes substantial use of dynamical methods but a purely geometric proof (also applying to the approximate setting) may be possible.

\subsection{Organization of the article}
The article is organized as follows. In Section~\ref{sec:fin_props} we introduce approximate groups and discuss their canonical coarse structure and finiteness properties. In particular, we prove Alonso's criterion (which he formulated in the quasi-isometric category) for coarse spaces and use it to define property $F_d$ for approximate groups. We also prove Proposition~\ref{prop:descentintro}
as Proposition~\ref{prop:descent}. In Section~\ref{sec:sarith_approx} we explain that $S$-arithmetic groups arise from a cut-and-project procedure and introduce $S$-arithmetic approximate groups. Section~\ref{sec:geometry_reductive_local} is concerned with the geometry pertaining to reductive groups over global fields: Bruhat--Tits buildings, symmetric spaces, and the building at infinity. In Section~\ref{sec:busemann} we show how Busemann functions on the spaces from the previous section arise from what we call an invariant horofunction, which plays a central role in reduction theory. Sections~\ref{sec:adelic_reduction} and~\ref{sec:s-arithmetic_reduction} are devoted to algebraic reduction theory. In Section~\ref{sec:adelic_reduction} we develop the adelic theory and in Section~\ref{sec:s-arithmetic_reduction} we deduce the $S$-arithmetic version from it. In Section~\ref{sec:geometric_reduction} we translate the algebraic reduction theory from the previous sections to geometry based on the content of Section~\ref{sec:busemann} and in analogy to \cite{bux13}. Theorem \ref{MainTheorem} is proven in Section~\ref{sec:proof_main}. In Section~\ref{sec:reductive} we discuss the assumptions of Theorem \ref{MainTheorem} and prove Corollary~\ref{cor:reductive}.

\subsection*{Acknowledgments} We want to thank Kai-Uwe Bux, Enrico Leuzinger, Bertrand Rémy, and Guy Rousseau for helpful discussions related to this article.

\setcounter{tocdepth}{1} 
\tableofcontents

\section{Finiteness properties of approximate groups}\label{sec:fin_props}

\subsection{Coarse connectedness properties}

The first goal of this section is to define coarse versions of higher connectedness properties of metric spaces. Throughout we will use the following terminology:

Let $X$ and $Y$ be metric spaces. We say that two maps $\varphi,\psi \colon X \to Y$ have \emph{bounded distance $C \ge 0$} if $d(\varphi(x),\psi(x)) \le C$ for all $x \in X$. We say that a map $\varphi \colon X \to Y$ is a  \emph{coarse Lipschitz map} if there exists a map  $\rho \colon \R_{\ge 0} \to \R_{\ge 0}$ with $\lim_{t \to \infty} \rho(t) = \infty$ such that
\[
d(\varphi(x),\varphi(x')) \le \rho(d(x,x')) \text{ for all }x,x' \in X.
\]
We then call $\rho$ an \emph{upper control} for $\varphi$ and say that $\varphi$ is \emph{$\rho$-Lipschitz}.

Being of bounded distance defines an equivalence class on coarse Lipschitz maps between metric spaces and the corresponding equivalence classes can be composed (by composing representatives). Metric spaces and equivalence classes of coarse Lipschitz maps form a category, called the \emph{category of coarse metric spaces}. Isomorphisms in this category are called \emph{coarse equivalences} and retracts in this category are called \emph{coarse retracts}.

More quantitatively, if there exist an upper control $\rho$ and some $C>0$, then $X$ is a \emph{$(\rho,C)$-retract} of $Y$ if there exist
$\rho$-Lipschitz maps $\iota \colon X \to Y$ and $\pi \colon Y \to X$ such that $\pi \circ \iota$ has distance $C$ from the identity. If in addition $\iota \circ \pi$ also has distance $C$ from the identity, then we say that $X$ and $Y$ are \emph{$(\rho,C)$-equivalent.} 

Given a metric space $(X, d)$ we denote by $[(X,d)]_c$ its coarse equivalence class. A property of metric spaces is called a \emph{coarse property} if it depends only on this coarse equivalence class. We now discuss coarse properties of metric spaces which can be seen as coarse versions of being $n$-connected. With every metric space we associate a directed system of topological spaces:
\begin{definition}
Let $X$ be a metric space. For $r \ge 0$, the \emph{Vietoris--Rips complex} $\VR_r X$ is the simplicial complex whose vertex set is $X$ and such that $x_0, \dots, x_n \in X$ are the vertices of an $n$-simplex if and only if $d(x_i,x_j) \le r$ for all $i,j \in \{0, \dots, n\}$. The \emph{Vietoris--Rips filtration} is the directed system $(\VR_r X)_{r \ge 0}$ with maps given by the inclusions $\VR_r X \into \VR_s X$ for $r \le s$.
\end{definition}
We now define two families of coarse connectedness properties for such directed families of spaces:
\begin{definition} Let $\mathcal A$ be a directed set and let $(S_\alpha)_{\alpha \in \mathcal A}$ (respectively $(G_\alpha)_{\alpha \in \mathcal A}$) be a directed system of topological spaces (respectively groups or pointed sets).
\begin{enumerate}[(i)]
\item $(G_\alpha)_{\alpha \in \mathcal A}$ is \emph{essentially trivial} if for every $\alpha \in \mathcal A$ there exists a $\beta \ge \alpha$ such that the map $G_\alpha \to G_\beta$ is trivial.
\item $(S_\alpha)_{\alpha \in \mathcal A}$ is \emph{essentially $n$-connected} if $(\pi_i(S_\alpha))_{\alpha \in \mathcal A}$ is essentially trivial for all $i \le n$.
\item  $(S_\alpha)_{\alpha \in \mathcal A}$ is \emph{essentially $n$-acyclic} if $(\tilde{H}_i(S_\alpha))_{\alpha \in \mathcal A}$  is essentially trivial for all $i \le n$.
\end{enumerate}
\end{definition}
We will usually apply this to the case where $(S_\alpha)_{\alpha \in \mathcal A}$ is a filtration of a space $S$ by subsets, the maps being given by inclusions. In this context, given a space $X$, we say that two filtrations $(Y_i)_i$ and $(Z_j)_j$ of $X$ are \emph{equivalent} if every term in one of them is contained in some term of the other, i.e.\ $\forall i\, \exists j\, Y_i \subseteq Z_j$, $\forall j\, \exists i\, Z_j \subseteq Y_i$. Equivalence of filtrations of groups is defined similarly.
\begin{lemma}\label{lem:equivalent_filtrations}
  If $(H_i)_i$ and $(K_j)_j$ are equivalent filtrations of a group $G$ then one is essentially trivial if and only if the other is. In particular, if $(Y_i)_i$ and $(Z_j)_j$ are equivalent filtrations of a topological space $X$ then one is essentially $k$-connected ($k$-acyclic) if and only if the other one is.
\end{lemma}
\begin{proof}
  We prove the statement about groups, the other statements being an immediate consequence. Assume that $(K_j)_j$ is essentially trivial.
  Let $i$ be arbitrary. By equivalence there exists a $j \ge i$ such that $H_i < K_j$. By essential triviality there exists $k > j$ such that $K_j \to K_{k}$ is trivial. By equivalence there exists an $\ell \ge k$ such that $K_k < H_\ell$. Then $H_i \to H_\ell$ is trivial factoring through the trivial map $K_j \to K_k$. This shows that $(H_i)_i$ is essentially trivial. The converse follows by symmetry.
\end{proof}
We now specialize to the case of the Vietoris-Rips filtration:
\begin{definition}
Let $X$ be a metric space. We say that $X$ is \emph{coarsely $n$-connected} (respectively \emph{coarsely $n$-acyclic}) if $(\VR_r X)_{r \ge 0}$ is essentially $n$-connected (respectively essentially $n$-acyclic).
\end{definition}

\begin{remark}\label{rem:hurewicz}
We have the following relations between coarse connectedness and coarse acyclicity:
\begin{enumerate}
\item Essential $n$-connectedness implies essential $n$-acyclicity, and hence coarse $n$-connectedness implies coarse $n$-acyclicity.
\item Hurewicz's theorem implies that essential $1$-connectedness together with essential $n$-acyclicity implies essential $n$-connectedness, and thus coarse $1$-connectedness together with coarse $n$-acyclicity implies coarse $n$-connectedness.
\end{enumerate}
\end{remark}
Our next goal is to show that coarse $n$-connectedness and coarse $n$-acyclicity are coarse properties. The proof is essentially the same as the one given in Alonso \cite{alonso94} for quasi-isometry invariance and is based on the following lemma:
\begin{lemma}\label{lem:homotopy}
Let $\varphi,\psi \colon X \to Y$ be $\rho$-Lipschitz maps that have distance $C$. Then for $q \ge \rho(p) + C$ the induced maps $\varphi_*,\psi_*\colon \VR_p(X) \to \VR_q(Y)$ are homotopic.
\end{lemma}

\begin{proof}
We want to show that if $\{x_0,\ldots,x_n\}$ is a simplex in $\VR_p(X)$ then its images under $\varphi_*$ and $\psi_*$ are contained in a common simplex. But in fact since $d(x_i,x_j) < p$ by assumption,
\[
d(\varphi(x_i),\psi(x_j)) \le d(\varphi(x_i),\varphi(x_j)) + d(\varphi(x_j),\psi(x_j)) < C + \rho(p) \le q\text{.}
\]
Hence $\{\varphi(x_0),\ldots,\varphi(x_n),\psi(x_0),\ldots,\psi(x_n)\}$ is a simplex in $\VR_q(Y)$.
\end{proof}

\begin{proposition}\label{prop:retracts}
If $X$ is a coarse retract of  $Y$ and $Y$ is coarsely $n$-connected (respectively $n$-acyclic) then so is $X$.
\end{proposition}

\begin{proof}
Assume that $X$ is a $(\rho,C)$-retract of $Y$ via maps $\iota \colon X \to Y$ and $\pi \colon Y \to X$. Let $K$ be either of $\tilde{H}_i$ or $\pi_i$ for $i \le n$. Let $p$ be arbitrary and let $q$ be such that $K(\VR_{\rho(r)}(Y) \to \VR_q(Y))$ is trivial. Let $r = \rho(q) + C$. Consider the following diagram:
\begin{center}
\begin{tikzpicture}[xscale=4,yscale=2]
\node (xp) at (0,0) {$K(\VR_p(X))$};
\node (yp) at (1,0) {$K(\VR_{\rho(p)}(Y))$};
\node (yq) at (1,-1) {$K(\VR_{q}(Y))$};
\node (xr) at (0,-1) {$K(\VR_{r}(X))$};
\draw (xp) edge[->] node[above] {$\iota_*$} (yp) (yp) edge[->] node[right] {$\id_*$} (yq) (yq) edge[->] node[below] {$\pi_*$} (xr) (xp) edge[->,bend right] node[left] {$\id_*$} (xr) (xp) edge[->,bend left] node[right] {$(\pi \circ \iota)_*$} (xr);
\end{tikzpicture}
\end{center}
The right squares commutes by construction and the left bigon commutes by Lemma~\ref{lem:homotopy}, hence the diagram commutes. The map on the right is trivial by assumption, hence so is the map on the left.
\end{proof}
\begin{corollary}\label{cor:coarse_connected_invariant}
Coarse $n$-connectedness and coarse $n$-acyclicity are coarse properties for all $n \in \mathbb N$.\qed
\end{corollary}

In the sequel, we say that a coarse equivalence class $[(X, d)]_c$ is coarsely $n$-connected if some (hence any) representative $(X,d)$ has this property. For future reference we record the following:

\begin{lemma}\label{lem:coarse_product}
Let $X_1,\ldots,X_k$ be metric spaces. The product $X = X_1 \times \ldots \times X_k$ is coarsely $n$-connected if and only if each $X_i$ is.
\end{lemma}

Note that up to coarse equivalence it is irrelevant which (reasonable) metric we choose on the product. We take $d_X((x_1,\ldots,x_k),(y_1,\ldots,y_k)) = \max_i d(x_i,y_i)$.

\begin{proof}
By our choice of metric, for every $r$ the product $\VR_r X_1 \times \ldots \times \VR_r X_k$ is a deformation retract of $\VR_r X$: taking the identity map on vertices extends to affine maps in both directions. Moreover, for $r \le s$ the diagram

\begin{center}
\begin{tikzpicture}[xscale=4,yscale=2]
\node (rp) at (0,0) {$\VR_r(X_1) \times \ldots \times \VR_r(X_k)$};
\node (rx) at (1,0) {$\VR_r(X)$};
\node (sp) at (0,-1) {$\VR_s(X_1) \times \ldots \times \VR_s(X_k)$};
\node (sx) at (1,-1) {$\VR_s(X)$};
\draw (rx) edge[->>,bend right] (rp)
(rp) edge[right hook->,bend right] (rx)
(sx) edge[->>,bend right] (sp)
(sp) edge[right hook->,bend right] (sx)
(rp) edge[->](sp)
(rx) edge[->] (sx);
\end{tikzpicture}
\end{center}
commutes. The claim now follows from \cite[Theorem~4.2]{hatcher}.
\end{proof}
\subsection{Finiteness properties of groups and their subsets}
Throughout this subsection, $\Gamma$ denotes a countable group. We pick a left-invariant proper pseudo-metric $d$ on $\Gamma$ and denote by $[\Gamma]_c := [(\Gamma, d)]_c$ the associated coarse equivalence class, which is independent of $d$ \cite[Theorem~2.B.4, Corollary~4.A.6]{CornulierDeLaHarpe16}. Explicit representatives of $[\Gamma]_c$ can be constructed as follows: 

Let $S_\ell, \ell \in \N_{>0}$ be an ascending sequence of finite symmetric subsets of $\Gamma$ such that $\bigcup_\ell S_\ell$ generates $\Gamma$. For instance, if $\Gamma = \{g_1,g_2,\ldots\}$ one may take $S_\ell = \{g_1^{\pm 1},\ldots,g_\ell^{\pm 1}\}$. Put $T_\ell = S_\ell \setminus S_{\ell-1}$ with the understanding that $S_o = \emptyset$. The \emph{weighted Cayley graph} $\Cay(\Gamma,(S_\ell)_\ell)$ is the simplicial metric graph with vertex set $\Gamma$ where $g$ and $h$ are connected by an edge of length $\ell$ if $g^{-1}h \in T_\ell$. If $\Gamma = \gen{S}$ is finitely generated with finite generating set $S$, one may take $S_\ell = S$ for all $\ell \ge 1$ and recovers the usual Cayley graph.
The weighted Cayley graph induces a path metric on $\Gamma$ which we denote by $d_{\Gamma,(S_i)_i}$ or just $d_{(S_i)_i}$ if the group is clear from context. This metric is obviously proper and left-invariant, hence $(\Gamma,  d_{(S_i)_i})$, or equivalently $\Cay(\Gamma,(S_\ell)_\ell)$, represents $[\Gamma]_c$.

Recall from the introduction the definition of the finiteness properties $F_n$ and $\FP_n$ for $\Gamma$. The following observation was made in \cite{alonso94} (for $n \geq2$).
\begin{proposition}[Alonso]\label{FnVsFnClassic}
A countable group $\Gamma$ is of type $F_n$ (respectively type $FP_n$) if and only if $[\Gamma]_c$ is coarsely $n$-connected (respectively coarsely $n$-acyclic).
\end{proposition}
Proposition \ref{FnVsFnClassic} is an immediate consequence of the following special case of Brown's criterion \cite{brown87}.
\begin{theorem}[Brown's criterion]\label{thm:browns_criterion}
Let $\Gamma$ be a countable group acting properly on a contractible $CW$-complex $S$. Let $(S_\alpha)_{\alpha \in D}$ be a filtration of $S$ by $\Gamma$-invariant, cocompact subcomplexes. In this situation $\Gamma$ is of type $F_n$ (respectively $\FP_n$) 
if and only if $(S_\alpha)_{\alpha}$ is essentially $(n-1)$-connected (respectively $(n-1)$-acyclic).\qed
\end{theorem}
\begin{proof}[Proof of Proposition~\ref{FnVsFnClassic}]
We equip $\Gamma$ with the left-invariant metric of a weighted Cayley graph coming from an ascending sequence of finite subsets $(S_i)_i$.
The complexes $\VR_r(\Gamma,d_{(S_i)_i})$ are locally finite because the $S_i$ are finite. Since $\Gamma$ acts transitively on vertices, it acts cocompactly on $\VR_r(\Gamma,d_{(S_i)_i})$. Since $G$ acts freely on vertices, it acts properly on $\VR_r(\Gamma,d_{(S_i)_i})$. The union $\bigcup_{r \in \N} \VR_r(\Gamma,d_{(S_i)_i})$ is the complete simplex with vertex set $G$ which is contractible. We can therefore apply Brown's criterion which proves the claim.
\end{proof}
Now let $A$ be a subset of a countable group $\Gamma$. If $d_1, d_2$ are proper left-invariant metrics on $\Gamma$ then the identity $(\Gamma, d_1) \to (\Gamma, d_2)$ is a coarse equivalence and hence restricts to a coarse equivalence $(A, d_1) \to (A, d_2)$. We may thus define a coarse equivalence class $[A]_c := [(A, d_1|_{A \times A})]_c = [(A, d_2|_{A \times A})]_c$. Note that if $\Lambda < \Gamma$ is a subgroup containing $A$ and $d$ is a proper left-invariant metric on $\Lambda$, then $(A, d|_{A \times A})$ represents $[A]_c$; we may thus assume that $\Gamma$ is generated by $A$. Now Proposition \ref{FnVsFnClassic} motivates the following definition:
\begin{definition}\label{DefFn} A subset $A$ of a countable group $\Gamma$ is said to be of type $F_n$ (respectively type $FP_n$) if $[A]_c$ is coarsely $n$-connected (respectively coarsely $n$-acyclic).
\end{definition}
If $A$ is a subgroup of $\Gamma$, we recover the usual finiteness properties. In the context of Definition~\ref{DefFn} Remark~\ref{rem:hurewicz} states that
\[
F_n \implies \FP_n\quad\text{and}\quad F_2 \wedge \FP_n \implies F_n\text{.}
\]

\begin{remark}[Beyond countable groups]
The assumption that $\Gamma$ be countable can be weakened substantially - all we need here is a canonical coarse class of pseudo-metrics on $\Gamma$. For example, using the class of left-invariant proper continuous metrics \cite[Milestone~4.A.8]{CornulierDeLaHarpe16} one can define finiteness properties of subsets of locally compact second countable groups in the same way. If one is willing to work with coarse equivalence classes of coarse spaces in the sense of Roe \cite{Roe} rather than metric spaces, then the countability assumptions can be dropped; in fact, a canonical coarse structure can be defined not only on all locally compact groups, but even on all locally bounded groups in the sense of \cite{Rosendal}. For such groups and their subsets, finiteness properties can be defined as above.
\end{remark}

\subsection{Approximate subgroups}
We will be interested in finiteness properties of a special class of subsets of groups called approximate subgroups. We recall the definition and fix some notation:
\begin{definition}
Let $\Gamma$ be a group. A subset $\Lambda \subseteq \Gamma$ is called an \emph{approximate subgroup} if 
\begin{enumerate}[({AG}1)]
\item $\Lambda$ is symmetric and unital, i.e.\ $\Lambda = \Lambda^{-1}$ and $e \in \Lambda$;
\item there exists a finite subset $F_\Lambda \subseteq  G$ such that $\Lambda^2 \subseteq \Lambda F_\Lambda$. 
\end{enumerate}
If $\Lambda$ generates $\Gamma$ then the pair $(\Lambda, \Gamma)$ is called an \emph{approximate group}. 
\end{definition}
If $\Lambda$ is an approximate subgroup of $\Gamma$ then we denote by $\Lambda^\infty$ the subgroup of $\Gamma$ the group generated by $\Gamma$. By definition, $(\Lambda, \Lambda^\infty)$ is then an approximate group. A \emph{global morphism} between approximate groups $(\Lambda, \Lambda^\infty)$ and $(\Delta, \Delta^\infty)$ is a group homomorphism $\Lambda^\infty \to \Delta^\infty$ which maps $\Lambda \to \Delta$ and there is an obvious notion of a (global) isomorphism.

As before, we can associate with every countable approximate group $(\Lambda, \Lambda^\infty)$ a coarse equivalence class $[\Lambda]_c$.  We say that two approximate subgroups $\Lambda_1$ and $\Lambda_2$ are \emph{coarsely equivalent} if $[\Lambda_1]_c = [\Lambda_2]_c$. This implies in particular that they have the same finiteness properties.

By definition, isomorphic approximate subgroups are coarsely equivalent. Moreover, let us call two subsets $\Lambda, \Lambda'$ of a group $G$ \emph{commensurable} if there exist finite subsets $F_1, F_2 \subseteq G$ such that $\Lambda \subseteq \Lambda' F_1$ and $\Lambda' \subseteq \Lambda F_2$. We observe that two commensurable approximate subgroups of a group are coarsely equivalent. 

\subsection{Model sets and descent}\label{SecModelSets}
Examples of approximate subgroups arise from cut-and-project constructions as follows. We say that $(G, H, \Gamma)$ is a \emph{generalized cut-and-project scheme} if $G$ and $H$ are topological groups and $\Gamma < G \times H$ is a discrete subgroup, which is mapped injectively under the projection $\pi_G \colon G \times H \to G$ and to a dense subgroup under the projection $\pi_H\colon G \times H \to H$. If moreover $G$ and $H$ are locally compact second countable  groups (\emph{lcsc} from now on) and $\Gamma$ is a (possibly non-uniform) lattice in $G\times H$ then we refer to $(G, H, \Gamma)$ as a \emph{cut-and-project scheme}. This is in accordance with the terminology from \cite{BH, BHP1}.

Given a generalized cut-and-project scheme $(G, H, \Gamma)$ we set $\Gamma_G \defeq \pi_G(\Gamma)$, $\Gamma_H \defeq \pi_H(\Gamma)$ and define the associated \emph{$*$-map} $\tau = \pi_H \circ \pi_G|_{\Gamma}^{-1}\colon \Gamma_G \to H$. 

\begin{definition} Let $(G, H, \Gamma)$ be a generalized cut-and-project scheme. A compact symmetric identity neighborhood $W\subseteq H$ will be called a \emph{window}. Given a window $W$, the associated \emph{generalized model set} is defined as 
\[
\Lambda(G, H, \Gamma, W) \defeq \tau^{-1}(W) = \pi_G((G\times W) \cap \Gamma).
\]
If $(G, H, \Gamma)$ is actually a cut-and-project scheme then $\Lambda(G, H, \Gamma, W) $ is called a \emph{model set}.
\end{definition}
\begin{lemma}\label{GenModelSetAppGrp} Every generalized model set $\Lambda = \Lambda(G, H, \Gamma, W)$ is an approximate subgroup of the ambient topological group $G$.
\end{lemma}

For actual model sets this statement is \cite[Proposition~{2.13}(ii)]{BH} but the proof carries over verbatim to the general case. In the actual model set case one can say much more: In the terminology of \cite{BH}, the model set $\Lambda(G, H, \Gamma, W)$ is a strong approximate lattice in $G$, and even a strong uniform approximate lattice if $\Gamma$ is a uniform lattice, cf. \cite[Proposition~2.13]{BH} and \cite[Corollary~3.5]{BHP1}.

Note that if $\Lambda = \Lambda(G, H, \Gamma, W)$ is a generalized model set and $\widetilde{\Lambda} \defeq (G \times W) \cap \Gamma \subseteq G \times H$ then $\pi_G$ induces an isomorphism $(\widetilde{\Lambda}, \widetilde{\Lambda}^\infty) \to (\Lambda, \Lambda^\infty)$ of approximate groups and $\widetilde{\Lambda}^\infty < \Gamma$ is a discrete subgroup of $G \times H$.

 The following proposition provides a very general descent principle for model sets. It generalizes at the same time the argument from \cite[Proposition 2.13(iii)]{BH} concerning relative density of uniform model sets and the arguments used in \cite[§12]{godement64} to reduce $S$-adelic reduction theory to adelic reduction theory. We will use this principle below to reduce  $S$-adelic approximate reduction theory to adelic reduction theory.
 
 \begin{proposition}[General descent principle]\label{prop:descent}
Let $(G, H, \Gamma)$ be a generalized cut-and-project scheme, let $W \subseteq H$ a window and let $\Lambda = \Lambda(G, H, \Gamma, W)$. Then for all compact subsets $I, K \subseteq H$ there exist finite sets $E,F \subseteq G$ such that for all subsets $\Pi \subseteq G$,  
\[
\Lambda F \Pi, E \Lambda \Pi \supseteq \pi_G(\Gamma (\Pi\times K) \cap (G \times I))\text{.}
\]
\end{proposition}

\begin{proof}
Since $\Gamma_H$ is dense in $H$ and $W$ has non-empty interior, the set $W\Gamma_H$ covers all of $H$ and in particular the compact set $I K^{-1}$. By compactness there is a finite set $F' \subseteq \Gamma_H$ such that $I K^{-1}$ is already contained in $WF'$. We can write $F'$ as $F' = \tau(F)$ for some finite subset $F \subseteq \Gamma_G$.

We claim that $F$ is as needed. To see this let $g \in \pi_G(\Gamma (\Pi \times K) \cap (G \times I))$ be arbitrary. This means that there is an $h \in I$ such that $(g,h) \in \Gamma (\Pi \times K)$. Write $(g,h) = (\gamma,\tau(\gamma))(\pi,k)$ according to that decomposition.

Then
\[
(g\pi^{-1},hk^{-1}) = (\gamma,\tau(\gamma)) \in (g\Pi^{-1} \times hK^{-1}) \cap \Gamma \subseteq (g\Pi^{-1} \times IK^{-1}) \cap \Gamma\text{.}
\]
In particular, the right hand side is non-empty.

By choice of $F$ we find that
\begin{align*}
\pi_G(G \times IK^{-1} \cap \Gamma) &\subseteq \pi_G(G \times (W\tau(F)) \cap \Gamma)\\
& = \pi_G(G \times W \cap \Gamma)F\\
&= \Lambda F\text{.}
\end{align*}

This shows that $g\Pi^{-1} \cap \Lambda F \ne \emptyset$ meaning that $g \in \Lambda F \Pi$.

The proof for $E$ is similar: $\Gamma_H W$ covers $IK^{-1}$ giving rise to a finite $E' \subseteq \Gamma_H$ with $IK^{-1} \subseteq E'W$. Writing $E' = \tau(E)$ gives the required set. Indeed, $\pi_G(G \times IK^{-1} \cap \Gamma) \subseteq E\Lambda$ so that, for $g$ as above, $g\Pi^{-1} \cap E\Lambda \ne \emptyset$ meaning $g \in E\Lambda\Pi$.
\end{proof}

\begin{remark}
Compactness of $W$ is not needed in Proposition~\ref{prop:descent}.
\end{remark}

\begin{remark}\label{rem:descent_detail}
The sets $E$ and $F$ in Proposition~\ref{prop:descent} are in fact contained in $\Gamma_G = \pi_G(\Gamma)$.
\end{remark}

As a first application of Proposition \ref{prop:descent} we show that up to commensurability model sets do not depend on the choice of window.

\begin{corollary}\label{Commens} Let $(G, H, \Gamma)$ and $(G,H,\Gamma')$ be generalized cut-and-project schemes with $\Gamma$ and $\Gamma'$ commensurable and let $W, W' \subseteq H$ be two windows. Then $\Lambda \defeq \Lambda(G, H, \Gamma, W)$ and $\Lambda' \defeq \Lambda(G, H, \Gamma', W')$ are commensurable.
\end{corollary}
\begin{proof}
We may assume without loss of generality that $W \subseteq W'$ and that $\Gamma \subseteq \Gamma'$ so that $\Lambda \subseteq \Lambda'$. For the converse let $T$ be a transversal for $\Gamma \backslash \Gamma'$ and apply Proposition \ref{prop:descent} with $I \defeq W'$, $K\defeq \pi_H(T)$ and $\Pi \defeq \{1_G\}$ to find a finite set $F$ such that $\Lambda F \supset \Lambda'$.
\end{proof}

In the situation of Corollary \ref{Commens} we refer to the commensurability class of $\Lambda$ as the \emph{canonical commensurability class} of the generalized cut-and-project scheme $(G, H, \Gamma)$ and denote the associated coarse class by $[(G, H, \Gamma)]_c$. For later reference we also record the following.

\begin{lemma}\label{lem:coarse_equiv}
Let $(G_1,H_1,\Gamma_1)$ and $(G_2,H_2,\Gamma_2)$ be generalized cut-and-project schemes and let $\varphi = \varphi_G \times \varphi_H \colon G_1 \times H_1 \to G_2 \times H_2$ be a homomorphism respecting the product decomposition. Assume that $\ker \varphi \cap \Gamma_1$ and $[\Gamma_2 : \varphi(\Gamma_1) \cap \Gamma_2]$ are finite. Then, for any choice of windows $W_1$ and $W_2$, the approximate subgroups $\Lambda_1 = \Lambda(G_1,H_1,\Gamma_1,W_1)$ and $\Lambda_2 = \Lambda(G_2,H_2,\Gamma_2,W_2)$ are coarsely equivalent.
\end{lemma}

\begin{proof}
The choice of windows is immaterial by Corollary~\ref{Commens}. Having chosen $W_2$ arbitrarily, we take $W_1 = \varphi_H^{-1}(W_2)$. Replacing $\Gamma_1$ by the finite-index subgroup $\Gamma_1 \cap \varphi^{-1}(\Gamma_2)$ we may assume that $\varphi(\Gamma_1) < \Gamma_2$. We then have
\begin{multline*}
\varphi_G(\Lambda_1) = \varphi_G(\pi_{G_1}(\Gamma_1 \cap (G_1 \times W_1))) = \pi_{G_2}(\varphi(\Gamma_1 \cap (G_1 \times W_1))) =\\
\pi_{G_2}(\varphi_G(\Gamma_1) \cap (G_2 \times W_2)) \subseteq \pi_{G_2}(\varphi_G(\Gamma_2) \cap (G_2 \times W_2)) = \Lambda_2
\end{multline*}
The restriction of the map $\Lambda_1 \to \Lambda_2$ has finite fibers by assumption and relatively dense image by Corollary~\ref{Commens}. Hence it is a coarse equivalence.
\end{proof}

\subsection{Finiteness properties of model sets in $\cato$-groups}

In this subsection, $G$ denotes a locally compact second countable group. We refer to a left-invariant, proper, continuous pseudo-metric $d$ on $G$ as a \emph{left-admissible pseudo-metric}. Every lcsc group admits a left-admissible metric, and if $d, d'$ are left-admissible pseudo-metrics on $G$ then the identity map $\id \colon (G, d) \to (G, d')$ is a coarse equivalence \cite[Theorem~2.B.4, Corollary~4.A.6]{CornulierDeLaHarpe16}. We denote the associated coarse equivalence class by $[G]_c := [(G,d)]_c$. 

If $d$ is a left-admissible pseudo-metric on $G$ and $L<G$ is a closed subgroup, then $d|_{L \times L}$ is a left-admissible metric on $L$. We deduce that if $\Lambda$ is an approximate subgroup of $G$ such that $\Lambda^\infty \eqdef L$ is closed in $G$, then $[\Lambda]_c = [(\Lambda, d|_{\Lambda \times \Lambda})]_c$. However, in many cases of interest, notably for many model sets, this is not the case. Nevertheless we would like to represent $[\Lambda]_c$ using the geometry of $G$. At least for model sets, this is always possible:
\begin{lemma}\label{ModelCoarseMetric} Let $\Lambda = \Lambda(G, H, \Gamma, W)$ be a model set as above and let $d_G$ be a left-admissible pseudo-metric on $G$. Then  $[\Lambda]_c = [(\Lambda, d_G|_{\Lambda \times \Lambda})]_c$.
\end{lemma}
\begin{proof} Fix an auxiliary left-admissible metric $d_H$ on $H$ and let $\widetilde{\Lambda} \defeq (G \times W) \cap \Gamma \subseteq G \times H$. A left-admissible metric on $G\times H$ is given by \[d((g_1,h_1), (g_2, h_2)) \defeq \max\{d_G(g_1, g_2), d_H(h_1, h_2)\}.\] Since $\widetilde{\Lambda}^\infty$ is discrete in $G \times H$, by the previous remark we have $[\widetilde{\Lambda}]_c = [\widetilde{\Lambda}, d|_{\widetilde{\Lambda} \times \widetilde{\Lambda}}]$. Since $\pi_G\colon (\widetilde{\Lambda}, \widetilde{\Lambda}^\infty) \to (\Lambda, \Lambda^\infty)$ is an isomorphism of approximate groups, the coarse equivalence class $[\Lambda]_c$ is represented by the metric $d'$ on $\Lambda$ given by
\[
d'(\lambda_1, \lambda_2) \defeq d(\lambda_1', \lambda_2'), \quad \text{where }\{\lambda_j'\} = \widetilde{\Lambda} \cap \pi_G^{-1}(\lambda_j).
\]
If we write $\lambda_j' = (\lambda_j, w_j)$ with $w_j \in W$ then we have
\[
d'(\lambda_1, \lambda_2) = \max\{d_G(\lambda_1, \lambda_2), d_H(w_1, w_2)\}.
\]
Since $d_H$ is continuous, hence bounded on the compact set $W \times W$, we deduce that $d'$ and $d_G$ restrict to coarsely equivalent metrics on $\Lambda$.
\end{proof}
More can be said if $G$ acts properly and continuously on a $\cato$-space $X$. Given a basepoint $o \in X$ and an approximate subgroup $\Lambda \subset G$ we denote by 
\[
\Lambda.o \defeq \{\lambda.o \mid \lambda \in \Lambda\} \subseteq X
\]
the quasi-orbit of $\Lambda$ in $X$. Given $r >0$ we will denote by $N_r(\Lambda.o)$ the $r$-neighborhood of this quasi-orbit; for different choices of $o$ these filtrations of $X$ are equivalent (cf. \cite{CHT}), hence they have the same essential connectedness properties.
\begin{proposition}[Finiteness criterion for $\cato$ model sets]\label{CheckFn} Let $(X, d)$ be a $\cato$-space with base point $o$ and let $G$ be a lcsc group acting properly and continuously by isometries on $(X,d)$. Then a model set $\Lambda \subseteq G$ is of type $F_{n}$ (respectively type $\FP_n$) if and only if the filtration $(N_r(\Lambda.o))_{r>0}$ is essentially $(n-1)$-connected (respectively essentially $(n-1)$-acyclic).
\end{proposition}

\begin{proof}
First observe that since $G$ acts properly on $G$ the pseudo-metric $d_G$ on $G$ given by
\[
d_G(g,h) \defeq d(g.o, h.o)
\]
is proper. It is also continuous (since the action is continuous) and left-invariant (since the action is isometric), hence left-admissible. It thus follows from Lemma \ref{ModelCoarseMetric} that $[\Lambda]_c$ is represented by $(\Lambda, d_G|_{\Lambda \times \Lambda})$, which is isometric to $(\Lambda.o, d|_{\Lambda.o \times \Lambda.o})$ via the quasi-orbit map $\lambda \mapsto \lambda.o$. In particular, $\Lambda$ is of type $F_n$ if and only if $(\Lambda.o, d|_{\Lambda.o \times \Lambda.o})$ is coarsely $(n-1)$-connected.
We have thus reduced the claim to the following lemma.
\end{proof}

\begin{lemma}\label{lem:metric_filtration}
Let $n \in \mathbb N$, let $(X, d)$ be a $\cato$-space and let $\Xi \subseteq X$ be a subset. Then the filtration $(N_r(\Xi))_{r>0}$ is essentially $(n-1)$-connected if and only if $(\Xi, d|_{\Xi \times \Xi})$ is coarsely $(n-1)$-connected. The same statement is true with ``essentially/coarsely $(n-1)$-connected'' replaced by ``essentially/coarsely $(n-1)$-acyclic''.
\end{lemma}

\begin{proof}
For every $r$ we will consider three spaces: the neighborhood $N_r(\Xi)$, the nerve $\Ner_r^X(\Xi)$ of its cover by the balls $B_r(x), x \in \Xi$, and the nerve $\Ner_r^\Xi(\Xi)$ of the cover of $\Xi$ by $r$-balls. Note that the latter is the Vietoris--Rips complex $\VR_r(\Xi)$ by definition.

It is clear that $r$-balls that meet in $\Xi$ also meet in $X$ hence there is a natural map $\Ner_r^\Xi(\Xi) \to \Ner_r^X(\Xi)$.

In the other direction, note that if $B_r(x_i),i=0,\ldots,n$ meet in a point $x$ then $d(x_i,x_j) < 2r$. This shows that there is a map $\Ner_r^X(\Xi) \to \Ner_{2r}^\Xi(\Xi)$.

Since $X$ is $\cato$, every metric ball in $X$ is convex. Therefore the nerve cover lemma implies that there is a homotopy equivalence $\Ner_r^X(\Xi) \to N_r(\Xi)$.

Now suppose that $(\VR_r(\Xi))_r$ is essentially $(n-1)$-connected. If $r$ is arbitrary, let $s \ge r$ be such that $\VR_{2r}(\Xi) \to \VR_s(\Xi)$ induces trivial maps for $\pi_i$ up to $i \le n-1$. The following commuting diagram

\begin{center}
\begin{tikzpicture}[xscale=3,yscale=-2]
\node (rneigh) at (0,0) {$N_r(\Xi)$};
\node (rnerve) at (1,0) {$\Ner_r^X(\Xi)$};
\node (rrips) at (2.5,0) {$\Ner_{2r}^\Xi(\Xi) = \VR_{2 r}(\Xi)$};
\node (sneigh) at (0,1) {$N_s(\Xi)$};
\node (snerve) at (1,1) {$\Ner_s^X(\Xi)$};
\node (srips) at (2.5,1) {$\Ner_{s}^\Xi(\Xi) =\VR_s(\Xi)$};
\draw (rnerve) edge[->] node[anchor=south] {$\simeq$} (rneigh);
\draw[->]  (rnerve) -- (rrips);
\draw (snerve) edge[->] node[anchor=south] {$\simeq$} (sneigh);
\draw[->] (srips) -- (snerve);
\draw[->] (rneigh) -- (sneigh);
\draw[->]  (rnerve) -- (snerve);
\draw[->]  (rrips) -- (srips);
\end{tikzpicture}
\end{center}

shows that $\id_*(N_r(\Xi) \to N_s(\Xi))$ factors through $\id_*(\VR_{2r}(\Xi) \to \VR_s(\Xi))$ for $i < n$.

Similarly, if $(N_r(\Xi))_r$ is essentially $(n-1)$-connected, let $s \ge r$ be such that $\id_*\colon N_r(\Xi) \to N_s(\Xi)$ is trivial. Then the diagram
\begin{center}
\begin{tikzpicture}[xscale=3,yscale=-2]
\node (rneigh) at (0,0) {$N_r(\Xi)$};
\node (rnerve) at (1,0) {$\Ner_r^X(\Xi)$};
\node (rrips) at (2.5,0) {$\Ner_{r}^\Xi(\Xi) =\VR_{r}(\Xi)$};
\node (sneigh) at (0,1) {$N_s(\Xi)$};
\node (snerve) at (1,1) {$\Ner_s^X(\Xi)$};
\node (srips) at (2.5,1) {$\Ner_{2s}^\Xi(\Xi) =\VR_{2s}(\Xi)$};
\draw (rnerve) edge[->] node[anchor=south] {$\simeq$} (rneigh);
\draw[->]  (rrips) -- (rnerve);
\draw (snerve) edge[->] node[anchor=south] {$\simeq$} (sneigh);
\draw[->] (snerve) -- (srips);
\draw[->] (rneigh) -- (sneigh);
\draw[->]  (rnerve) -- (snerve);
\draw[->]  (rrips) -- (srips);
\end{tikzpicture}
\end{center}
shows that $\VR_r(\Xi) \to \VR_{2s}(\Xi)$ is trivial as well.

The proof for ``$(n-1)$-acyclic'' is the same with $\pi_i$ replaced by $\tilde{H}_i$.
\end{proof}

\section{$S$-arithmetic approximate groups}\label{sec:sarith_approx}

\subsection{Adeles and Ideles}

By a \emph{global field} we shall mean either a \emph{number field}, namely a finite extension of $\Q$, or a \emph{global function field}, namely a finite extension of $\F_p(t)$ for some prime $p$. A \emph{place} of $k$ is an equivalence class of absolute values on $k$. We denote by $V_k$ or simply $V$ the set of all places of $k$.

Let $s \in V$ be a place. Taking the metric completion of $k$ with respect to any absolute value representing $s$ results in the \emph{completion} $k_s$ of $k$ at $s$ which is a locally compact topological field. The place $s$ is then called \emph{infinite} (or \emph{Archimedean}) if the completion $k_s$ is an Archimedean field, i.e.\ $\R$ or $\C$, and \emph{finite} (or \emph{non-Archimedean}) otherwise. A function field has no infinite places, whereas the set of infinite places of a number field is non-empty and finite. If $S$ is a set of places, we denote by $S\fin$ and $S\inf$ the subset of finite and infinite places in $S$ respectively.

If $s$ is a finite place, the ring of integers $\calO_s$ of $k_s$ is compact and open. If $s$ is an infinite place then $k_s$ contains neither a compact subring nor a proper subring with open interior. For a finite set $S$ of places that contains all infinite places, the ring of \emph{$S$-adeles} is defined as
\begin{equation}\label{eq:s-adeles}
\A_S \defeq \prod_{s \in S} k_s \times \prod_{s \not \in S} \calO_s\text{.}
\end{equation}
If $S \subseteq S'$ are finite subsets of $V$, then there is an obvious inclusion $\A_S \to \A_{S'}$ and the ring of adeles of $k$ is defined as the colimit
\begin{equation}\label{eq:adeles}
\A \defeq \lim_S \A_S\text{.}
\end{equation}
The corresponding  \emph{idele group} is defined as $\J \colon= \GL_1(\A)$. More explicitly, an element of $\J$ is a sequence $x = (x_s)_{s \in V}$ with $x_s \in k_s^\times$ and $x_s \in \calO_s^\times$ for all but finitely many $s$. To define a norm on $\J$ one proceeds as follows: Given a place $s \in V$, the \emph{normalized absolute value} $\abs{\cdot}_s$ on $k_s$ is defined as follows. The additive group $k_s$ is a locally compact group and thus admits a Haar measure $\mu^s$. For every $x \in k_s^\times$ the measure $\mu^s_x$ defined by $\mu^s_x(A) = \mu^s(x \cdot A)$ is a Haar measure as well and therefore is a multiple of $\mu^s$. This allows to define the normalized absolute value on $k_s$ via
\[
\mu_x^s = \abs{x}_s \cdot \mu^s\text{.}
\]
The \emph{idele norm} of an element  $x = (x_s)_{s \in V} \in \J$ is then defined as
\[
\abs{x} = \prod_{s \in V} \abs{x_s}_s\text{.}
\]
Note that $k$ (respectively $k^\times$) embeds diagonally into $\A$ (respectively $\J$), and this embedding has discrete image. Indeed, the latter follows from the fact that, with the normalization defined above, for elements of $k$ the local absolute values average out in the following sense, see \cite[Theorem~II.12]{CasselsFroehlich}:
\begin{proposition}[Product formula]\label{prop:product_formula}
$\abs{x} = \prod_{s \in V} \abs{x}_s = 1$ for $x \in k^\times$.
\end{proposition}

\subsection{Adelic groups}
From now on we work in the following general setting; we will add additional assumptions as we move along.
\begin{convention}\label{ConventionAdelicGroups}
\begin{enumerate}
\item $k$ is a global field (of arbitrary characteristic) with set of places $V = V_k$, ring of adeles $\A = \A_k$ and idele group $\J = \J_k$; 
\item $V\fin$ and $V\inf$ denote the finite, respectively infinite places of $k$;
\item $\bfG$ is an algebraic group over $k$ and $\Gamma \defeq \bfG(k)$ denotes its group of $k$-points;
\item $\X[k](\bfG) \defeq \Hom_k(\bfG, \GL_1)$ denotes the abelian group of $k$-characters of $\bfG$, written additively;
\item we fix a faithful representation of $\bfG$ and thereby consider $\bfG$ as a $k$-closed subgroup of $\GL_N$ for some $N \in \mathbb N$.
\end{enumerate}
\end{convention}
As explained in  \cite[Section 5.1]{PlatonovRapinchuk}, one can associate with every $k$-variety $\mathbf X$ a locally compact space $\mathbf{X}(\A)$, the \emph{space of adelic points} of $\mathbf X$, and this construction is functorial  \cite[Prop. 2.1]{Conrad}. This functoriality implies that for $\bfG$ as above the space $\bfG(\A)$ is actually a locally compact topological group, called the \emph{group of adelic points} of $\bfG$, and that every $\chi \in \X[k](\bfG)$ extends to a continuous homomorphism $\chi_\A: \bfG(\A) \to \GL_1(\A) = \J$. We will often abuse notation and denote this homomorphism simply by $\chi$.

It will be important for us that the group $\bfG(\A)$ can be described more explicitly as follows, using the representation $\bfG \hookrightarrow \GL_n$, see \cite[p.249]{PlatonovRapinchuk}. Given a place $s \in V$ we denote by $\bfG(k_s) < \GL_n(k_s)$ the group of $k_s$-points of $\bfG$, and we define $\bfG(\calO_s) \defeq \bfG(k_s) \cap \GL_N(\calO_s)$.  For every finite subset $S \subseteq V$ containing all infinite places we then define the locally compact topological group $\bfG(\A_S)$ as
\[
\bfG(\A_S) \defeq \prod_{s \in S} \bfG(k_s) \times \prod_{s \in V \setminus S} \bfG(\calO_s).
\]
Then for every $S \subset T$ we have a natural inclusion map $\bfG(\A_S) \hookrightarrow \bfG(\A_T)$ and
\[
\bfG(\A) = \lim_S \bfG(\A_S)
\]
is the colimit of the corresponding system in the category of topological groups. In particular, this colimit is independent of the choice of representation $\bfG \hookrightarrow \GL_N$, whereas the groups $\bfG(\A_S)$ may indeed depend on this choice.

The group $\Gamma = \bfG(k)$ embeds diagonally into $\bfG(\A)$, and the image of this embedding is discrete. Note that if $\chi \in \X[k](\bfG)$, then for every $g \in \bfG(k) \subset \bfG(\A)$ we have $\chi_\A(g) \in \GL_1(k) = k^\times$ and hence $|\chi_\A(g)| = 1$ by the product formula (Proposition~\ref{prop:product_formula}). This shows that $\bfG(k)$ is contained in the subgroup
\begin{equation}\label{eq:ring}
\bfG(\A)^0 \defeq \bigcap_{\chi \in \X[k](\bfG)}  \{g \in \bfG(\A) \mid |\chi_\A(g)| = 1\} < \bfG(\A)\text{.}
\end{equation}

\subsection{Adelic cut-and-project schemes} We keep the notation of the previous section. We identify $\Gamma = \bfG(k)$ with its image under the diagonal embedding into $\bfG(\A)$.
In addition we consider a subset $S \subset V$ of places, not necessarily finite. We then have a continuous group homomorphisms 
\[
\pi_S \colon \bfG(\A) \to \prod_{s \in S}\bfG(k_s)\text{,}
\]
and we denote by $\bfG_S \defeq \pi_S(\bfG(\A))$ the image of this homomorphism, equipped with the quotient topology. If $S$ is finite, then  $\bfG_S = \prod_{s \in S} \bfG(k_s)$, but if $S$ is infinite then $\bfG_S$ is a restricted product and its topology is finer than the restriction of the product topology. 

We observe that, for every $S \subset V$, we have a splitting of topological groups 
\[
\bfG(\A) \cong \bfG_S \times \bfG_{V \setminus S}\text{,}
\]
where the projections onto the two factors are given by $\pi_S$ and $\pi_{V \setminus S}$ respectively. In particular, we can consider $\Gamma = \bfG(k)$ as a discrete subgroup of the product $\bfG_S \times \bfG_{V \setminus S}$, and its projection $\pi_S|_\Gamma$ onto the first factor is injective.Thus if we define
\[
\bfG_{V\setminus S}^\sharp \defeq \overline{\pi_{V \setminus S}(\Gamma)} < \bfG_{V\setminus S}, 
\]
then $(\bfG_S, \bfG_{V \setminus S}^\sharp, \bfG(k))$ is a generalized cut-and-project scheme.
\begin{definition}\label{adeliccup} If $S \subset V$ is finite, then $(\bfG_S, \bfG_{V \setminus S}^\sharp, \bfG(k))$ is called the \emph{adelic generalized cut-and-project scheme} with \emph{parameters} $(k, S, \bfG)$. Any approximate subgroup $\Lambda \subseteq \pi_S(\Gamma) \subseteq \bfG_S$ which is contained in the canonical commensurability class of $(\bfG_S, \bfG_{V\setminus S}^\sharp, \bfG(k))$ is called an \emph{$S$-arithmetic approximate subgroup}.
\end{definition}

More explicitly, if $W$ is any window in $\bfG_{V\setminus S}$, then a discrete approximate subgroup $\Lambda \subset \bfG_S$ is an $S$-arithmetic approximate subgroup if and only if it is commensurable to the generalized model set $\Lambda(\bfG_S, \bfG_{V\setminus S}^\sharp, \bfG(k), W \cap \bfG_{V\setminus S}^\sharp)$. In the sequel we will be mostly interested in $S$-arithmetic approximate subgroups of reductive groups. Nonetheless we take a moment to consider the more general situation.

\begin{remark} Let $(G,H, \Gamma)$ be an adelic generalized cut-and-project scheme with parameters $(k, S, \bfG)$. Then two questions suggest themselves:
\begin{enumerate}
\item Is $\Gamma$ a lattice in $G\times H$, or equivalently, is $(G, H, \Gamma)$ a cut-and-project scheme?  (This would imply in particular that $\Lambda(G, H, \Gamma, W)$ is a strong approximate lattice in the sense of \cite{BH}, which has major structural consequences.)
\item Is $H = \bfG_{V \setminus S}$ and hence $G \times H = \bfG(\A)$?
\end{enumerate}
The answer two both questions is negative in general, but it is positive in many cases of interest. By \cite[Thm. 5.5(2)]{PlatonovRapinchuk}, $\Gamma$ is a lattice in $\bfG(\A)$ if and only if $\X[k](\bfG^0) = \{0\}$, and in this case $\Gamma$ is in particular a lattice in $G \times H$. This shows, that the answer to (1) is positive whenever $\bfG$ is connected semisimple. For the rest of this discussion let us assume that $\bfG$ is connected;  then $\Gamma$ is always a lattice in $\bfG(\A)^0 \cap (G \times H)$ by \cite[Thm.~5.6]{PlatonovRapinchuk}, and hence (1) is equivalent to asking whether the subgroup
\[
\bfG(\A)^0 \cap (G \times H) \subset G \times H
\]
has finite covolume. The answer to this question is positive if $\bfG(\A)^0$ is cocompact in $\bfG(\A)$ or if $\bfG$ is reductive and $\bfG(k_s)$ has compact center for all $s \in S$ (since  $G \times H$ is the closure of $(G \times \{e\})\Gamma$ in $\bfG(\A)$), but negative even for many (non-semisimple) reductive groups. Question (2) is essentially the problem of strong approximation, which is well-studied in the literature. While the answer to (2) is negative for general reductive groups, it is positive for all $S$ if $\bfG$ is unipotent \cite[Lemma 5.5]{PlatonovRapinchuk} and we claim that it is also positive if $\bfG$ is connected, simply connected, absolutely almost simple and $k$-isotropic. Indeed, in this case $\bfG$ is absolutely semisimple as well as $k$-almost simple (since it is absolutely almost simple) and the assumption that $\bfG$ be $k$-isotropic implies that it is $k_s$-isotropic for every $s$, and hence that every $\bfG_S$ is non-compact; our claim thus follows from the strong approximation theorem of Prasad \cite[Theorem~A]{prasad77}.
\end{remark}
For our purposes here, the following partial answer to Questions (1) and (2) will be sufficient:
\begin{proposition}\label{PropNiceSimpleG} Let  $\bfG$ be connected, simply connected, absolutely almost simple and $k$-isotropic. Then $(G, H, \Gamma) = (\bfG_S, \bfG_{V \setminus S}, \bfG(k))$ is a cut-and-project scheme with total space $G \times H = \bfG(\A)$ and hence for every window $W \subset  \bfG_{V \setminus S}$ the set $\Lambda(G, H, \Gamma, W)$ is a strong approximate lattice.\qed
\end{proposition}

\subsection{Relation to $S$-arithmetic groups} 
We now compare our definition of an $S$-arithmetic approximate group to the usual definition of an $S$-arithmetic group. Assume that we are in the situation of Convention \ref{ConventionAdelicGroups} and that, in addition, $S \subset V$ is a non-empty set containing $V\inf$. In this situation, one defines the $S$-arithmetic group $\bfG(\calO_S)$ as
\begin{equation}\label{SArGroupCase}
\bfG(\calO_S) \defeq \pi_S(\bfG(k) \cap \bfG(\A_S)) < \bfG_S.
\end{equation}
While this group may depend on the chosen embedding, its commensurability class depends only on the parameters $(k, S, {\bf G})$. Any subgroup of $\bfG(k)$ in this commensurability class is called an \emph{$S$-arithmetic group}. The relation between $S$-arithmetic groups and $S$-arithmetic approximate groups is described by the following proposition.

\begin{proposition}\label{SGroupVsApproxGp} In the situation of Convention \ref{ConventionAdelicGroups} let $S \subseteq V$ be a non-empty set of places and let $S^+ \defeq S \cup V\inf$. Assume that
\begin{equation}\label{eq:compact_subgroups}
\bfG(k_s)\text{ is compact for all } s \in (V \setminus S)\inf\text{.}
\end{equation}
The every $S$-arithmetic approximate group $\Lambda \subset \bfG_S \subset \bfG_{S^+}$ with parameters $(k, S, \bf G)$ is commensurable to $\bfG(\calO_{S^+})$. In particular, this is the case if $\chr(k) > 0$ or if $\chr(k) = 0$ and $S$ contains all infinite places.
\end{proposition}
In view of the proposition, the theory of $S$-arithmetic approximate groups extends the theory of $S$-arithmetic groups, but as far as commensurability invariant properties are concerned, this extension is only interesting if $\chr(k) = 0$ and under the assumption that $S$ does \emph{not} contain all infinite places. In fact, in the present article we will focus on the case where $S$ contains no infinite places at all, leaving the mixed case for future work.
\begin{remark}
Assume that $\chr(k) = 0$ and let $\bfG$ be as in Proposition \ref{PropNiceSimpleG}. From the above discussion it is easily seen that for every finite place $s\in V\fin$ there are $S$-arithmetic approximate subgroups in $\bfG(k_s)$ that are not relatively dense, and these can be chosen to be cut-and-project sets. In the terminology of \cite{BH} and \cite{BHP1} such a cut-and-project set is a non-uniform strong approximate lattice. In particular, $\SL_n(\Q_p)$ admits non-uniform strong approximate lattices for all $n \geq 2$.

This should be compared to the fact that there cannot be a non-uniform lattice in $\SL_n(\Q_p), n \ge 3$. Indeed, by Margulis arithmeticity \cite[Theorem (1), p.2]{Margulis91} a lattice in $\SL_n(\Q_p)$ has to be arithmetic, commensurable to the image of an arithmetic group of the form $\bfH(\calO_S)$ where $\bfH$ is defined over a number field $k$, see \cite[Chapter~IX]{Margulis91}. But then $\bfH$ has to be $k_s$-anisotropic for all infinite places $s$ of $k$. In particular, it is $k$-anisotropic, so the lattice is in fact uniform. A more elementary argument for the non-existence of a finitely generated non-uniform lattice is that it would need to have unbounded torsion while by Selberg's lemma it is virtually torsion-free.
\end{remark}

For the proof of Proposition \ref{SGroupVsApproxGp} we observe that Condition~\eqref{eq:compact_subgroups} is equivalent to $H \defeq \bfG_{V\setminus S}$ admitting a compact identity neighborhood $W$ that is a \emph{subgroup}. Indeed, if $s$ is a finite place then $\bfG(k_s)$ admits a compact \emph{open} subgroup, for example
\[
\bfG(\calO_s) \defeq  \bfG(k_s) \cap \GL_N(\calO_s)\text{,}
\]
but if $s$ is infinite then any compact subgroup of $\bfG(k_s)$ has empty interior unless $\bfG(k_s)$ is compact itself. 
\begin{proof}[Proof of Proposition \ref{SGroupVsApproxGp}] Let $(G, H, \Gamma)$ be the adelic generalized cut-and-project scheme with parameters $(k, S, \bfG)$.
By the previous remark the subset
\[
W \defeq  \prod_{s \in V\fin \setminus S} \bfG(\calO_s) \times \prod_{s \in V\inf \setminus S}\bfG(k_s) \subset H
\]
is a compact open subgroup in $H$. Similarly,
\[
W\fin \defeq   \prod_{s \in V\fin \setminus S} \bfG(\calO_s) \subset \bfG_{V \setminus S^+}
\]
is a compact open subgroup in $\bfG_{V \setminus S^+}$. Unraveling definitions we then see that
\[
\bfG(\calO_{S^+}) = \pi_{S^+}(\bfG(k) \cap \bfG(\A_{S^+}))= \Lambda(\bfG_{S^+}, \bfG_{V\setminus S^+}^\sharp, \bfG(k), W\fin \cap \bfG_{V\setminus S^+}^\sharp).
\]
On the other hand we may assume that $\Lambda = \Lambda(\bfG_{S}, \bfG_{V\setminus S}^\sharp, \bfG(k), W \cap \bfG_{V\setminus S}^\sharp)$ by Corollary \ref{Commens}. Considering $\Lambda$ as a subset of $\bfG_{S^+}$ via the we then have
\[
\Lambda = \bfG(\calO_{S^+}) \cap (\bfG_S \times \prod_{s \in V\inf \setminus S}\bfG(k_s)),
\]
and hence $\Lambda$ is a commensurable subset of $\bfG(\calO_{S^+})$.
\end{proof}

\section{Geometry of reductive groups over local and global fields}\label{sec:geometry_reductive_local}

While $S$-arithmetic approximate subgroups can be defined for any $k$-algebraic group $\bfG$, in this article we will focus on $S$-arithmetic approximate subgroups of reductive groups. Our approach to study these approximate subgroups relies on the fact that reductive groups over local fields act geometrically on \cato{} spaces: Bruhat-Tits buildings in the non-Archimedean case and Riemannian symmetric spaces in the Archimedean case. Both of them have in common that their visual boundaries are spherical buildings and contain a copy of the spherical building of the underlying global field. 

The purpose of this section is to collect the basic facts concerning reductive groups over local fields and their natural  \cato{}-actions which will be needed in the sequel.
For most of the section we will restrict attention to semisimple groups to avoid various technicalities. The generalization to reductive groups will be discussed at the end of this section.

\subsection{Parabolic subgroups and their characters}\label{sec:pars_chars}
The geometry of semi\-simple groups is closely related to the combinatorics of their parabolic subgroups. For this reason we start this chapter by discussing our notations and normalizations related to parabolic subgroups and their characters. A general reference is \cite{BorelTits65}.
Throughout this section  $\bfG$ denotes a semisimple group defined over a field $k$. We denote by $r$ the $k$-rank of $\bfG$.

We fix once and for all an embedding of $\bfG$ as a $k$-closed subgroup of some $\GL_N$. Given a $k$-closed subgroup $\bfH < \bfG$ we denote by $\X[k](\bfH) \defeq \Hom_k(\bfH, \GL_1)$ the abelian group of $k$-characters of $\bfH$, written additively. We also denote by $\coX[k](\bfH) \defeq \Hom_k(\GL_1, \bfH)$ the group of $k$-cocharacters of $\bfH$. If $\bfP$ is a minimal parabolic subgroup and $\bfT < \bfP$ is a maximal $k$-split torus then the inclusion $\bfT \hookrightarrow \bfP$ induces an isomorphism $\X[k](\bfP) \otimes \R \cong \X[k](\bfT) \otimes \R$. 
 
 We now fix a minimal parabolic subgroup $\bfP$ and a maximal $k$-split torus $\bfT < \bfP$. We denote by $\Phi \defeq \Phi(\bfG, \bfT) \subseteq \X[k](\bfT)$ the set of $(\bfG, \bfT)$-roots, by $\Phi^+ \subseteq \Phi$ the positive subsystem corresponding to $\bfP$, and by $\Pi \subseteq \Phi^+$ the corresponding set of simple roots. We fix once and for all an enumeration $\Pi = \{\alpha_1^\bfP, \dots, \alpha_r^\bfP\}$ of $\Pi$ and identify $\Pi$, $\Phi^+$ and $\Phi$ with their respective images in the $\R$-vector space $\X[k](\bfP) \otimes \R$ under the maps
 \[
 \Pi \subseteq \Phi^+ \subseteq \Phi \subseteq \X[k](\bfT) \hookrightarrow \X[k](\bfT) \otimes \R \cong \X[k](\bfP) \otimes \R
 \]
In particular, $\Pi$ is then an $\R$-basis of $\X[k](\bfP) \otimes \R$ and every $\alpha \in \Phi$ (respectively $\Phi^+$) is a $\Z$-linear (respectively non-trivial $\N_0$-linear) combination of $\alpha_1^\bfP, \dots, \alpha_r^\bfP$.

An alternative basis of $\X[k](\bfP) \otimes \R$ can be given as follows: Given $i \in \{1, \dots, r\}$ we define 
        \[\bfP_i \defeq \langle \bfT, \bfU_\alpha, \alpha \in \Phi_i\rangle, \quad \text {where} \quad \Phi_i \defeq \bigg\{\sum_{j=0}^r n_j \alpha^\bfP_j \in \Phi \mid n_j \in \Z, n_i \geq 0\bigg\}
        \]
where $\bfU_\alpha$ is the root group associated to $\alpha$.
Then $\bfP_1, \dots, \bfP_r$ are precisely the maximal parabolic subgroup of $\bfG$ containing $\bfP$ and they are mutually not conjugate. In particular, for every $i \in \{1, \dots, r\}$ the space $\X[k](\bfP_i) \otimes \R$ is one-dimensional and embeds into $\X[k](\bfP)$ via the canonical restriction map. We then have
\[
\X[k](\bfP) \otimes \R = \bigoplus_{i=1}^r \X[k](\bfP_i) \otimes \R\text{.}
\]
Thus there is a basis of $\X[k](\bfP)\otimes \R$ consisting of generators of the $\X[k](\bfP_i), i \in \{1, \dots, r\}$. The precise choice of these generators will not be important, but it is convenient to choose canonical generators as follows.

Following \cite[Section 1.3]{harder69} we define the \emph{canonical character} $\chi^{\bfR} \in \X[k](\bfR)$ of a maximal parabolic $\bfR$ as follows. Take a sequence $R_u(\bfR) = \bfU_0 > \ldots > \bfU_{\ell+1} = 1$ of normal subgroups of the unipotent radical such that $\bfU_j/\bfU_{j+1}$ is a vector space of dimension $d_i$. Then $\bfR$ acts on the one-dimensional exterior power $\bigwedge^{d_j}(\bfU_j/\bfU_{j+1})$ via a character $\psi_{j}$ (the determinant of the action on $\bfU_j/\bfU_{j+1}$). The canonical character is the sum of these, $\chi^\bfR = \sum_{j=0}^{\ell} \psi_{j}$. We also write $\chi_i^\bfP$ for the canonical character $\chi^{\bfP_i}$ associated to $\bfP_i$. Then $(\chi^\bfP_1, \dots, \chi^\bfP_r)$ is a basis of $\X[k](\bfP)\otimes \R$. Explicitly, the Lie algebra $\mathfrak u_i$ of the unipotent radical of $\bfP_i$ decomposes under the adjoint action of $\bfT(k)$ into root spaces $\mathfrak u_{i,\alpha}$, and under the identification  $ \X[k](\bfP)\otimes \R \cong \X[k](\bfT) \otimes \R$ the $i$th canonical character $\chi^\bfP_i$ is given by
\[
\chi^\bfP_i(t) = \sum_\alpha \dim \mathfrak{u}_{i,\alpha}\, \alpha. 
\]

In the sequel we refer to the basis $\{\alpha_1^\bfP, \dots, \alpha^\bfP_r\}$ as the \emph{root basis} of $\X[k](\bfP) \otimes \R$ and to the basis $(\chi^\bfP_1, \dots, \chi^\bfP_r)$ as its \emph{weight basis}. The two bases are related by transformation matrices
\begin{equation}\label{RootsVsCharacters}
\alpha^\bfP_i = \sum_{j=1}^r c_{ij}\chi^\bfP_j \quad \text{and} \quad \chi^\bfP_j = \sum_{i=1}^r n_{ji} \alpha^\bfP_i.
\end{equation}
In the following lemma we omit the superscript $\bfP$ for all roots and weights.

\begin{lemma}\label{lem:root_weight}
The root basis and the weight basis satisfy
\begin{equation}\label{eq:root_weight}
\gen{\alpha_i,\chi_i} > 0 \quad \text{and} \quad \gen{\alpha_i,\chi_j} = 0 \text{ for } i \ne j\text{.}
\end{equation}
where $\gen{\cdot,\cdot}$ is a scalar product invariant under the Weyl group. The coefficients in \eqref{RootsVsCharacters} satisfy
\begin{equation}\label{eq:weight_coefficients}
n_{ii} > 0 \quad \text{and} \quad n_{ji} \ge 0.
\end{equation}
\end{lemma}

\begin{proof}
The coefficients satisfy $n_{ji} \ge 0$ for all $i,j$ because the unipotent radical of $\bfP_i$ is contained in $\bfP$ so the characters $\chi_{i,j}$ in the discussion above are sums of positive roots with positive coefficients.

Let $s_i$ be the element of the Weyl group such that $s_i\alpha_i = -\alpha_i$. Note that $s_i$ is represented by an element of $\bfP_j$ so that $s_i\chi_j=\chi_j$ for all $j \ne i$. Thus
\[
\gen{\alpha_i,\chi_j} = \gen{s_i\alpha_i,s_i\chi_j} = \gen{-\alpha_i,\chi_j} = -\gen{\alpha_i,\chi_j}
\]
showing the equations in \eqref{eq:root_weight}.

Now we have
\[
0 < \gen{\chi_i,\chi_i} = \sum_{j=1}^r n_{ji} \gen{\alpha_j,\chi_i} = n_{ii} \gen{\alpha_i,\chi_i}
\]
and the inequality in \eqref{eq:root_weight} follows, as well as $n_{ii} > 0$.
\end{proof}

So far we have discussed special bases of $\X[k](\bfP) \otimes \R$ for a fixed choice of minimal parabolic $\bfP$. We now extend our constructions to arbitrary minimal parabolic subgroups; the following basic fact  \cite[Théorème~4.13(b)]{BorelTits65} will be tacitly used from now on:
\begin{proposition}\label{prop:min_par_conj}
Any two minimal $k$-parabolics are conjugate under the action of $\bfG(k)$.\qed
\end{proposition}
We say that a maximal parabolic is of \emph{type $i$} if it is conjugate to $\bfP_i$. It follows from Proposition \ref{prop:min_par_conj} and the fact that $\bfP_1, \dots, \bfP_r$ are mutually non-conjugate that every maximal parabolic has a unique type and that every minimal parabolic subgroup is contained in a unique maximal parabolic subgroup of each type. If $\bfQ$ is a minimal parabolic subgroup then we denote by $\bfQ_i$ the unique maximal parabolic subgroup of type $i$ containing $\bfQ$. We then denote by $\chi^{\bfQ}_i \in \X[k](\bfQ)$ (or simply $\chi_i$ if $\bfQ$ is clear from context) the canonical character of $\bfQ_i$. Note that if $g \in \bfG(k)$ is such that $\bfQ = {}^g\bfP$ then $\bfQ_i = {}^g\bfP_i$ for all $i \in \{1, \dots, r\}$. We deduce that in this case $\chi_i^{\bfQ}(p) = \chi_i^{\bfP}(g^{-1}pg)$.

Similarly we define roots $\alpha_i^\bfQ \in \X[k](\bfQ)$ by $\alpha_i^\bfQ(p) = \alpha_i^\bfP(g^{-1}pg)$. This is well-defined since roots are conjugation-invariant and $\bfP$ is self-normalizing. With these definitions $(\alpha_1^\bfQ,\ldots,\alpha_r^\bfQ)$ and $(\chi_1^\bfQ,\ldots,\chi_r^\bfQ)$ are bases of $\X[k](\bfQ) \otimes \R$ and the relations \eqref{RootsVsCharacters} hold for every minimal parabolic $\bfQ$ with constants $c_{ij}$ and $n_{ji}$ independent of $\bfQ$. In particular, Lemma~\ref{lem:root_weight} holds for arbitrary minimal parabolics.

\subsection{Euclidean buildings and Riemannian symmetric spaces}\label{sec:local_spaces}
We now assume that $k, V, V\fin, V\inf$ and $\bfG \subset \GL_N$ are as in Convention \ref{ConventionAdelicGroups}. Moreover, we assume that $\bfG$ is semisimple. For every place $s \in V$ we denote by $k_s$ the corresponding completion of $k$. We are going to associate a \cato{}-space $X_s$ with each of the groups $\bfG(k_s)$.

Assume first that $s \in V\inf$ is an infinite place, so that $k_s \in \{\R, \C\}$ is Archimedean. In this case, $\bfG(k_s)$ is a semisimple real Lie group with finite center and hence
the Lie algebra $\mathfrak g_s$ of $\bfG(k_s)$ admits a Cartan decomposition $\mathfrak g_s = \mathfrak k_s \oplus \mathfrak p_s$ in the sense of \cite[Section III.7]{Helgason}, which is unique up to inner automorphisms. We fix a choice of Cartan decomposition for each $s \in V\inf$ and set $C_s \defeq \exp(\mathfrak k_s)$ and $X_s \defeq \bfG(k_s)/C_s$. The tangent space $\mathfrak p_s$ of $X_s$ is then a Lie triple system \cite[Section IV.7]{Helgason}, and hence gives rise to the structure of a Riemannian symmetric space on $X_s$, which is either trivial (if $\bfG(k_s)$ is compact) or of non-compact type. In any case, $X_s$ is a \cato{} space (see \cite[Theorem~II.10.58]{BridsonHaefliger}) and $\bfG(k_s)$ acts properly and transitively by isometries on $X_s$. It then follows from the Bruhat--Tits fixed-point theorem that $C_s$ is a maximal compact subgroup of $\bfG(k_s)$ and that any other maximal compact subgroup of $\bfG(k_s)$ is conjugate to $C_s$. Note that $X_s$ is independent of the choice of Cartan decomposition, whereas the basepoint $C_s$ depends on this choice. We refer to $X_s$ as the \emph{symmetric space} of $\bfG(k_s)$ and denote by $o_s \defeq 1 \cdot C_s \in X_s$ our chosen basepoint.

We now consider the case of a finite place $s\in V^{\mathrm{fin}}$ so that $k_s$ is non-Archimedean. By \cite{BruhatTits2} the group $\bfG(k_s)$ admits a root group datum with a valuation which by \cite{BruhatTits1} gives rise to a Euclidean building $X_s$ whose dimension is the $k_s$-rank $\mathrm{rk}_{k_s}(\bfG)$ of $\bfG$, see also \cite{KalethaPrasad}. If we consider $X_s$ as a metric space, by turning each apartment into a Euclidean space of dimension $\mathrm{rk}_{k_s}(\bfG)$ then $X_s$ is a \cato{} space by \cite[Lemme~2.5.14]{BruhatTits1}, see also \cite[Theorem~11.16]{abramenko08}. Moreover, the action of $\bfG(k_s)$ of $X_s$ is proper and cocompact.

Again the Bruhat--Tits fixed point theorem implies that the maximal compact subgroups of $\bfG(k_s)$ are stabilizers of points in $X_s$. This time, however, there is more than one such group up to conjugation. Our embedding $\bfG(k_s) \hookrightarrow \GL_N(k_s)$ allows us to define a compact-open subgroup \[C_s' \defeq \bfG(k_s) \cap \GL_N(\mathcal O_s)\]
of $\bfG(k_s)$. The following example shows that in general we cannot expect $C_s'$ to be maximal for \emph{all} places $s$:
\begin{example}\label{exmp:not_maximal}
Consider $\bfG < \GL_N$ and let $g \in \bfG(k)$ be arbitrary. Then the subgroup $\bfH$ of $\GL_{2N}$ consisting of block matrices
\[
\left(
\begin{array}{c|c}
A&0\\
\hline
0&gAg^{-1}
\end{array}
\right)
\]
with $A \in \bfG$ is clearly isomorphic to $\bfG$ over $k$. But $\bfH(k_s) \cap \GL_{2N}(\calO_s) = C_s' \cap g^{-1}C_s'g$ which is a generally proper subgroup of $C_s'$ for the finitely many places $s$ for which the entries of $g$ and $g^{-1}$ do not lie in $\calO_s$. In particular, using $\bfH$ instead of $\bfG$ leads to subgroups $C_s'$ that are not maximal.
\end{example}
\begin{lemma}\label{obs:c_c'}
The group $C'_s$ is maximal compact for almost all $s \in V\fin$ and of finite index in a maximal compact subgroup $C_s$ for the remaining $s \in V\fin$.
\end{lemma}
\begin{proof} We denote by $C_s$ a maximal compact subgroup of $\bfG(k_s)$ containing $C_s'$. Since $C_s'$ is open, it has finite index in the compact group $C_s$. Moreover, $C_s'$ is maximal compact (even hyperspecial) for almost all $s \in V\fin$ by \cite[3.9.1]{tits79}.
\end{proof}
From now on, given $s \in V\fin$, we reserve the latter $C_s$ to denote a maximal compact subgroup of $\bfG(k_s)$ containing $C_s'$ so that $C_s = C_s'$ for almost all $s \in V\fin$. By the Bruhat--Tits fixed-point theorem, the group $C_s$ fixes a point $o_s \in X_s$. By properness of the action, the stabilizes of $o_s$ is then compact, and hence coincides with $C_s$ by maximality. 

To summarize, for every $s \in V$, we have constructed a \cato{} space $X_s$ on which $\bfG(k_s)$ acts properly and cocompactly, and the maximal compact subgroup $C_s < \bfG(k_s)$ is a point stabilizer in $X_s$. In the finite case we could even arrange for $C_s$ to be the stabilizer of a special vertex of $X_s$ but we will have no need for that. 

\begin{lemma}\label{obs:self-normalizing}
For every $s \in V$ the group $C_s$ is self-normalizing in $\bfG(k_s)$.
\end{lemma}

\begin{proof}
If $X_s$ is a thick building, we first argue that the fixed-point set $\Fix(C_s)$ of $C_s$ in $X_s$ is bounded, which follows from strong transitivity: if $c \subseteq X_s$ is a (compact) chamber and $p \subseteq \partial c$ is a panel, the stabilizer of $c$ acts transitively on the at least two other chambers that also contain $p$. This shows that $\Fix(\Stab_{\bfG(k_s)}(c)) = c$. It follows that fixed point set of the stabilizer of any point is at most a chamber, hence bounded.

Since the normalizer $N \defeq N_{\bfG(k_s)}(C_s)$ acts on $\mathrm{Fix}(C_s)$ it then has a bounded orbit in $X_s$, hence a fixed-point by the Bruhat--Tits fixed-point theorem. By properness of the action, this implies that $N$ is compact, so by maximality we have $N = C_s$.

If $X_s$ is a symmetric space, the Iwasawa decomposition implies that the stabilizer of a point $x$ acts transitively on the Weyl chambers that have $x$ as (unique) tip. This implies in particular that $x$ is the unique fixed point of this stabilizer, hence it is also fixed by its normalizer.
\end{proof}

The previous lemma is not true if $\bfG$ is merely reductive: if $s$ is finite and one takes the extended model $X_s$ then chambers are generally not compact but products of polysimplices with a flat space; if $s$ is infinite, Weyl chambers generally have a flat space as tip. In both cases the center acts on a copy of $\R^n$ cocompactly but the action on the boundary is trivial. This is the reason why we are restricting to semisimple groups at the moment.

In the sequel we denote by $\mathcal K_s$ the set of all maximal compact subgroups of $\bfG(k_s)$ which are conjugate to $C_s$. By Lemma \ref{obs:self-normalizing} we have bijections
\begin{eqnarray*}
\bfG(k_s).o_s \leftarrow &\bfG(k_s)/C_s & \to \mathcal K_s\\
g.o_s \mapsfrom & gC_s & \mapsto gC_sg^{-1}.
\end{eqnarray*}
In particular, since the orbit $\bfG(k_s).o_s$ is relatively dense in $X_s$, we may coarsely identify $\bfG(k_s)$, $X_s$ and $\mathcal K_s$. 

In our analysis of the spaces $X_s$ an important role will be played by maximal flats, a flat being a subspace isometric to a Euclidean space. If $X_s$ is the symmetric space or Bruhat building of $\bfG(k_s)$ then every flat (and in particular every geodesic, and even every geodesic ray) is contained in a maximal flat. Moreover, $\bfG(k_s)$ acts transitively on maximal flats, and hence all maximal flats have the same dimension, which coincides with the $k_s$-rank of $\bfG$. Following standard building terminology, we are going to refer to a maximal flat as an \emph{apartment}. This terminology is not standard in the Archimedean case, but will help us to avoid case distinctions.

\subsection{The spherical building at infinity}\label{sec:bldg_infty}
Recall that if $X$ is a \cato{} space then two geodesic rays $\rho, \rho'\colon [0,\infty) \to X$ are called \emph{asymptotic} if their images are at bounded Hausdorff distance. This defines an equivalence relation, and we denote by $\rho(\infty)$ the equivalence class of $\rho$. The set $\partial X$ of all such equivalence classes is called the \emph{visual boundary} of $X$, see \cite[Chapter~II.8]{BridsonHaefliger}. There are several natural topologies on the visual boundary of a \cato{} space such as the visual and the Tits topology \cite{BridsonHaefliger}, but we will not use these topologies and only consider $\partial X$ as a set.

Our next goal is to describe some additional structure on the visual boundaries of the spaces $X_s$ from the previous section. By a \emph{simplicial complex} $\Delta$ we shall always mean the geometric realization of an abstract simplicial complex, i.e.\ a space glued together by simplices according to the usual axioms; by a \emph{simplicial structure} on a space $Y$ we shall mean a bijection $\Delta \to Y$, where $\Delta$ is a simplicial complex.

If $\bfG$ is a semisimple group over an arbitrary field $k$ then we denote by $\Delta_{\bfG, k}$ the \emph{spherical building} of the group $\bfG(k)$. By definition this is the unique simplicial complex whose set of simplices $\mathcal S(\Delta_{\bfG, k})$ is given by the proper $k$-parabolic subgroups of $\bfG(k)$, ordered by reverse inclusion. The maximal simplices (i.e.\ minimal $k$-parabolic subgroups) are called the \emph{chambers} of $\Delta_{\bfG, k}$ and their collection is denoted by $\mathcal C(\Delta_{\bfG, k})$. Similarly we denote by
$\mathcal V(\Delta_{\bfG, k})$ the set of vertices, i.e.\ maximal $k$-parabolic subgroups of $\Delta_{\bfG, k}$. By definition, a family of vertices lies in a common simplex if and only if its intersection is a parabolic subgroup. For $i \in \{1, \dots, r\}$ we will denote by $\mathcal V_i(\Delta_{\bfG, k}) \subseteq \mathcal V(\Delta_{\bfG, k})$ the subset of vertices (i.e.\ maximal $k$-parabolic subgroups) of type $i$. Then 
\[
\mathcal V(\Delta_{\bfG, k})=\bigsqcup_{i=1}^r \mathcal V_i(\Delta_{\bfG, k})
\]
and each chamber contains precisely one vertex of each type.  The group $\bfG(k)$ acts on $\mathcal S(\Delta_{\bfG, k})$ by conjugation, and this induces a simplicial and type-preserving action on $\Delta_{\bfG, k}$. By Proposition \ref{prop:min_par_conj} this action is chamber-transitive. If $\bfG$ (respectively $\bfG$ and $k$) are clear from context we simply write $\Delta_k$ (or even $\Delta$) instead of $\Delta_{\bfG, k}$.

If $k$ is an arbitrary field and $\bar{k}$ is a separable closure then the Galois group $\Gal(\bar{k}/k)$ acts on $\Delta_{\bar{k}}$ by simplicial isometries. However, this action is not without inversions on the $1$-skeleton, and hence the fixed point set $\Delta_{\bar{k}}^{\Gal(\bar{k}/k)} \subset \Delta_{\bar{k}}$ is not a subcomplex (but only a subcomplex of the barycentric subdivision of $\Delta_{\bar{k}}$). In any case, this fixed point set carries a simplicial structure and is isomorphic to the spherical building $\Delta_k$, cf.\ \cite[\S 6]{BorelTits65}, \cite[Section~2]{TitsClassSemiSimple}.
If $k'/k$ is an extension and $\bar{k}'$ is a separable closure of $k'$ containing $\bar{k}$ then, since $\bfG$ is split over $\bar{k}$ as well as over $\bar{k}'$, the building $\Delta_{\bar{k}}$ is a chamber subcomplex of $\Delta_{\bar{k}'}$. Thus $\Delta_k$ and $\Delta_{k'}$ are naturally subspaces of $\Delta_{\bar{k'}}$. Since $\bar{k}$ is $\Aut(\bar{k}'/k)$-invariant, so is $\Delta_{\bar{k}}$ and there is an action via $\Gal(\bar{k}'/k') \to \Aut(\bar{k}'/k) \to \Gal(\bar{k}/k)$. It follows that $\Delta_k \subseteq \Delta_{k'}$ as subspaces of $\Delta_{\bar{k}'}$.

We now return to the previous setting where $k$ is a global field with set of places $V$. For each $s \in V$ the space $X_s$ is a union of apartments, and since every geodesic is contained in such an apartment, the visual boundary $\partial X_s$ is the union the boundary spheres $\partial A$, where $A$ ranges over the set $\mathcal A_s$ of maximal flats in $\partial X_s$. Each of these boundary spheres can be identified with an apartment in the spherical building $\Delta_{k_s}$, and these identifications can be glued together to obtain a $\bfG(k_s)$-equivariant bijection $\Delta_{k_s} \to \partial X_s$ (since the corresponding stabilizers match). This defines a $\bfG(k_s)$-invariant simplicial structure on $\partial X_s$, hence we can talk about vertices, simplices and chambers in $\partial X_s$. Although we do not need any topology on $\partial X_s$, we mention in passing that the visual topology on $\partial X_s$ corresponds to the topology on $\Delta_{k_s}$ obtained by topologizing the set of chambers by the quotient topology with respect to $\bfG(k_s)$, whereas the Tits topology corresponds to the topology on $\Delta_{k_s}$ obtained by topologizing the set of chambers by the discrete topology.

We observe for later use that for every $s \in V$ the embedding $k \hookrightarrow k_s$ induces a canonical embedding
\[
\Delta_k \hookrightarrow \Delta_{k_s} \cong \partial X_s,
\]
by the above discussion, i.e.\ the rational building embeds canonically (but not simplicially) into the visual boundary of each of the \cato{} spaces $X_s$.

\subsection{$S$-adic and adelic geometry}\label{sec:sadic_geometry}
We now consider the following situation:
\begin{convention}\label{ConventionSemisimpleGroups}
\begin{enumerate}
\item $k$, $V$, $V\fin$, $V\inf$ are as in Convention \ref{ConventionAdelicGroups}; 
\item $\bfG$ and $\Gamma$ are as in Convention \ref{ConventionAdelicGroups}, but in addition $\bfG$ is assumed to be semisimple;
\item for every $s \in V$ we denote by $X_s$ the \cato{} space associated with $\bfG(k_s)$ and by $o_s \in X_s$ a basepoint such that $C_s \defeq \mathrm{Stab}(o) < \bfG(k_s)$ is maximal compact and contains  $C_s' \defeq \bfG(k_s) \cap \GL_N(\mathcal O_s)$ if $s \in V\fin$; 
\item $S \subset V$ is finite and non-empty.
\end{enumerate}
\end{convention}
\begin{remark}
In the situation of Convention \ref{ConventionSemisimpleGroups}, $X_S \defeq \prod_{s \in S} X_s$ is a \cato{} space, and $\bfG_S$ acts properly and cocompactly on $X_S$. The boundary $\partial X_S$ is the spherical join $\partial X_S = \bigast_{s \in S} \partial X_s$, and the spherical building $\Delta_k$ embeds diagonally into this boundary. Moreover, the stablilizer of the basepoint $o_S \defeq (o_s)_{s \in S}$ is the maximal compact subgroup $C_S \defeq \prod_{s \in S} C_s$ of $\bfG_S$, which is self-normalizing by Lemma \ref{obs:self-normalizing}. In particular, if we denote by $\mathcal K_S$ the collection of conjugates of $C_s$ then $\mathcal K_S$ can be identified with the $\bfG_S$-orbit of $o_S$, and hence we have coarse equivalences
\[
\mathcal K_S \cong X_S \cong \bfG_S.
\]
Note that these spaces are also quasi-isometric to the group $\bfG(\A_S) = \bfG_S \times \prod_{s \in V \setminus S} C_s'$ and the slightly larger group $\bfG_S \times \prod_{s \in V \setminus S} C_s$.
\end{remark}

\begin{remark} In previous sections we have identified the abstract group $\bfG(\A)$ of adelic points with the colimit of the groups $\bfG(\A_S)$ over all finite subsets $S \subseteq V$ containing $V\inf$. In this model of $\bfG(\A)$, the compact groups $C'_s$ play a special role in our model for the group $\bfG(\A)$, whereas in the sequel we prefer to work with the point stabilizers $C_s$, which have a clearer geometric meaning. To this end we observe that by Lemma~\ref{obs:c_c'} the group $\bfG(\A)$ can also be identified with the colimit
\[
\lim_S \left(\bfG_S \times \prod_{s \in V \setminus S} C_s\right).
\]
In the sequel, we will prefer to work in this model of $\bfG(\A)$. All the results of the previous sections remain valid with $C_s'$ replaced by $C_s$.
\end{remark}
\begin{remark}
For every finite $S \subset V$ we have given two coarse models for the groups $\bfG_S$ and $\bfG(\A_S)$, namely $X_S$ and $\mathcal K_S$. There are a similar coarse models for the full adelic group $\bfG(\A)$, although they are no longer proper. It will not technically be needed in what follows, but we think it helps to understand what is going on.

Given a finite subset $S \subseteq V$, we consider the \cato{} space $\hat{X}_S \defeq X_S \times \prod_{s \in V \setminus S} \{o_s\} \cong X_S$. For $S \subseteq S'$ we then have a natural inclusion $\hat X_S \hookrightarrow \hat X_{S'}$. The colimit $X = \lim_S \hat X_S$ over all finite subsets $S \subseteq V$ is then a \cato{} space on which $\bfG(\A)$ acts and the action is proper in the following sense: for every $x \in X$ there is an $r > 0$ such that $\{g \in \bfG(\A) \mid d(g.x,x) < r\}$ is relatively compact in $\bfG(\A)$ (provided we scale the metrics on $X_s, s\in V\fin$ so that edge lengths are uniformly bounded away from $0$). The space $X$ is not complete, the metric topology is strictly weaker than the colimit topology and neither is locally compact. In any case, if we denote by $o = (o_s)_{s \in V} \in X$ the canonical basepoint then $C$ is precisely the stabilizer of $o$ in $\bfG(\A)$.

In the sequel we denote by $\mathcal K$ the collection of all conjugates of $C$ in $\bfG(\A)$. Note that $C$ is self-normalizing by Lemma \ref{obs:self-normalizing}, and hence we may identify $\mathcal K$ with the $\bfG(\A)$-orbit of $o$ in $X$. We then have coarse equivalences
\[
\calK \cong X \cong \bfG(\A)
\]
where $\bfG(\A)$ is equipped with its canonical coarse structure as a locally compact group, $X$ with the coarse structure coming from the metric and $\calK$ with the coarse structure inherited from $\bfG(\A)/C$ or $X$, which coincide. Note that working with $\calK$ as a coarse model for $X$ is in line with \cite{harder69}, and that the rational spherical building $\Delta_k$ still embeds diagonally into the visual boundary $\partial X$.
\end{remark}

We conclude this section with the following almost transitivity result concerning $C$:
\begin{lemma}
\label{lem:iwasawa}
For every $\bfQ \in \mathcal S(\Delta_k)$ the double-coset spaces $C \backslash \bfG(\A) / \bfQ(\A)$ and $\bfQ(\A) \backslash \bfG(\A) /C$ are finite.
\end{lemma}

\begin{proof} If we replace $C$ by the group

\[
\widehat{C} \defeq \prod_{s \in V\fin} C'_s \times \prod_{s\in V\inf} \bfG(k_s),
\]
then the corresponding statement is established on p.10 in Section~7 of \cite{godement64} based on the fact that $ \bfG(\A) / \bfQ(\A) =  (\bfG/ \bfQ)(\A)$ is compact and $\widehat{C}$ is open. Now, firstly, the difference between $C_s$ and $C_s'$ is immaterial by Lemma \ref{obs:c_c'}. Secondly, for every infinite place the double quotient $C_s \backslash \bfG(k_s) /\bfQ(k_s)$ is reduced to a point by the Iwasawa decomposition, and hence the lemma follows.
\end{proof}

\subsection{Reductive groups}

We now briefly explain, how the results for semisimple groups discussed so far extend to reductive groups, so we take $\bfG$ to be a reductive $k$-group. We use some standard facts. The radical $\bfZ = R(\bfG)$ of $\bfG$ is the connected center and is a torus defined over $k$ \cite[2.15(e)]{BorelTits65}. The derived subgroup $\bfH = \mathscr{D}\bfG$ is semisimple and $\bfG$ is an almost-direct product of $\bfZ$ and $\bfH$ \cite[2.2]{BorelTits65}. In particular, $\bfG/\bfZ$ is semisimple. By our previous discussion, for every $s \in V$ there is a space $X_s$ attached to it on which $\bfG(k_s)$ acts through the quotient. Moreover $\bfG(k_s)$ acts through $(\bfG/\bfH)(k_s)$ on $F_s \defeq \coX[k_s](\bfG/\bfH) \otimes \R$ whose dimension is the rank of $\bfZ$ and can be equipped with a flat metric. Putting things together we get a proper and cocompact action of $\bfG(k_s)$ on the space $X_s^1 = X_s \times F_s$. The space $X_s^1$ is still a symmetric space respectively a Euclidean building. Its boundary $\partial X_s^1 = \partial X_s * \partial F_s$ is a join of that of $X_s$ and a sphere, hence a spherical building. As in the semisimple case we can then choose a maximal compact subgroup $C_s$ of $\bfG(k_s)$ which is given by the stabilizer of some basepoint $o_s \in X_s^1$ and in the non-Archimedean case contains the compact-open subgroup $C_s' \defeq \bfG(k_s) \cap \GL_N(\mathcal O_s)$. We then still have coarse equivalences
\[
\bfG(k_s) \cong \bfG(k_s)/C_s \cong X_s^1,
\]
but the set $\calK_S$ of conjugates of $C_s$ is \emph{not} in general coarsely equivalent to these spaces, since $C_s$ is not-self-normalizing due to the center. In fact, by the almost-direct product decomposition, the action of $\bfH(k_s)$ on $\calK_s$ is cocompact and hence gives rise to coarse equivalences
\[
\bfH(k_s) \cong X_s \cong \calK_s.
\]
In the sequel we will work in the following setting, which is slightly more general than that of Convention \ref{ConventionSemisimpleGroups}:
\begin{convention}\label{ConventionReductiveGrps}
\begin{enumerate}
\item $k$, $V$, $V\fin$, $V\inf$ are as in Convention \ref{ConventionAdelicGroups}; 
\item $\bfG$ and $\Gamma$ are as in Convention \ref{ConventionAdelicGroups}, but in addition $\bfG$ is assumed to be reductive of global rank $r \defeq \rk_k \bfG$ with derived subgroup $\bfH \defeq \mathscr{D}\bfG$;
\item for every $s \in V$ we denote by $X_s$ and $X_s^1$ respectively  the \cato{} space associated with $\bfH(k_s)$ and $\bfG(k_s)$; we also denote by $o_s \in X_s$ a basepoint such that $C_s \defeq \mathrm{Stab}(o) < \bfG(k_s)$ is maximal compact and contains  $C_s' \defeq \bfG(k_s) \cap \GL_N(\mathcal O_s)$ if $s \in V\fin$; 
\item $S \subset V$ is finite and non-empty.
\end{enumerate}
\end{convention}
In this situation we then set 
\[
 X_S^{1} \defeq \prod_{s \in S} X_s^1, \quad  X_S \defeq \prod_{s \in S} X_s \quad \text{and} \quad C_S \defeq \mathrm{Stab}((o_s)_{s\in S}) = \prod_{s \in S} C_s.
\]
We also denote by $\calK_S$ the space of conjugates of $C_S$ in $\bfG(k_S)$ and obtain coarse equivalences
\begin{equation}\label{CEq1}
\bfG(\A_S) \cong \bfG_S \cong \bfG_S/C_S \cong X^1_S \quad \text{and} \quad \bfH(\A_S) \cong \bfH_S \cong X_S \cong \calK_S.
\end{equation}
Finally, we also have infinite versions of these coarse equivalences: For this we denote by $\calK$ the conjugacy class of the maximal compact subgroup $C \defeq \prod_{s \in V} C_s$ in $\bfG(\A)$, which is the stabilizer of $o = (o_s)_{s\in V}$ in 
\[
X^1 = \lim_S \left(X_S^1 \times \prod_{s \in S \setminus V} \{o_s\}\right).
\]
If we similarly denote by $X$ the adelic space associated with the semisimple group $\bfH$, then we obtain coarse equivalences
\begin{equation}\label{CEq2}
\bfG(\A) \cong \bfG(\A)/C \cong X^1 \quad \text{and} \quad \bfH(\A) \cong X \cong \calK.
\end{equation}

\begin{remark}
Our version of adelic reduction theory in Section~\ref{sec:adelic_reduction} will be formulated for general reductive groups $\bfG$, but since it is based on the set $\calK$ of conjugates of $C$ it only really sees the semisimple group $\bfH$ and disregards the center $\bfZ$.
\end{remark}

\subsection{$S$-adic approximate groups of reductive type}
In the following sections we will work in the following section:
\begin{convention}\label{SettingReductiveApprox}
\begin{enumerate}
\item $k$, $V$, $V\fin$, $V\inf$ are as in Convention~\ref{ConventionAdelicGroups}; 
\item $\bfG$, $r$, $S$, the sequences $(o_s)_{s\in V}$, $(X_s)_{s\in V}$, $(X_s^1)_{s \in V}$, $(C_s)_{s \in V}$, $(C'_s)_{s \in V\fin}$ and $(\calK_s)_{s \in V}$ and the products $X^1_S$, $X_S$ and $C_S$ are as in Convention~\ref{ConventionReductiveGrps} and the following remarks; 
\item  $G \defeq \bfG_S$, $H \defeq \bfG_{V \setminus S}^\sharp$ and $\Gamma \defeq \bfG(k)$ so that $(G, H, \Gamma)$ is the adelic generalized cut-and-project scheme with parameters $(k, S, \bfG)$ (cf.\ Definition~\ref{adeliccup});
\item $W_s \subseteq \bfG(k_s)$ is a compact symmetric subset with non-empty interior for all $s \in V \setminus S$ such that $W_s = C_s$ if $s$ is finite, and $W = \prod_{s\in V \setminus S} W_s$;
\item $\bfG(\calO_S)$ denotes the $S$-arithmetic approximate subgroup $\bfG(\calO_S) \subset \bfG(k)$ given by
\[
\bfG(\calO_S) \defeq \Lambda(G, H, \Gamma, W) = \{\gamma \in \bfG(k) \mid \forall\, s \in V \setminus S: \gamma \in W_s\}.
\]
\end{enumerate}
\end{convention}
\begin{remark} In the case where $S$ contains all infinite places, our new definition of $\bfG(\calO_S)$ differs from our previous definition of the $S$-arithmetic group $\bfG(\calO_S)$ in that we use the maximal compact groups $C_s$ instead of the compact-open groups $C'_s$ at the finite places. However, by Lemma \ref{obs:c_c'} this does not change the commensurability class of $\bfG(\calO_S)$, which is all that we care about here. If $s \in V \setminus S$ is infinite, then we allow the corresponding window $W_s$ to be completely arbitrary, but again these choices will not influence the commensurability class of $\bfG(\calO_S)$ by Corollary \ref{Commens}.
\end{remark}
\begin{remark}\label{NonCommutingDiagram}
We now consider the quasi-actions of $\bfG(\calO_S)$ on the space $X_S$ and on the subset $\Delta_k \subset \partial X_S$. It is important to note that these two quasi-actions are not quite compatible, unless $S$ contains all infinite places.

To explain this phenomenon we use the notation of Subsection \ref{SecModelSets} with $\Gamma = \bfG(k) < G \times H < \bfG(A)$.
We recall that there is an isomorphism $\Gamma_G \to \Gamma_H$, $\gamma \mapsto \gamma^*$ so that
\[
\Gamma = \{(\gamma, \gamma^*)\mid \gamma \in \Gamma_G\}
\]
The approximate group $\bfG(\calO_S)$ is by definition a subset of $\Gamma_G \subseteq G = \bfG_S$, and it is through this embedding that $\bfG(\calO_S)$ quasi-acts on $X_S$. Equivalently, we can also consider $\bfG(\calO_S)$ as embedded into $\bfG(\A)$ as $\{(\gamma,1) \mid \gamma \in \bfG(\calO_S)\}$. Then the quasi-action on $X_S$ is via the action of $\bfG(\A)$.

The natural inclusion of $\bfG(\calO_S)$ into $\Gamma = \bfG(k) \subseteq \bfG(\A)$, however, takes $\gamma$ to $(\gamma,\gamma^*)$. It is through this map that $\bfG(\calO_S)$ acts on the rational building $\Delta_k$, and we have the following non-commutative diagram:
\begin{center}
\begin{tikzpicture}[xscale=2,yscale=-2]
\node (Sarith) at (0,0) {$\bfG(\calO_S)$};
\node (global) at (1,0) {$\bfG(k)$};
\node (local) at (0,1) {$\bfG_S$};
\node (adele) at (1,1) {$\bfG(\A)$};
\node (Selt) at (.25,.25) {$\gamma$};
\node (gelt) at (1.5,.25) {$(\gamma,\gamma^*)$};
\node (lelt) at (.25,1.25) {$\gamma$};
\node (aelt1) at (1,1.25) {$(\gamma,1)$};
\node (aelt2) at (1.5,1) {$(\gamma,\gamma^*)$};
\draw[->]  (Sarith) -- (global);
\draw[->] (Sarith) -- (local);
\draw[->] (local) -- (adele);
\draw[->]  (global) -- (adele);
\draw[|->]  (Selt) -- (gelt);
\draw[|->] (Selt) -- (lelt);
\draw[|->] (lelt) -- (aelt1);
\draw[|->]  (gelt) -- (aelt2);
\end{tikzpicture}
\end{center}
In the group case, i.e.\ if $S$ contains all infinite places and $W$ is chosen to be the compact open subgroup $C_{V \setminus S}$, this is not a problem because then $\gamma^* \in C_{V \setminus S}$ acts trivially on $\calK_{V \setminus S}$ and hence the actions of $\bfG(\calO_S)$ on $X_S$ and $\Delta_k \subset \partial X_S$ are compatible. In the approximate setting, however, this is not the case, and hence the two actions are only ``almost compatible''.
\end{remark}

\section{Rescaled Busemann functions and invariant horofunctions}\label{sec:busemann}

In this section we discuss Busemann functions on \cato-spaces. Specifically, Busemann functions on the space $X_S$ from the last section will arise from an algebraically defined map that we call an invariant horofunction.

\subsection{Rescaled Busemann functions for Bruhat-Tits buildings and symmetric spaces}\label{sec:building_busemann}

Recall that if $X$ is a \cato{} space, then points in the visual boundary $\partial X$ of $X$ can also be represented by Busemann functions: Given a geodesic ray $\rho\colon [0,\infty) \to X$ the associated \emph{Busemann function} is
\begin{equation}\label{eq:busemann}
\beta_\rho \colon X \to \R, \quad \beta_\rho(x) \defeq\lim_{t \to \infty} d(\rho(t),\rho(0)) - d(\rho(t), x)
\end{equation}
where $d(\rho(t),\rho(0)) = t$. We say that $\beta_\rho$ is \emph{centered at} $\xi \defeq \rho(\infty) \in \partial X$. Two Busemann functions $\beta,\beta'$ centered at the same point $\xi$ differ from each other by an additive constant, i.e.\ they have the same level-sets and are determined by their zero level-set. If a zero level has been fixed or if it does not matter, we will often denote a Busemann function centered at $\xi$ by $\beta_\xi$. A \emph{rescaled Busemann function} is a function $\alpha \colon X \to \R$ for which there is a constant $b > 0$ and a Busemann function $\beta$ such that $\alpha(x) = b \cdot \beta(x)$ for all $x \in X$. 

\begin{example}[Euclidean space]\label{AffineForms}
If $\Sigma$ is an affine space with underlying vector space $V$ and $\langle \cdot, \cdot \rangle$ is an inner product on $V$, then we refer to the triple $(\Sigma, V, \langle \cdot, \cdot \rangle)$ (or just $\Sigma$ for short) as a \emph{Euclidean space}. Given $x,y \in \Sigma$ we denote by $y-x \in V$ the unique vector such that $x+(y-x) = y$; then $\Sigma$ is a metric space with respect to $d(x,y) \defeq \sqrt{\langle x-y, x-y\rangle}$ and we denote by $\Sigma^1$ the unit sphere in $V$ and by $\Sigma^1_o$ the unit sphere around some $o \in \Sigma$ so that $\Sigma^1 \to \Sigma^1_o$, $v \mapsto o+v$ is an isometry.

A point $o \in \Sigma$ and a non-zero vector $v \in V \setminus \{0\}$ determine a non-constant affine form $\beta_{o,v} \colon V \to \R$, $x \mapsto \gen{v,x-o}$. If $v \in \Sigma^1$, then $\beta_{o,v}$ is the Busemann function associated to the geodesic $\rho_{o,v} \colon t \mapsto o + t v$. In particular, we obtain a homeomorphism $\Sigma^1 \to \partial \Sigma$, $v \mapsto \rho_{o,v}(\infty)$ for every $o \in V$.

More generally, if $v \in V \setminus \{0\}$ (not necessarily a unit vector), then the affine form $\beta_{o,v} = \norm{v} \cdot \beta_{o,v/\norm{v}}$ is a rescaled Busemann function, and hence the map $v \mapsto \beta_{o,v}$ defines a bijection between $V \setminus \{0\}$ and the set of rescaled Busemann functions (a.k.a.\ non-constant affine forms). Under the identification $\Sigma^1 \cong \partial \Sigma$ the center of the affine form $\beta_{o,v}$ corresponds to $v/\norm{v}$. In the sequel we refer to $v$ as the \emph{linear part} of $\beta=\beta_{o,v}$ and denote it by $\mathring{\beta}$.
\end{example}
Since every geodesic in a Euclidean building or Riemannian symmetric space of non-compact type is contained in an apartment, which is a Euclidean space, this example generalizes as follows:
\begin{proposition}\label{prop:busemann_char}
Let $X$ be a Euclidean building or a Riemannian symmetric space of non-compact type and let $c \subseteq \partial X$ be a chamber of the building at infinity. A map $\beta \colon X \to \R$ is a rescaled Busemann function centered at a point of the closed chamber $c$ if and only if the following conditions are satisfied:
\begin{enumerate}
\item For every apartment $\Sigma$ of $X$ that contains $c$ in its boundary the restriction $\beta|_\Sigma$ is affine.\label{eq:busemann_char_affine}
\item For some apartment $\Sigma_0$ of $X$ that contains $c$ in its boundary $\beta|_{\Sigma_0}$ is centered at a point of $c$.\label{eq:busemann_char_direction}
\item If $X$ is a symmetric space, then moreover $\beta(x) = \beta(u.x)$ for every $x \in X$ and every unipotent element $u$ in the stabilizer of $c$ in $\Isom(X)$.\label{eq:busemann_char_unipotent}
\end{enumerate}
\end{proposition}
\begin{proof} Assume first that $\beta$ is an affine Busemann function centered at a point $\xi \in c$. If $\Sigma$ is an apartment that contains $c$ in its boundary then there exists a geodesic ray $\rho$ contained in $\Sigma$ with $\xi = \rho(\infty)$. Then, up to rescaling we have $\beta|_{\Sigma} = \beta_\rho$, which is an affine function on $\Sigma$ centered at $\xi$ by Example \ref{AffineForms}, and hence (1) and (2) hold. In the symmetric space case, also (3) holds by \cite[Lemma 10.26(2)]{BridsonHaefliger}.

Conversely assume that (1) and (2) (and, in the symmetric space case also (3)) hold. For every $\Sigma$ as in (1), then $\beta|_\Sigma$ is a rescaled Busemann function by Example \ref{AffineForms}. Moreover, by (2) there exists an apartment $\Sigma_0$, a geodesic ray $\rho \colon [0,\infty) \to \Sigma_0$ with $\rho(\infty) = \xi \in c \subset \partial \Sigma_0 \subset \partial X$ and associated Busemann function $\beta' = \beta_\rho: X \to \R$ and a constant $b>0$ such that $\beta|_{\Sigma_0} = (b \cdot \beta')|_{\Sigma_0}$. We claim that in fact $\beta$ and $b\cdot \beta'$ coincide on all of $X$. In fact, it suffices to show that they coincide on all apartments $\Sigma$ which contain $c$ in their boundary, since these cover $X$.

Thus let $\Sigma$ be an apartment containing $c$ in its boundary. If $X$ is a Euclidean building, then $\Sigma$ meets $\Sigma_0$ in a set containing a Weyl chamber. Thus $\beta$ coincides with $b \cdot \beta'$ on an open set of $\Sigma$ and since both are affine on $\Sigma$, they coincide on all of $\Sigma$. If $X$ is a symmetric space then there is an element $u$ as in (3) such that $u.\Sigma = \Sigma_0$. It follows that
\[
b \cdot \beta'(x) = b \cdot \beta'(u.x) = \beta(u.x) = \beta(x)
\]
for all $x \in \Sigma$ and hence in either case $\beta$ and $b\cdot \beta'$ coincide on $\Sigma$. This establishes the claim and thereby finishes the proof.
\end{proof}
Our next goal is to translate the proposition into group-theoretical terms in the specific setting of Convention \ref{SettingReductiveApprox}. For this we need to understand what it means for the restriction of a rescaled Busemann function to an apartment $\Sigma$ to be centered at a point of a given Weyl chamber $c$ at infinity.

We first consider the Archimedean case: In this case, $\Sigma$ is a Euclidean Coxeter complex with Coxeter group $\tilde{W}$ and spherical Coxeter group $W$. The reflection hyperplanes of $W$ decompose $V$ into simplicial cones, called the \emph{Weyl chambers} of $\Sigma$, and every such Weyl chamber $C$ defines a Weyl chamber $c \defeq \partial C \subseteq \partial \Sigma$ at infinity; the latter chambers define the given simplicial structure on $\partial \Sigma$. 

In view of Example \ref{AffineForms} every rescaled Busemann function $\beta$ on $\Sigma$ admits a linear part $\mathring{\beta}$ and a center $\mathring{\beta}/\|\mathring{\beta}\| \in \Sigma^1 \cong \partial \Sigma$, and this center is contained in at least one chamber $c$ at infinity. In order to characterize those rescaled Busemann functions $\beta$ whose centers lie inside $c$, we introduce the following notation:

The Weyl chamber $C$ determines a positive system $\Phi^+_C$ of roots such that $C \subseteq \{x \in V \mid \forall\, \alpha \in \Phi^+_C:\;\langle{\alpha,x}\rangle \ge 0\}$ and a corresponding set $\Pi_C$ of simple roots $\alpha \in \Phi^+_C$ such that $C \cap \{x \in V \mid \langle \alpha,x\rangle = 0\}$ is a maximal face of $C$. 
With this notation we then have the following characterization of rescaled Busemann functions on $\Sigma$ centered inside $c$.
\begin{lemma}\label{CenterBusemannRoots} Let $C \subset \Sigma$ be a Weyl chamber with boundary $c \defeq \partial C$ and let $\beta$ be a rescaled Busemann function on $\Sigma$. 
Then the following are equivalent:
\begin{enumerate}[(a)]
\item $\beta$ is centered at a point of $c$
\item $\mathring{\beta} \in C$
\item $\gen{\alpha,\mathring{\beta}} \ge 0$ for all $\alpha \in \Pi_C$.\qed
\end{enumerate}
\end{lemma}
In the Archimedean space the situation is similar: If we fix an apartment $\Sigma$ and an arbitrary basepoint $x \in \Sigma$, then the other apartments through $x$ intersect $\Sigma$ in a hyperplane arrangement. The connected components of the complement, the \emph{Weyl chambers at $x$}, then correspond to the Weyl chambers at infinity, and Lemma \ref{CenterBusemannRoots} holds mutatis mutandis then also in this context.

Now recall that $\mathcal K_s$ denotes the set of conjugates of $C_s$, which is in $G$-equivariant bijection with the orbit of $o_s$ in $X_s$. For the following corollary we will identify $\mathcal K_s$ with this orbit and also identify minimal $k$-parabolics with chambers in $\Delta_k \subseteq \Delta_{k_s}=\partial X_s$.

\begin{corollary}\label{cor:busemann}
In the situation of Convention \ref{SettingReductiveApprox} let $s \in V$ be a place (finite or not). Let $\bfP < \bfG$ a minimal $k$-parabolic and let $\chi \in \X[k](\bfP)$ be a character defined over $k$. Assume that we are given a map $\pi \colon \mathcal{K}_s \to \R_{>0}$ which satisfies
\begin{equation}\label{eq:parabolic_transformation}
\pi({}^gC_s) = \abs{\chi(g)}_s \cdot \pi(C_s) \quad \text{for all}\quad g\in \bfP(k_s)\text{.}
\end{equation}
Then $\log \pi \colon \mathcal{K}_s \to \R$ can be uniquely extended to a map $\beta \colon X_s \to \R$ that is affine on every apartment containing $\bfP$ in its boundary. This extension is a rescaled Busemann function centered at a point of the closure of $\bfP$ if and only if
\begin{equation}\label{eq:parabolic_direction}
\gen{\alpha,\chi} \ge 0 \quad \text{ for every simple root $\alpha$ of $\bfP$.}
\end{equation}
\end{corollary}

\begin{remark}
If $s$ is an infinite place then $\mathcal{K}_s$ is all of $X_s$ and $\log \pi$ and $\beta$ coincide. If $s$ is finite then $\mathcal{K}_s$ is still relatively dense in every apartment and hence there is at most one affine extension to that apartment.
\end{remark}

\begin{proof}
The plan is to invoke Proposition~\ref{prop:busemann_char}. In short, equation \eqref{eq:parabolic_transformation} will give Condition \eqref{eq:busemann_char_affine} when restricted to tori and Condition \eqref{eq:busemann_char_unipotent} when restricted to the unipotent radical, and the inequalities in \eqref{eq:parabolic_direction} will give condition \eqref{eq:busemann_char_direction}. We now spell out the details.

Let $\bfS \subset \bfQ$ be a maximal $k_s$-split torus inside a minimal $k_s$-parabolic contained in $\bfP$. Geometrically, $\bfQ$ corresponds to a chamber $c$ in $\Delta_{k_s} = \partial X_s$ and $\bfS$ corresponds to an apartment $\Sigma$ of $X_s$ containing $c$ in its boundary. We also fix a maximal $k$-split torus $T$ contained in $S$ and let $\alpha_1,\ldots,\alpha_r \in \X[k](\bfT)$ be the simple roots associated to $\bfP$. Then $\Sigma$ is a Euclidean space with underlying vector space $E_s = \coX[k_s](\bfS) \otimes \R$ equipped with a Weyl group invariant scalar product. Note that this means that the boundaries of $\Sigma$ and of $E_s$ can be identified.

In the building case, the action of $\bfS(k_s)$ on $E_s$ is given by the homomorphism $\omega \colon \bfS(k_s) \to \coX[k_s](\bfS)$ that is characterized by
\begin{equation}\label{eq:torus_action}
\log_{q_s}\abs{\psi(t)}_s = \langle \psi, \omega(t) \rangle
\end{equation}
for any $\psi \in \X[k_s](\bfS)$, where $q_s$ is the order of the residue field of $k_s$ and $\langle \cdot, \cdot \rangle \colon \X[k_s](\bfS) \times \coX[k_s](\bfS) \to \Z$ is the dual pairing, see \cite{KalethaPrasad} for details. In particular, we see from \eqref{eq:parabolic_direction} that $\chi(\omega(g)) = \log_{q_s} \pi({}^gC_s) - \log_{q_s} \pi(C_s)$. Thus the restriction of $\log_{q_s} \pi$ to $\Sigma \cap \calK_s$ extends uniquely to an affine form with linear part $\chi \in E_s^* = \X[k_s](\bfS) \otimes \R$, showing \eqref{eq:busemann_char_affine} of Proposition~\ref{prop:busemann_char}. (Changing the base of the logarithm only amounts to a scaling.)

In the symmetric space case, no extension is necessary: If $\mathfrak s$ denotes the Lie algebra of $\bfS(k_s)$, then $\iota_x: \mathfrak s \to E_s$, $X \mapsto \exp(X).x$ is an isomorphism of affine spaces for any $x \in E_s$. Since $\iota_x^* \log \pi$ is evidently affine on $\mathfrak s$, we deduce that \eqref{eq:busemann_char_affine} of Proposition~\ref{prop:busemann_char} is satisfied also in this case.

We now establish \eqref{eq:busemann_char_direction} in the Euclidean building case. For this it will be convenient to identify $E_s$ and $E_s^*$ via the scalar product. The space $E = \coX[k](\bfT) \otimes \R$ is canonically a subspace of $E_s$ in a way that is compatible with the inclusion $\Delta_k \subseteq \Delta_{k_s}$. The set $\mathfrak{C} = \{\rho \in E \mid \gen{\alpha, \rho} \ge 0 \text{ for all $k$-simple roots $\alpha$ of } \bfP\}$ is a Weyl chamber in $E$. Its boundary is  the (closed) chamber $\bfP$ in $\Delta_k$. Note that $\chi \in \coX[k](\bfT)$ by assumption so it is centered at a point in the boundary of $E$. We know from Lemma \ref{CenterBusemannRoots} that it is centered at $\bfP$ if and only if $\chi \in \mathfrak{C}$, which is precisely condition \eqref{eq:parabolic_direction}.
The set $\mathfrak{C}_s = \{\rho \in E_s \mid \gen{\alpha, \rho} \ge 0 \text{ for all $k_s$-simple roots $\alpha$ of } \bfQ\}$ is a Weyl chamber whose boundary in $\Delta_{k_s} \cong \partial X_s$ is the closed chamber $\bfQ$. This chamber contains the closed simplex $\bfP$ of $\Delta_{k_s}$ which in turn contains the chamber $\bfP$ of $\Delta_k$, see \cite[\S 6.10]{BorelTits65} and the discussion in Section~\ref{sec:bldg_infty}. We deduce that condition \eqref{eq:busemann_char_direction} of Proposition~\ref{prop:busemann_char} holds; in the symmetric space case one argues similarly, using Weyl chambers with a given tip $x$ instead of affine Weyl chambers.

Finally, if $u \in \bfQ(k_s) < \bfP(k_s)$ is unipotent then $\chi(u) = 1$ so \eqref{eq:parabolic_transformation} implies that $\pi({}^uC_s) = \pi(C_s)$ giving \eqref{eq:busemann_char_unipotent}.
\end{proof}

\subsection{Busemann functions vs.\ horofunctions}\label{sec:horo_busemann}

We continue to work in the setting of Convention \ref{SettingReductiveApprox}. If $\bfQ \in \calV(\Delta_k)$ is a maximal $k$-parabolic subgroup, considered a vertex of the rational building, then for every $s \in V$ there is a Busemann function $\beta_{\bfQ, s}$ of $X_s$ centered at the image of $\bfQ$ inside $\partial X_s = \Delta_{k_s}$. Our next goal is to encode all of these Busemann functions for varying values of $s$ into a single function.
\begin{definition}\label{def:pi_trafo}
We say that a function $\pi \colon \mathcal V(\Delta_k) \times \mathcal K \to \mathbb R_{>0}$ is
\begin{enumerate}[(i)]
\item a \emph{horofunction} if $\pi(\bfQ, {}^pK) = \abs{\chi_\bfQ(p)} \cdot \pi(\bfQ, K)$ for all $\bfQ \in \mathcal V(\Delta_k)$, $q\in \bfQ(\A)$ and $K \in \mathcal K$;\label{item:pi_trafo_parabolic}
\item \emph{normalized} if $\pi(\bfQ, C) = 1$ for every $\bfQ \in \mathcal V(\Delta_k)$\;
\item \emph{invariant} if $\pi({}^g\bfQ, {}^gK) = \pi(\bfQ,K)$ for all $g \in \bfG(k)$, $\bfQ \in \mathcal V(\Delta_k)$ and $K\in \mathcal K$.\label{item:pi_trafo_global}
\end{enumerate}
\end{definition}
We will see in the next subsection that normalized invariant horofunctions exist (Proposition \ref{ExNIHoro}), but before we prove this fact we would like to comment on their geometric meaning.
For the following proposition we identify every maximal $k$-parabolic $\bfQ \in \mathcal V(\Delta_k)$ with the image of the corresponding vertex in $\Delta_k$ under the diagonal embedding $\Delta_k \hookrightarrow \partial X_S$. We also consider $\mathcal K_S$ as a subset of $X_S$ in the usual way.
\begin{proposition} Let $\pi$ be a normalized horofunction. Then for every maximal $k$-parabolic $\bfQ \in \mathcal V(\Delta_k)$
There exists a unique family $(\beta_\bfQ)_{\bfQ \in \calV(\Delta_k)}$ of rescaled Busemann functions $\beta_{\bfQ, S} \colon X_S \to \R$ centered at the vertex $\bfQ$ such that
\begin{equation}\label{eq:horo_to_busemann}
 \beta_{\bfQ, S}|_{\calK_S} = \log \pi(\bfQ,\cdot)|_{\calK_S}\text{.}
\end{equation}
\end{proposition}
\begin{proof} Let $\bfQ \in \calV(\Delta_k)$ be a maximal $k$-parabolic and let $\bfP \in \calC(\Delta_k)$ be a minimal $k$-parabolic contained in $\bfQ$ so that $\bfQ = \bfP_i$ for some $i$. We denote by $\chi \defeq \chi^{\bfQ}$ the canonical character of $\bfQ$; Lemma~\ref{lem:root_weight} implies that $\gen{\alpha,\chi} \ge 0$ for every simple root $\alpha$ of $\bfP$. We now consider the usual embedding $\iota_s: \calK_s \to \calK$ and $i_s: \bfG(k_s) \hookrightarrow \bfG(\A)$ given respectively by $K_s \mapsto K_s \times \prod_{s' \neq s} C_{s'}$ and  $g_s \mapsto (\dots, e, g_s, e, \dots)$ and define
\[
\pi_s: \calK_s \to \R_{>0}, \quad K_s \mapsto  \pi(\bfQ, \iota_s(K_s)).
\]
Since $\pi$ is a horofunction we then have for all $g\in \bfP(k_s) \subset \bfQ(k_s)$ the identity
\[
\pi_s({}^{g_s}C_s) =  \pi(\bfQ, {}^{i_s(g_s)}\iota_s(K_s)) =  \abs{\chi_\A(i_s(g_s))} \cdot   \pi(\bfQ, \iota_s(K_s)) =  |\chi(g_s)|_s\cdot \pi_s(C_s).
\]
It thus follows from Corollary~\ref{cor:busemann} that $\log \pi_s$ can be uniquely extended to a rescaled Busemann function $\beta_{\bfQ,s}: X_s \to \R$ which is centered at a point in the closed chamber corresponding to $\bfP$. Since this holds true for \emph{all} chambers containing the vertex corresponding to $\bfQ$, it is actually centered at the vertex corresponding to $\bfQ$.

Now we consider the places of $S$ simultaneously. We define
\begin{align*}
\beta_{\bfQ,S} \colon X_S &\to \R\\
(x_s)_s &\mapsto \sum_{s \in S} \beta_{\bfQ,s}(x_s)\text{.}
\end{align*}

Note that since $\pi$ is normalized we have for all $g \in \bfG_S$
\[
\exp \beta_{\bfQ,S}(g.o_S) = \prod_{s \in V} \pi(\bfQ,{}^{g_s}C) = \prod_{s \in S} \abs{\chi_{\bfQ}(g_s)}_s \pi(\bfQ,C) = \abs{\chi_{\bfQ}(g)},
\]
but also
\[
\pi(\bfQ, {}^g C) =  \abs{\chi_{\bfQ}(g)} \pi(\bfQ, C) =  \abs{\chi_{\bfQ}(g)},
\]
and hence the restrictions of $\exp \beta_\bfQ$ and of $\pi(\bfQ,\cdot)$ to $\calK_S$ coincide.

Uniqueness of the family $(\beta_\bfQ)_{\bfQ \in \calV(\Delta_k)}$ follows from the fact that $\calK_S$ intersects each apartment $A$ of $X_S$ in a relatively dense set, and that two affine functions agreeing on a relative dense subset of $A$ agree on all of $A$.
\end{proof}
In other words, a normalized horofunction gives rise to a family of rescaled Busemann function, centered at the vertices of the rational building, for each of the \cato{} spaces $X_S$. In fact, these families are compatible in the sense that for $S \subset T$ and $X_S$ is embedded into $X_T$ via $(x_s)_{s\in S} \mapsto (x_s)_{s\in S} \cup (o_t)_{t \in T\setminus S}$, then we have $\beta_{\bfQ,T}|_{X_S} = \beta_{\bfQ,S}$, and hence there exists a function $\beta_{\bfQ}: X = \lim X_S \to \R$ such that $\beta_{\bfQ}|_{X_S} = \beta_{\bfQ, S}$. The function $\beta_{\bfQ}$ can then be seen as a rescaled Busemann function for the infinite-dimensional space $X$ (centered at $\bfQ$), and the normalized horofunction $\pi$ carries exactly the same information as the family $(\beta_{\bfQ})_{\bfQ \in \Delta_k}$ of rescaled Busemann functions on $X$.

Classically, geometric reduction theory of ($S$-)arithmetic groups is described using Busemann functions. Because of the above equivalence, we may as well work with normalized horofunctions, which carry the same information. 

\subsection{Invariant horofunctions}

We continue to work in the setting of Convention \ref{SettingReductiveApprox}. We have seen in the previous subsection that (normalized) horofunctions parametrize families of rescaled Busemann functions on the spaces $X_S$. Our next goal is to show that such horofunctions always exist and can even be chosen to be invariant; in fact, this invariance determines them almost completely.
\begin{proposition}[Normalized invariant horofunctions]\label{ExNIHoro}
\begin{enumerate}[(i)]
\item For every reductive $k$-group $\bfG$ there exists a normalized invariant horofunction $\pi\colon \mathcal V(\Delta_k) \times \mathcal K \to \R_{>0}$.
\item If $\pi$ and $\pi'$ are two invariant horofunctions (not necessarily normalized) then there exists $C = C(\pi, \pi') > 0$ such that 
\begin{equation}\label{HorofunctionsEquivalent}\pushQED{\qed}
1/C \cdot \pi(\bfQ,K) < \pi'(\bfQ,K) < C \cdot \pi(\bfQ,K) \quad (\bfQ \in \mathcal V(\Delta_k), K \in\mathcal K).
\end{equation}
\end{enumerate}
\end{proposition}

The proof of the proposition is based on two facts concerning adelic points of maximal $k$-parabolic subgroups of $\bfG$.

Firstly, we recall from Lemma \ref{lem:iwasawa} that for every maximal $k$-parabolic subgroup $\bfQ$ the double coset space $\bfQ(\A) \backslash \bfG(\A)/C$ is finite.

Secondly, for every such $\bfQ$ the canonical character $\chi^{\bfQ}$ defines a character $\chi^{\bfQ}_{\A}\colon \bfQ(\A) \to \mathbb J = \GL_1(\bf A)$, and composing with the idele norm we obtain a multiplicative function
\[
\abs{\chi^{\bfQ}_{\A}(\cdot)} \colon \bfQ(\A) \to (0, \infty)
\]
with the following three properties:
\begin{enumerate}
\item If $g \in \bfG(k)$ then $\chi_{\A}^{\bf Q}(q) = \chi_\A^{{}^g\bf Q}(gqg^{-1})$ for all $q,g \in \bfQ(\A)$. 
\item $\abs{\chi_\A^\bfQ(\cdot)}$ is bounded, hence trivial, on compact subgroups of $\bfQ(\A)$.
\item  $\abs{\chi_\A^\bfQ(\cdot)}$ is trivial on $\bfQ(k)$ by Proposition \ref{prop:product_formula}.
\end{enumerate}
We recall that if $\bfP$ is a minimal $k$-parabolic then $\bfP_i$ denotes the unique maximal $k$-parabolic of type $i$ containing $\bfP$.

\begin{lemma} \label{InvHFFormula}
Let $\bfP \in \calC(\Delta_k)$ and let $\xi_{i1}, \dots, \xi_{im_i}$ be representatives for the finitely many double cosets in $\bfP_i(\A) \backslash \bfG(\A)/C$. Then for all $\pi_{ij} \in (0, \infty)$ there exists a unique invariant horofunction $\pi$ such that
\begin{equation}\label{piij}
\pi({}^g\bfP_i, {}^{gp\xi_{ij}} C) =  \abs{\chi_{\bfP_i}(p)} \cdot \pi_{ij} \quad \text{for all } g\in \bfG(k) \text{ and } p \in \bfP_i(\A),
\end{equation}
and conversely every invariant horofunction is of this form.
\end{lemma}
\begin{proof} Firstly, it is immediate from the definition that every invariant horofunction $\pi$ must satisfy \eqref{piij} with $\pi_{ij} \defeq \pi(\bfP_i,{}^{\xi_{ij}} C)$.

Secondly, since $\bfG(k)$ acts transitively on $\mathcal V_i(\Delta_k)$ for all $i$ and every $(\bfP_i(\A), C)$-double coset is represented by some $\xi_{ij}$ it is clear that there is at most one function $\pi$ satisfying \eqref{piij}. 

Thirdly, we show that \eqref{piij} gives rise to a well-defined function $\pi$. Assume that $({}^g\bfP_i, {}^{gp\xi_{ij}} C) = ({}^{g'}\bfP_{i'}, {}^{g'p'\xi_{i'j'}} C)$ for some $g, g' \in \bfG(k)$, $p \in \bfP_i(\A)$ and $p' \in  \bfP_{i'}(\A)$. Comparing first coordinates we deduce that $i = i'$. Since $\bfP_i(k)$ is self-normalizing in $\bfG(k)$, we then also deduce that $g^{-1}g' \in \bfP_i(k)$, and in particular $|\chi_\A^{\bfP_i}(g^{-1}g')| = 1$. Comparing second coordinates now yields
\[
{}^{p\xi_{ij}} C =  {}^{p''\xi_{ij'}} C, \quad \text{where }p'' \defeq g^{-1}g'p' \in \bfP_i(\A),
\]
which implies $\bfP_i(\A)\xi_{ij}C = \bfP_i(\A)\xi_{ij'}C$ and hence $j = j'$. Since ${}^{\xi_{ij}}C$ is self-normalizing we also deduce that $p^{-1}p'' \in {}^{\xi_{ij}}C \cap \bfP_i(\A)$, and since the latter is compact this implies $\abs{\chi_\A^{\bfP_i}(p^{-1}p'')}=1$. We conclude that
\[
 \abs{\chi_\A^{\bfP_i}(p')} \cdot \pi_{i'j'} =  \abs{\chi_\A^{\bfP_i}((g^{-1}g')^{-1}p(p^{-1}p''))} \cdot \pi_{ij} =  \abs{\chi_\A^{\bfP_i}(p)} \cdot \pi_{ij},
\]
which shows that \eqref{piij} is well-defined. 

Finally, we show that $\pi$ is an invariant horofunction; invariance is immediate from the formula. To see that it is also a horofunction, let $\bfQ \in \mathcal V(\Delta_k)$ and $q \in \bfQ(\bf A)$. Then there exist $g \in \bfG(k)$, $i \in \{1, \dots, r\}$ and $p \in \bfP_i(\A)$ such that $\bfQ = {}^g\bfP_i$ and $q = gpg^{-1}$.  Given $K \in \mathcal K$ write $K = {}^{gp_1\xi_{ij}}C$ for some  $p_1 \in \bfP_i(\A)$ and $\xi_{ij}$ as above so that
\[
\pi(\bfQ, K) = \pi({}^g\bfP_i, {}^{gp_1\xi_{ij}}C) = \abs{\chi_\A^{\bfP_i}(p_1)} \cdot \pi_{ij}.
\]
Since ${}^qK ={}^{(gpg^{-1})gp_1\xi_{ij}}C = {}^{gpp_1\xi_{ij}}C$ and $\chi_\A^{\bfP_i}(p) = \chi_\A^{\bfQ}(q)$ we obtain
\[
\pi(\bfQ, {}^{q}K) = \pi({}^g\bfP_i, {}^{gpp_1\xi_{ij}}C) = \abs{\chi_\A^{\bfP_i}(p)}\abs{\chi_\A^{\bfP_i}(p_1)} \cdot \pi_{ij} = \abs{\chi_\A^{\bfQ}(q)} \cdot \pi(\bfQ, K).
\]
This shows that $\pi$ is a horofunction and finishes the proof.
\end{proof}
\begin{proof}[Proof of Proposition \ref{ExNIHoro}] (i) Choose $\pi_{ij} \defeq 1$ in \eqref{piij} to obtain a normalized invariant horofunction.

(ii) We may assume that $\bfQ = {}^g\bfP_i$ and $K = {}^{gp\xi_{ij}} C$ for some choice of $i,j$. Then \eqref{piij} implies that
\[
\frac{\pi'(\bfQ, K)}{\pi(\bfQ, K)} = \frac{\pi'_{ij}}{\pi_{ij}}.
\]
It thus suffices to choose $C$ to be larger than $\pi'_{ij}/\pi_{ij}$ and $\pi_{ij}/\pi'_{ij}$ for all possible choices of $i$ and $j$.
\end{proof}


While normalized horofunctions give rise to families of rescaled Busemann functions on $X_S$, normalized \emph{invariant} horofunctions correspond to families of rescaled Busemann functions on $X_S$, which are moreover quasi-invariant under the $S$-arithmetic approximate subgroup $\bfG(\calO_S)$.
\begin{theorem}[Quasi-invariance of Busemann functions]\label{thm:busemann_family}
Let $\pi$ be a normalized horofunction and let $(\beta_\bfQ)_{\bfQ \in \calV(\Delta_k)}$ be the associated family of rescaled Busemann functions on $X_S$. It $\pi$ is invariant, then there is constant $b \geq 0$ such that
\[
\abs{\beta_{{}^g\bfQ}(\lambda.x) - \beta_\bfQ(x)} \le b
\]
for all $\lambda \in \bfG(\calO_S)$ and all $x \in X_S$.
\end{theorem}
Here the action of $\bfG(\calO_S)$ is through the embedding $\bfG(\calO_S) \hookrightarrow \bfG(k)$, whereas the action on $X_S$ is through the inclusion $\bfG(\calO_S) \hookrightarrow \bfG(k_s)$, see the discussion in Remark \ref{NonCommutingDiagram}. We will see that in the case where $S$ contains all infinite places (and hence $\bfG(\calO_S)$ is a group), the constant $b$ can be chosen to be $0$, but this is no longer the case in general. The remainder of this subsection is devoted to the proof of Theorem \ref{thm:busemann_family}. Since $G$ acts transitively on $\calK_S$, we can reformulate Theorem \ref{thm:busemann_family} in terms of normalized horofunctions as follows:
\begin{proposition}[Quasi-invariance of invariant horofunctions]\label{prop:horo_quasi-inv} There exists $B>0$ such that
for all $g \in G$, $\lambda \in \bfG(\calO_S)$ and $\bfQ \in \mathcal V(\Delta_k)$ we have
\[
\frac 1 B \cdot \pi(\bfQ, {}^{(g, 1_H)} C) \leq {\pi({}^\lambda\bfQ, {}^{(\pi_G(\lambda)g, 1_H)}C)} \leq B \cdot \pi(\bfQ, {}^{(g, 1_H)} C)\text{.}
\]
Moreover, if $S \supset V\inf$, then we may choose $B=1$, i.e.\
\begin{equation*}\label{eq:S-invariance}
\pi({}^\lambda \bfQ,{}^{(\pi_G(\lambda), 1_H)}K) = \pi(\bfQ,K) \quad \text{ for all } \lambda \in \bfG(\calO_S), \bfQ \in \calV(\Delta_k) \text{ and } K \in \calK_S\text{.}
\end{equation*}
\end{proposition}
Here the constant $B$ is the exponential of the constant $b$ from Theorem \ref{thm:busemann_family} and $ {}^{(g, 1_H)} C \in \calK_S$ corresponds to a point $x \in X_S$. Let us first deal with the case where $S \supset V\inf$. If we write $\lambda \in \bfG(\calO_S)$ as $\lambda = ((\lambda)_{s \in S}, (\lambda)_{v \in V \setminus S}) \in G \times H$, then by definition of $\bfG(\calO_S)$ we have $\lambda \in W_v = C_v$ for all $v \in V \setminus S$ and hence ${}^{(g,\pi_H(\lambda)^{-1})}C = {}^{(g,1_H)}C$ for all $g \in G$. By invariance of $\pi$ this implies 
\[
{\pi({}^\lambda\bfQ, {}^{(\pi_G(\lambda)g, 1_H)}C)} = {\pi(\bfQ, {}^{(g, \pi_H(\lambda)^{-1})}C)} ={\pi(\bfQ, {}^{(g, 1_H)}C)} \text{,}
\]
thus establishing invariance under $\bfG(\calO_S)$ in the case where $S \supset V\inf$. 

In the general case we will no longer have ${}^{(g,\pi_H(\lambda)^{-1})}C = {}^{(g,1_H)}C$ for all $\lambda \in \bfG(\calO_S)$ and $g \in G$, since $W_v C_v \neq C_v$ for infinite places $v \in V\inf \setminus S$. In this case, the proof of Proposition \ref{prop:horo_quasi-inv} relies on the following compactness result. If $\bfQ \in \mathcal V_i(\Delta_k)$ is a maximal parabolic subgroup of type $i$ and $s \in V^{\inf}$ then $\bfQ(k_s)$ is a parabolic subgroup of $\bfG(k_s)$ and we denote by $\mathcal F^i_{k_s}$ the conjugacy class of this parabolic subgroup in $\bfG(k_s)$. Then $\mathcal F_{k_s}$ depends only on $\bfG$, $k_s$ and $i$, but not on $\bfQ$, and is compact. This implies:
\begin{lemma}\label{CompactnessArgument} Let $s \in V^{\inf}$ and let
\[
\mathcal Q^i \defeq \{(q, Q) \in \bfG(k_s) \times \mathcal F^i_{k_s}\ \mid q \in Q\} \subseteq \bfG(k_s) \times \mathcal F^i_{k_s}.
\]
Then the map $\chi\colon \mathcal Q^i \to \R^\times$, $(q, \bfQ) \mapsto \chi^{\bfQ}(q)$ is continuous. In particular, if $L_s \subseteq \bfG(k_s)$ is compact then there exists $C>0$ such that
\[
\frac{1}{C} \leq |\chi^{\bfQ}(q)| \leq C \quad \text{ for all } \bfQ \in \calV_i(\Delta_k) \text{ and }p \in\bfQ(k_s) \cap L_{k_s}.
\]
\end{lemma}

\begin{proof} If we fix $Q_o \in \mathcal F^i_{k_s}$ then we have a quotient map
\[
p\colon Q_o \times \bfG(k_s) \to \mathcal Q^i, \quad (q_o, g) \mapsto ({}^g q_o, {}^gQ_o).
\]
Since $p^*\chi(q_o, g) = \chi_{{}^gQ_o}({}^gq_o) = \chi_{Q_o}(q_o)$, the map $p^*\chi$ is continuous, and hence $\chi$ is continuous as well. This implies in particular that $\chi$ is bounded on subsets of the compact set $\mathcal Q^i \cap (L_s\times \mathcal F_{k_s})$.
\end{proof}
This implies the following quasi-invariance at the infinite places; the key point of the corollary is that the bound $B$ depends only on the sets $W_{s_o}$, but neither on $w$ nor on the parabolic $\bfQ$.
\begin{corollary}\label{cor_quasiinv1} Let $V_0 \subseteq V^{\inf}$ and for $s_o \in V_0$ let $W_{s_o} \subseteq \bfG(k_{s_o})$ be relatively compact. Then there exists $B>0$ with the following property: If $w_s \in Z_{\bfG(k_s)}(C_s)$ for all $s \in V \setminus V_0$ and $g_{s_o}\in Z(\bfG(k_{s_o}))$ and $w_{s_o}\in W_{s_o}$ for all $s_o \in V_0$ then $g = (g_s)_{s \in S}$ and $w = (w_s)_{s \in S}$ satisfy
\[
\frac{1}{B} \le \frac{\pi(\bfQ,{}^{wg}C)}{\pi(\bfQ,{}^gC)} \le B \quad \text{ for all }\bfQ \in \calV(\Delta_k).
\]
\end{corollary}
\begin{proof} Let us fix a parabolic $\bfQ \in \calV(\Delta_k)$. Since $V^{\inf}$ is finite, we may assume that $V_0 = \{s_o\}$ is a singleton. Since $k_{s_o}$ is Archimedean and thus admits an Iwasawa decomposition, we can find for every $g \in \bfG(k_{s_o})$ some $q_{s_o} \in \bfQ(k_{s_o})$ such that ${}^g C_{s_o} = {}^{q_{s_o}}C_{s_o}$. In particular,
${}^{w_{s_o}} C_{s_o} = {}^{q_{s_o}}C_{s_o}$. Since $w_{s_o}$ and $q_{s_o}$ both commute with $g_{s_o}$ this implies that
\[
{}^{w_{s_o}g_{s_o}}C_{s_o} = {}^{q_{s_o}g_{s_o}}C_{s_o}, \quad \text{and hence} \quad {}^{wg}C = {}^{(e, \dots,e, q_{s_o}, e, \dots)}({}^gC).
\]
Since $\pi$ is a horofunction this implies that
\[
\frac{\pi(\bfQ,{}^{wg}C)}{\pi(\bfQ,{}^gC)} = |\chi_\A^{\bfQ}(e, \dots, e, q_{s_o},e, \dots)| = |\chi^{\bfQ}(q_{s_o})|\text.
\]
Since $w_sC_s = q_sC_s$, the point $q_s$ is contained in the compact subset $L_{s_o} \defeq \overline{W}_{s_o}C_{s_o}$ of $\bfG(k_{s_o})$, hence the proposition follows from Lemma \ref{CompactnessArgument} (and the fact that there are only finitely many types of vertices).
\end{proof}
The desired proposition is now immediate:
\begin{proof}[Proof of Proposition \ref{prop:horo_quasi-inv}]  Since $\pi$ is invariant we have
\[
\pi({}^\lambda\bfQ, {}^{(\pi_G(\lambda)g, 1_H)}C) = \pi(\bfQ, {}^{(g,\pi_H(\lambda)^{-1})}C) = \pi(\bfQ, {}^{(e_G, \pi_H(\lambda)^{-1})(g, 1_H)} \bfQ).
\]
Then the proposition follows from Corollary~\ref{cor_quasiinv1}.
\end{proof}
At this point we have finished the proof of Theorem \ref{thm:busemann_family}.

\subsection{Dual invariant horofunctions}

We continue to work in the setting of Convention \ref{SettingReductiveApprox}. In view of Proposition \ref{ExNIHoro} we may fix an invariant horofunction $\pi\colon \calV(\Delta_k) \times \calK \to \R_{>0}$. If $\pi$ is normalized, then the functions
\[
\pi_j\colon \mathcal C(\Delta_k) \times \mathcal K \to \R_{>0}, \quad (\bfP, K) \mapsto \pi(\bfP_j, K).
\]
are (restrictions of) families of exponentials of rescaled Busemann functions on $X$, and for every chamber $\bfP \in \mathcal C(\Delta_k)$ the rescaled Busemann functions corresponding to $\pi_1(\bfP, \cdot), \dots, \pi_r(\bfP, \cdot)$ are centered at the vertices of the respective chamber. (The index $i$ of $\pi_i$ indicates the type of the vertex.) If $\pi$ is invariant, but not normalized, then there is no immediate such geometric interpretation, but in view of Proposition \ref{ExNIHoro} the difference to the normalized case is negligible. In either case, as long as $\pi$ is invariant, the functions $\pi_i$ transform as follows: For all $p \in \bfP_j(\A)$ (and in particular for all $p \in \bfP(\A)$) and all $g \in \bfG(k)$ we have
\begin{equation}\label{pij}
    \pi_j(\bfP, {}^pK) =  |\chi^{\bfP_j}(p)|\cdot \pi_j(\bfP, K) \; \text{and} \;  \pi_j({}^g\bfP, {}^gK) = \pi_j(\bfP,K)\text{.}
\end{equation}
For short, the functions $\pi_j(\bfP, \cdot)$ transform according to the canonical characters of the vertices of $\bfP$. We would now like to define a dual set of functions on $\mathcal K$ which instead transform according to the \emph{roots} of $\bfP$. We recall from \eqref{RootsVsCharacters} that the roots of $\bfP$ are related to the canonical characters of the vertices $\bfP_j$ by the formulas
\[
\alpha^\bfP_i = \sum_{j=1}^r c_{ij}\chi^{\bfP_j} \quad \text{and} \quad \chi^{\bfP_j} = \sum_{i=1}^r n_{ji} \alpha^\bfP_i\text{.}
\]
This motivates the following definition:
\begin{definition}\label{def:dual_invariant} Given an invariant horofunction $\pi\colon \mathcal V(\Delta_k) \times \mathcal K\to (0, \infty)$, the corresponding \emph{dual invariant horofunctions} $\nu_1, \dots, \nu_r$ are defined as
\begin{equation}\label{eq:horofunction_to_dual}
\nu_i\colon \mathcal C(\Delta) \times \mathcal K \to (0, \infty), \quad \nu_i \defeq \prod_{j=1}^r \pi_j^{c_{ij}}\text{.}
\end{equation}
\end{definition}
It is immediate from the above formulas that, dually,
where the $c_{ij}$ are defined by \eqref{RootsVsCharacters}.  It follows that
\begin{equation}\label{eq:dual_to_horofunction}
\pi_j = \prod_{i= 1}^r \nu_i^{n_{ji}}\text{.}
\end{equation}
From \eqref{pij} and \eqref{RootsVsCharacters} we obtain the following transformation behavior of  the dual invariant horofunctions:
\begin{lemma}\label{lem:nu_trafo}
For all $\bfP \in \mathcal C(\Delta)$, $K \in \mathcal K$ and $i \in \{1,\dots, r\}$ the following hold.
\begin{enumerate}[(i)]
\item $\nu_i(\bfP,{}^gK) = \abs{\alpha_i^{\bfP}(g)} \cdot \nu_i(\bfP,K)$ for $g  \in \bfP(\A)$.\label{item:nu_trafo_parabolic}
\item $\nu_i({}^g\bfP,{}^gK) = \nu_i(\bfP,K)$ for $g \in \bfG(k)$.\label{item:nu_trafo_global}\qed
\end{enumerate}
\end{lemma}
Note that \eqref{item:nu_trafo_parabolic} only holds for elements of the \emph{minimal} parabolic $\bfP(\A)$ (rather than for elements of $\bfP_i(\A)$), since we need to apply \eqref{pij} for all $j\in \{1, \dots, r\}$ simultaneously. If we denote, as before, by  $(\beta_\bfQ)_{\bfQ \in \calV(\Delta_k)}$ the unique family of rescaled Busemann functions such that
\[\log \beta_\bfQ|_{\calK_S} = \pi(\bfQ,\cdot)|_{\calK_S}\text{,}\]
then we obtain from \eqref{eq:horofunction_to_dual} that
\begin{equation}\label{Defmus}
\log  \nu_i(\bfP, \cdot)|_{\calK_S} = \mu_{i}^{\bfP}|_{\calK_S}, \quad \text{where} \quad \mu^{\bfP}_i =  \sum_{j=1}^r c_{ij} \beta_{\bfP_j}.
\end{equation}
The functions $\mu^{\bfP}_i$ are not rescaled Busemann functions, but as linear combinations of rescaled Busemann functions centered at the vertices of $\bfP$, they are at least affine on apartments that contain the chamber $\bfP$ in their boundary, but Condition (2) of Proposition \ref{prop:busemann_char} is violated. (As before we include maximal flats in symmetric spaces when speaking about apartments).


\section{Adelic reduction theory}\label{sec:adelic_reduction}

\subsection{The fundamental theorems of adelic reduction theory}
The purpose of classical reduction theory is to exhibit a weak form of fundamental domain for an $S$-arithmetic group $\bfG(\calO_S)$ on its ambient group $\bfG_S$. 
Since $\bfG_S$ acts properly and cocompactly on the associated \cato{} space $X_S$, this amounts to the same as providing a fundamental domain, in the same weak sense, for $\bfG(\calO_S)$ acting on $X_S$. Here we would like to develop a similar version of reduction theorem for $S$-arithmetic approximate groups. Instead of working with $\bfG(\calO_S) \subset \bfG_S$ directly, it will be more convenient for us to first develop a reduction theory for $\bfG(k) \subset \bfG(\A)$ in the adelic setting and then deduce the $S$-arithmetic case by a descent based on Proposition~\ref{prop:descent}. This approach has two main advantages: Firstly, the adelic reduction theory for $\bfG(k) \subset \bfG(\A)$ has already been developed in the generality needed for our purposes, and secondly we can work with the group $\bfG(k)$ rather than the approximate group $\bfG(\calO_S)$. 

Throughout this section we will work in the following setting.
\begin{convention}\label{ConventionReductiveGroups}
\begin{enumerate}
\item $k$, $V$, $\bfG$, $r$, $C$, $\calK$ are as in Convention~\ref{SettingReductiveApprox} and the subsequent remarks; in particular, $\bfG$ is reductive;
\item $\pi\colon \mathcal V(\Delta_k) \times \mathcal K \to \mathbb R_{>0}$ denotes a fixed choice of invariant horofunction (in the sense of Definition \ref{def:pi_trafo}) and $\nu_1, \dots, \nu_r\colon \mathcal C(\Delta_k) \times \mathcal K \to \R_{>0}$ are the associated dual invariant horofunctions (Definition \ref{def:dual_invariant});
\item $\bfG$ is assumed to be $k$-isotropic.
\end{enumerate}
\end{convention}
\begin{remark}
From a geometric point of view, one should choose the invariant horofunctions in (2) to be normalized, so that it corresponds to a family of rescaled Busemann functions. However, in view of Proposition \ref{ExNIHoro}, all invariant horofunctions are equivalent for our purposes, and hence the precise choice does not matter. The assumption that $\bfG$ be isotropic is actually not necessary at all, but the anisotropic case is of little interest. Indeed, if $\bfG$ is $k$-anisotropic then $\Delta$ is empty and $\bfG(k)$ is a uniform lattice in $\bfG(\A)^\circ$. It is a basic exercise in formal logic to interpret the following statements in such a way that they become tautologically true in this case ($\bfG$ itself is the only minimal $k$-parabolic \ldots). Since it does not produce any new results, we leave this exercise to the interested reader.
\end{remark}

Given real numbers $c_1<c_2$ we introduce, for every minimal parabolic $\bfP$ the subsets

\begin{align}
\Omega^\bfP_{c_1} &\defeq \{ g \in \bfG(\A) \mid \forall i\ c_1 \le \nu_i(\bfP,{}^gC)\}\label{OmegaPc1}
\\
\Omega^\bfP_{c_1,c_2} &\defeq \{ g \in \bfG(\A) \mid \forall i\ c_1 \le \nu_i(\bfP,{}^gC) \le c_2\}\label{OmegaPc1c2}\text{.}
\end{align}

The following three theorems, which hold in arbitrary characteristic, constitute adelic reduction theory.

\begin{theorem}[First fundamental theorem]\label{thm:lower}
There exists a $c_1 > 0$ such that for every $g \in \bfG(\A)$ there is a minimal parabolic $\bfP$ with $g \in \Omega^\bfP_{c_1}$. 
\end{theorem}

Any $c_1$ satisfying the conclusion of Theorem~\ref{thm:lower} will be called a \emph{lower reduction bound}. A corresponding constant $c_2$ as in the following theorem will be an \emph{upper reduction bound}.

\begin{theorem}[Second fundamental theorem]\label{thm:uniqueness}
For every lower reduction bound $c_1 > 0$ there exists a $c_2 > 0$ such that if $g \in \Omega_{c_1}^\bfP \cap \Omega_{c_1}^{\bfP'}$ and $\nu_i(\bfP,{}^gC) \ge c_2$ then $\bfP_i = \bfP'_i$.
\end{theorem}

The formulation of Mahler's compactness criterion refers to the group $\bfG(\A)^\circ$ which was defined in \eqref{eq:ring}.

\begin{theorem}[Mahler's compactness criterion]\label{thm:mahler}
Let $c_1 > 0$ be a lower reduction bound.
A subset $A \subseteq \bfG(\A)^\circ$ is relatively compact modulo $\bfG(k)$ if and only if there is a $c' > c_1$ such that $A \subseteq \bigcup_{\bfP} \Omega_{c_1,c'}^\bfP$.
\end{theorem}
While all three theorems hold in arbitrary characteristic, there is so far no characteristic-free proof known. Adelic reduction theory in positive characteristic was developed by Harder \cite{harder67,harder68,harder69,harder77}; the above geometric formulation of the three main theorems was given by Bux, K\"ohl and Witzel in \cite[Section~11]{bux13}, hence there is nothing more to say for us.

Adelic reduction theory in characteristic $0$ was developed by Godement \cite{godement64}. In the next subsection we derive the above three main theorems from Godement's theory, following \cite{niesdroy15}. Our approach is in close analogy with \cite[Section~11]{bux13}.

\subsection{Proof of the fundamental theorems in characteristic zero}

From now on we assume that $k$ is of characteristic $0$, i.e.\ a global number field. The sets $\Omega^{\bfP}_c$ and $\Omega^{\bfP}_{c_1, c_2}$ are defined as in the previous section. Note that by Lemma \ref{lem:nu_trafo}.(ii) we have for every $\gamma \in \bfG(k)$ the identities 
\begin{equation}\label{OmegaEquivariance}
\gamma^{-1} \Omega^\bfP_{c_1} = \Omega^{{}^\gamma \bfP}_{c_1} \quad \text{and} \quad \gamma^{-1} \Omega^\bfP_{c_1,c_2} = \Omega^{{}^\gamma \bfP}_{c_1,c_2}
\end{equation} 
In order to deduce the three main theorems from \cite{godement64} we need to briefly introduce the setup of this article. Following Godement (albeit with slightly different notation) we fix a minimal parabolic $\PGod$ and a maximal $k$-split torus $\TGod \subseteq \PGod$ and then choose a compact subset $M \subseteq \bfG(A)$ and an open relatively compact subset $F \subseteq \PGod(\A)^\circ$ such that $\bfG(\A) = M \PGod(\A)$ and $\PGod(\A)^\circ = F \PGod(k)$. Godement shows that $\PGod(\A) = \TGod_{V\inf}^+ \PGod(\A)^\circ$, where $\TGod_{V\inf}^+$ denotes the identity-component of $\TGod_{V\inf}$. 

We now denote by $Z\defeq\{\zeta_1,\ldots,\zeta_m\}$ a set of representatives for the finitely many double cosets in $\PGod(\A)\backslash \bfG(\A)/C$. Then $\bfG(\A) = C Z^{-1} \PGod(\A)$, hence we may (and will) decide to choose $M\defeq CZ^{-1}$. Combining all of these decompositions we see that
\begin{equation}\label{GADecompositions}
    \bfG(\A) 
    = C Z^{-1}\TGod_{V\inf}^+F \PGod(k) 
    = \PGod(k)F^{-1}\TGod_{V\inf}^+Z C.
\end{equation}
In particular, every $g \in \bfG(\A)$ can be written as
\begin{equation}\label{gDecomp}
  g = \gamma f^{-1}t\zeta_jx \quad \text{with} \; \gamma \in \PGod(k), f \in F, t \in \TGod_{V\inf}^+, 1\leq j \leq m, x\in C.
\end{equation}
In \cite[Section 10]{godement64}, Godement constructs Siegel sets for the group $\bfG(k)$ in $\bfG(\A)$. For every $d>0$ he first defines a subset
\[
T_\infty^+(d) = \{t \in \TGod_{V\inf}^+ \mid \abs{\alpha_i(t)} < d \text{ for all }i\in \{1, \dots, r\}\} \subseteq \TGod_{V\inf}^+;
\]
here we will also need the similarly defined subsets
\[
T_\infty^+(d_2, d_1) = \{t \in \TGod_{V\inf}^+ \mid d_2 < \abs{\alpha_i(t)} < d_1 \text{ for all }i\in \{1, \dots, r\}\}.
\]
He then defines for every $d>0$ the corresponding \emph{Siegel set} $CZ^{-1}T_\infty^+(d) F$. We are going to work with the $\PGod(k)$-saturations of these Siegel sets and their truncated versions, i.e.
\[
\OmegaGod_d^{\PGod} = CZ^{-1} T_\infty^+(d) F \PGod(k) \quad \text{and} \quad \OmegaGod_{d',d}^{\PGod}\defeq CZ^{-1} T_\infty^+(d',d) F \PGod(k).
\]
By the following lemma the filtrations by these sets are essentially equivalent to the filtrations by our sets $\Omega_c^\bfP$ and $\Omega_{c,c'}^{\bfP}$, except for a difference in parametrization and passing to inverses. In fact, the parametrizations are opposite in the sense that increasing the parameter $c$ increases the sets  $\OmegaGod_c^{\PGod}$ and decreases our sets $\Omega_c^\bfP$.
\begin{lemma}\label{lem:godement_translation}
For every $c>0$ and every $d>0$ there is a $c'>0$ and a $d'>0$ such that
\[
\Omega_c^{\PGod} \subseteq (\OmegaGod_{d'}^{\PGod})^{-1}\quad\text{and}\quad\OmegaGod_{d}^{\PGod}\subseteq (\Omega_{c'}^{\PGod})^{-1}.
\]
\end{lemma}
To establish the lemma we need to estimate each dual invariant horofunctions $\nu_i$ in terms of the corresponding simple root $\alpha_i$. To this end we fix $\varepsilon, \delta >0$ such that
\begin{equation}\label{DefEpsDelta}
1/\varepsilon < \nu_i(\PGod, {}^{\zeta_j}C)
< \varepsilon \quad \text{and}\quad 1/\delta < \abs{\alpha_i(f)} < \delta
\end{equation}
for all $i \in \{1, \dots, r\}$, $j \in \{1, \dots, m\}$ and $f \in F$ and compute:
\begin{lemma}\label{lem:nu_alpha}
If $g \in \bfG(\A)$ is decomposed as in \eqref{gDecomp} then for every $i \in \{1,\dots, r\}$ we have
\[
\delta^{-1}\varepsilon^{-1} \abs{\alpha_i(t)} < \nu_i(\PGod,{}^gC) < \delta\varepsilon\abs{\alpha_i(t)}.
\]
\end{lemma}
\begin{proof} Since $\gamma$ normalizes $\PGod$ and $x$ normalizes $C$ we deduce with Lemma \ref{lem:nu_trafo} that
\[
\nu_i({}\PGod,{}^gC) 
= \nu_i(\PGod,{}^{f^{-1}t\zeta_j}C)
= \abs{\alpha_i(f^{-1})}\abs{\alpha_i(t)}\nu_i(\PGod,{}^{\zeta_j}C)\text{.}
\]
The claim now follows from \eqref{DefEpsDelta}.
\end{proof}
\begin{proof}[Proof of Lemma \ref{lem:godement_translation}] 
Assume first that $g \in \Omega_c^{\PGod}$ so that $\nu_i(\PGod, {}^gC)\geq c$
and decompose $g$ as in \eqref{gDecomp}. By Lemma \ref{lem:nu_alpha} we then have
\[
|\alpha_i(t)| > \delta^{-1}\varepsilon^{-1} \nu_i(\PGod,{}^gC) \geq  \delta^{-1}\varepsilon^{-1}c  \implies |\alpha_i(t^{-1})| < \delta\varepsilon c^{-1} \defeq d'.
\]
This shows that $t^{-1} \in T^+_\infty(d')$ and hence $g^{-1} = x^{-1}\zeta_j^{-1}t^{-1}f\gamma^{-1} \in \OmegaGod_{d'}^{\PGod}$.

Conversely assume that $g^{-1}\in \OmegaGod_{d'}^{\PGod}$ and decompose $g$ as in \eqref{gDecomp}. Then $t^{-1} \in T^+_\infty(d)$, i.e. $\abs{\alpha_i(t)}>d^{-1}$,
and hence by Lemma \ref{lem:nu_alpha} we have
\[
\nu_i(\PGod,{}^gC) \geq \delta^{-1}\varepsilon^{-1} \abs{\alpha_i(t)} > \delta^{-1}\varepsilon^{-1}d^{-1} \eqdef c'.
\]
This shows that $g \in \Omega_{c'}^{\PGod}$ and finishes the proof.
\end{proof}
The same argument also shows that for all $\infty > c_2 > c_1$ there exists $0 < d_2 < d_1$ such that
\[
\Omega^{\PGod}_{c_1,c_2} \subseteq (\OmegaGod^{\PGod}_{d_2, d_1})^{-1},
\]
and since $\OmegaGod^{\PGod}_{d_1, d_2}$ is relatively compact modulo $\PGod(k)$ we may record:
\begin{corollary}\label{RelCompact}
The image of $\Omega^{\PGod}_{c_1,c_2}$ in $\PGod(k)\backslash \bfG(\A)$ is relatively compact.\qed
\end{corollary}
We can now derive the first fundamental theorem from Godement's work:
\begin{proof}[Proof of Theorem~\ref{thm:lower}]
The statement of \cite[Théorème~7]{godement64} is that there exists a $d$ such that $\bfG(\A) = \OmegaGod_d^{\PGod} \bfG(k)$. With Lemma~\ref{lem:godement_translation} and \eqref{OmegaEquivariance} we get
\[
\bfG(\A) = \bfG(k) (\OmegaGod_d^{\PGod})^{-1} \subseteq \bfG(k) \Omega_{c'}^{\PGod} = \bigcup_{\gamma \in \bfG(k)} \Omega_{c'}^{{}^\gamma \PGod}=\bigcup_{\bfP \in \mathcal C(\Delta)} \Omega_{c'}^\bfP,
\]
for some $c'>0$, and hence we can choose $c_1 \defeq c'$.
\end{proof}
The second fundamental theorem of reduction theory is essentially \cite[Lemme~3]{godement64}. We reproduce the statement here (in our notation) for cleaner reference. 
\begin{lemma}\label{lem:godement}
For every $d_1 \in \R^\times$ there is a $d_2 \in \R^\times$ such that the following holds.
Let $m,m' \in CZ^{-1}$, $t,t' \in T_\infty^+(d_1)^{-1}$, and $f,f' \in F$ be such that $(m'(t')^{-1}f')^{-1}(mt^{-1}f) = \beta \in \bfG(k)$ and write $\beta = (\pi')^{-1} w \pi$ with $\pi,\pi' \in \bfP(k)$ and $w \in N(\bfT)(k)$ according to the Bruhat decomposition and let $i \in \{1, \dots, r\}$. Then $\abs{\alpha_j(t)} \le d_2$ implies $w \in \bfP_i(k)$.
\end{lemma}

\begin{proof}[Proof of Theorem~\ref{thm:uniqueness}]
Since $\bfG(k)$ is transitive on minimal parabolics, we may assume that $\bfP = \PGod$ and $\bfP' = {}^\lambda \PGod$ for some $\lambda \in \bfG(k)$. Assume this from now on and let $g \in \Omega_{c_1}^\bfP \cap \Omega_{c_1}^{\bfP'}$. By Lemma~\ref{lem:godement_translation} and \eqref{OmegaEquivariance} there then exists $d_1' >0$ such that $g^{-1} \in \OmegaGod^{\PGod}_{d_1} \cap \lambda^{-1} \OmegaGod^{\PGod}_{d_1}$. Thus if we decompose $g$ as in \eqref{gDecomp}, i.e. $g=\gamma f^{-1}t\zeta_jx$ with $t^{-1}\in T^+_\infty(d_1)$. Applying the same argument to $\lambda g$ we obtain a similar decomposition of the form $\lambda g = \gamma' (f')^{-1}t'\zeta_{j'}x'$ with $(t')^{-1} \in  T^+_\infty(d_1)$. If we now set $m\defeq (\zeta_j x)^{-1}$ and $m' \defeq (\zeta_j'x')^{-1}$ then we obtain
\[
(m'(t')^{-1}f')^{-1}(mt^{-1}f) = (\gamma')^{-1}\lambda g g^{-1} \gamma= (\gamma')^{-1}\lambda \gamma =\colon \beta\in \bfG(k).
\]
Now let $d_2$ be the constant from Lemma \ref{lem:godement} (which depends only on $d_1$, and hence only on $c_1$) and set $c_2 \defeq d_2/\delta \varepsilon$. If we assume that $\nu_i(\bfP,{}^gC) \ge c_2$ then by Lemma \ref{lem:nu_alpha} we have $\abs{\alpha_i(t)} \leq \delta \varepsilon c_2 = d_2$, and hence we may apply Lemma \ref{lem:godement}. We then deduce that $\beta  \in \PGod_i(k)$ and since $\gamma, \gamma' \in \PGod(k)$ and $\beta = (\gamma')^{-1}\lambda \gamma$ we obtain $\lambda \in \PGod_i(k)$. This implies that
\[
\bfP_i = \PGod_i = {}^\lambda\PGod_i = ({}^\lambda\PGod)_i = \bfP'_i.\qedhere
\]
\end{proof}

\begin{proof}[Proof of Theorem~\ref{thm:mahler}] The set $\Omega_{c_1,c'}^{\PGod}$ is relatively compact in $\bfG(k)\backslash \bfG(\A)$ by Corollary \ref{RelCompact}. Now by \eqref{OmegaEquivariance} the assumption on $A$ amounts to
\[
A \subseteq \bigcup_{\bfP} \Omega_{c_1,c'}^\bfP = \bigcup_{\gamma \in \bfG(k)} \Omega_{c_1,c'}^{{}^\gamma\PGod} = \bigcup_{\gamma \in \bfG(k)}\gamma^{-1} \Omega_{c_1,c'}^{{}\PGod} = \bfG(k)\Omega_{c_1,c'}^{\PGod},
\]
and hence shows that $A$ is relatively compact modulo $\bfG(k)$.

Conversely assume that $A \subseteq \bfG(\A)^\circ$ is relatively compact and let $(g_n)_{n \in \N}$ be a sequence in $\bfG(\A)$. Since $c_1$ is a lower reduction bound we can find a sequence $\bfP_n$ of minimal parabolic subgroups such that $g_n \in \Omega_{c_1}^{\bfP_n}$. We have to show that $\nu_i(\bfP_n, {}^{g_n}C) \leq c'$ for some uniform constant $c'$. Assume otherwise; passing to a subsequence we may assume that $g_n \to g \in \overline{A}$ and that $\nu_i(\bfP_n,{}^{g_n}C) \to \infty$ for some $i$. From the latter we deduce with Theorem~\ref{thm:uniqueness} that the sequence $(\bfP_n)_i$ eventually stabilizes to some maximal parabolic $\bfQ$ of type $i$, and we may assume that $(\bfP_n)_i = \bfQ$ for all $n \in \mathbb N$. Now observe that
\[
\pi(\bfQ,{}^{g_n}C) = \prod_{i=1}^r \nu_j((\bfP_n)_i,{}^{g_n}C)^{n_{ji}} \ge \nu_i((\bfP_n)_i,{}^{g_n}C)^{n_{ii}}
\]
using first \eqref{eq:dual_to_horofunction} and then Lemma~\ref{lem:root_weight}. It follows that $\pi(\bfQ,{}^{g_n}C)$ tends to infinity. Now let $\gamma \in \bfG(k)$ be such that ${}^{\gamma}\bfQ = \PGod_i$ and write $\gamma g_n = p_n \zeta_{j_n} x_n$ with $p_n \in \PGod(\A)$ and $x_n \in C$ according to (a coarser version of) \eqref{gDecomp}. Passing to a subsequence we may assume that $\zeta_{j_n} = \zeta_j$ is constant and that $x_n \to x \in C$. Then
\begin{align*}
\pi(\bfQ,{}^{g_n}C) = \pi(\PGod_i,{}^{\gamma g_n}C) = \pi(\PGod_i,{}^{p_n \zeta_{j} x_n}C) = \abs{\chi^{\PGod}_i(p_n)} \cdot \pi(\PGod_i,{}^{\zeta_{j}} C)\text{.}
\end{align*}
The second factor on the right is constant so it follows that $\abs{\chi^{\PGod}_i(p_n)}$ tends to infinity. But since $g_n$ converges to $g$, the sequence $p_n = \gamma^{-1} g_n x_n^{-1} \zeta_{j}^{-1}$ converges to $\gamma^{-1} g x^{-1}\zeta_{j}^{-1}$, a contradiction.
\end{proof}

\section{Reduction theory for $S$-arithmetic (approximate) groups}\label{sec:s-arithmetic_reduction}

We now want to pass from the adelic setting to the (approximate) $S$-adelic setting by a cut-and-project construction. More precisely, we will consider the following setting:
\begin{convention}\label{ConventionReductiveCAP}
\begin{enumerate}
\item $k$, $V$, $S$, $\bfG$ and $r$ are as in Convention \ref{SettingReductiveApprox}; as before we assume in addition that $\bfG$ is isotropic to avoid trivial cases;
\item $\pi$, $\nu_1, \dots, \nu_r$ are given as in Convention \ref{ConventionReductiveGroups};
\item $(G, H, \Gamma)$ is the generalized adelic cut-and-project scheme with parameters $(k, \bfG, S)$ and $\Lambda \defeq \bfG(\calO_S) = \Lambda(G, H, \Gamma, W)$ is the $S$-adic approximate group from Convention \ref{SettingReductiveApprox};
\item given a minimal $k$-parabolic $\bfP<\bfG$ and constants $c_2 > c_1 >0 $ we define subsets of $G$ by \[\Omega_{c_1}^{\bfP, S} \defeq \pi_G(\Omega_{c_1}^{\bfP} \cap (G \times \{e_H\})) \quad \text{and} \quad \Omega_{c_1, c_2}^{\bfP, S} \defeq \pi_G(\Omega_{c_1, c_2}^{\bfP} \cap (G \times \{e_H\}))\text{,}\] where $\Omega_{c_1}^{\bfP}$ and $\Omega_{c_1}^{\bfP}$ are defined by \eqref{OmegaPc1} and \eqref{OmegaPc1c2} respectively. 
\end{enumerate}
\end{convention}
Note that our notation suppresses the dependence of $\bfG(\calO_S)$ of the window $W$, but this will not cause any problems to us since we only care about the commensurability class of $\bfG(\calO_S)$ (and anyway a similar abuse of notation is already present in the group case, where $\bfG(\calO_S)$ depends on the choice of representation of $\bfG$). Technically also the sets $\Omega_{c_1}^{\bfP, S}$ and $\Omega_{c_1, c_2}^{\bfP, S}$ depend on various choices, but for the qualitative results that we are interested in, none of these choices will matter.

The $S$-adic versions of the two fundamental theorems of adelic reduction theory (Theorems \ref{thm:lower} and \ref{thm:uniqueness}) can be stated as follows:
\begin{theorem}[Fundamental theorems of $S$-adic reduction theory]\label{thm:lower_s} If $c_1>0$ is a lower reduction bound and $c_2 \geq c_1 > 0$ is an upper reduction bound for $c_1$ then the following hold:
\begin{enumerate}[(i)]
    \item For every $k$-minimal parabolic $\bfP$ there exist a finite set $F \subseteq \pi_S(\bfG(k))$ such that \[\bfG_S \subseteq \bfG(\calO_S)F\Omega^{\bfP, S}_{c_1}.\]
\item If $g \in \Omega_{c_1}^{\bfP, S} \cap \Omega_{c_1}^{\bfP', S}$ and $\nu_i(\bfP,{}^gC) \ge c_2$ then $\bfP_i = \bfP'_i$.
\end{enumerate}
\end{theorem}
Mahler's compactness criterion (Theorem \ref{thm:mahler}) can also be formulated in the $S$-adic setting, at least in the case where $\X[k](\bfG) = 0$ (which is the relevant case for us later on).

 However, one runs into the technical issue that $\bfG(\calO_S) \subseteq G$ is not a subgroup, but only a subset, hence it does not make sense to talk about compactness modulo $\bfG(\calO_S)$. This issue can be resolved by formulating cocompactness in metric terms, using the fact that, given a discrete subgroup $\Lambda < G$, a subset $M \subseteq G$ has compact image in $\Lambda \backslash G$ iff it is at bounded distance from $\Lambda$ with respect to some (hence any) left-admissible pseudo-metric on $G$. 
\begin{theorem}[Mahler's compactness criterion, $S$-adic version]\label{thm:mahler_s}
Let $d_G$ be a left-admissible pseudo-metric on $G$ and let $c_1$ be a lower reduction bound. If $\X[k](\bfG) = 0$ then 
a subset $A \subseteq G$ has bounded distance from $\bfG(\calO_S)$ if and only if there is a $c' > c_1$ such that $A \subseteq \bigcup_{\bfP} \Omega^{\bfP,S}_{c_1,c'}$.
\end{theorem}
Note that the statement is independent of the choice of left-admissible pseudo-metric, since any two such pseudo-metrics are coarsely equivalent. We will later use pseudo-metrics induced from associated metric spaces. The following reformulation of Theorem~\ref{thm:mahler_s} is adapted to our later application.
\begin{corollary}\label{cor:mahler_s}
Assume that $\X[k](\bfG) = 0$. Fix $c_1 > 0$. For every $r > 0$ there is a $c' > c_1$ such that $N_r(\bfG(\calO_S)) \subseteq \bigcup_{\bfP} \Omega^{\bfP.S}_{c_1,c'}$.
For every $c' > c_1$ there is an $r > 0$ such that $\bigcup_{\bfP} \Omega^{\bfP,S}_{c_1,c'} \subseteq N_r(\bfG(\calO_S))$.
\end{corollary}

We now turn to the proofs of Theorems \ref{thm:lower_s} and \ref{thm:mahler_s}. In the case where $S \supseteq V\inf$ and hence $\bfG(\calO_S)$ is an actual $S$-arithmetic \emph{subgroup} of $\bfG_S$ by Proposition~\ref{SGroupVsApproxGp}, both theorems are well-known. In this classical setting, adelic reduction theory is traditionally deduced from the Archimedean case (where $S = V\inf$), and general $S$-adic reduction theory is obtained as an intermediate step, see \cite[§8]{Borel}. It was first suggested by Godement and Weil \cite[§12]{godement64} that $S$-adic reduction theory could also be obtained from adelic reduction theory via descent. The main advantage of this descent approach for our purposes is that it applies equally well to the case where $S$ does not contain all infinite places, as we are going to demonstrate. We are going to apply Proposition~\ref{prop:descent} in the following form:
\begin{corollary}\label{cor:descent} 
(i) For all compact subsets $K \subseteq H$ there exists a finite set $F \subseteq G$ such that for all subsets $\Pi \subseteq G$,
\[
\pi_G(\Gamma (\Pi\times K) \cap (G \times \{e_H\})) \subseteq \bfG(\mathcal O_S) F \Pi \text{.} 
\]
(ii) For every $\bar \Omega \subseteq G$ there exists a finite set $F \subseteq \Gamma_G$ such that
\[
\pi_G(\Gamma(\bar \Omega \times W) \cap (G \times \{e_H\})) \subseteq \bfG(\mathcal O_S)F\bar \Omega.
\]
\end{corollary}
\begin{proof} The first statement is Proposition~\ref{prop:descent} applied to the triple $(G, H, \Gamma)$ from Convention \ref{ConventionReductiveCAP} with $I \defeq \{e_H\}$, and the second statement with $F \subseteq G$ follows by choosing $\Pi \defeq \bar \Omega$ and $K \defeq W$; the fact that $F$ can be chosen inside  $\Gamma_G$ follows from  Remark~\ref{rem:descent_detail}.
\end{proof}

\begin{proof}[Proof of Theorem \ref{thm:lower_s}]
(i) Let $G, H, \Gamma$ as in Convention \ref{ConventionReductiveCAP}. Since $c_1$ is a lower reduction bound, we have $\bfG(\A) = \bigcup_{\bfP \in \mathcal C(\Delta_k)} \Omega^{\bfP}_{c_1}$. We now fix a minimal $k$-parabolic $\bfP \in \mathcal C(\Delta_k)$. Since $\Gamma = \bfG(k)$ acts transitively on $\mathcal C(\Delta_k)$ 
by Proposition~\ref{prop:min_par_conj}, we then have $\bfG(\A) = \Gamma \Omega$, where $\Omega \defeq \Omega^{\bfP}_{c_1}$. We now claim that $\Gamma(\bar{\Omega} \times W) = \bfG(\A)$, where $\bar{\Omega} \defeq \pi_G(\Omega \cap (G \times \{e_H\}))$.

To prove the claim, we fix $g \in \bfG(\A)$ and decompose $g = \gamma \omega$ with $\gamma \in \Gamma$ and $\omega \in \Omega$. Since $\Gamma$ projects densely to $H$ and $W$ has non-empty interior, there exists a $\gamma' \in \Gamma$ with $\pi_H(\gamma'\omega) \in W$. Thus $g = (\gamma \gamma'^{-1}) (\gamma'\omega)$ with $\gamma\gamma'^{-1} \in \Gamma$ and $\gamma'\omega \in \bar{\Omega} \times W$, and the claim follows.

We deduce that $\pi_G(\Gamma(\bar \Omega \times W) \cap (G \times \{e_H\})) = G$ and thus $G \subseteq \bfG(\mathcal O_S)F\bar \Omega$ for some finite set $F \subseteq \Gamma_G$ by Corollary \ref{cor:descent}.

(ii) This is immediate from Theorem \ref{thm:uniqueness} applied to $(g,e_H) \in G \times H$, since the latter is a uniqueness rather than an existence theorem.
\end{proof}

\begin{proof}[Proof of Theorem \ref{thm:mahler_s}] We fix an auxiliary left-admissible pseudo-metrics $d_H$ on $H$ and equip $G \times H$ with the left-admissible pseudo-metric
\[
d((g_1,h_1),(g_2,h_2)) \defeq \max\{d_G(g_1,g_2),d_H(h_1,h_2)\}\text{.}
\]
Assume first that $A \subseteq G$ is a subset whose distance from $\bfG(\calO_S)$ is bounded by $D$. Since $W$ is compact we may assume that $d(w, e_H) \leq D$ for all $w \in W$. For every $a \in A$ we pick an element $\gamma_a \in \bfG(\calO_S)$ with $d_G(a, \gamma_a) \leq D$ and denote by $\gamma_a^*$ the unique element of $W$ such that $(\gamma_a, \gamma_a^*) \in \Gamma$. For every $a \in A$ we then have
\[
d((a, e_H), (\gamma_a, \gamma_a^*)) = \max\{d_G(a,\gamma_a), d_H(\gamma_a^*, e_H)\} \leq D,
\]
hence $A \times \{e_H\}$ is at bounded distance from $\Gamma$ and hence relatively compact modulo $\Gamma$. We deduce from Theorem \ref{thm:mahler} that there is $c'>c_1$ such that
\[
A \times \{e_H\} \subseteq \bigcup_{\bfP} \Omega^{\bfP}_{c_1,c'}, \quad \text{and hence} \quad A \subseteq \bigcup_{\bfP} \Omega^{\bfP, S}_{c_1,c'}.
\]

Conversely assume that $A \subseteq \bigcup_{\bfP} \Omega^{\bfP, S}_{c_1,c'}$ for some $c'>c_1$. By Theorem \ref{thm:mahler} there exist compact subsets $\Pi \subseteq G$ and $K \subseteq H$ such that $\bigcup_{\bfP} \Omega_{c_1, c'}^{\bfP}  \subseteq \Gamma(\Pi \times K)$ and hence
\[
A \subseteq \pi_G\left(\bigcup_{\bfP} \Omega_{c_1, c'}^{\bfP} \cap (G \times \{e_H\})\right) \subseteq \pi_G(\Gamma(\Pi \times K) \cap (G \times W)).
\]
By Corollary \ref{cor:descent} we thus have $A \subseteq \bfG(\calO_S)F\Pi$, and since $d_G$ is left-invariant and bounded on the compact set $F\Pi$ we deduce that $A$ is at bounded distance from $\bfG(\calO_S)$ with respect to $d_G$.
\end{proof}

\section{Geometric reduction theory}\label{sec:geometric_reduction}

In this section we finally translate algebraic reduction theory as developed in the previous sections to the geometric language of \cite[Section~12]{bux13}. We include infinite places (i.e.\ symmetric spaces) for completeness although we will not need them for the proof of the Main Theorem. 

\subsection{Some convex geometry}

For our geometric formulation of reduction theory we need some notions from convex geometry in Euclidean spaces. Throughout, $\Sigma$ denotes a Euclidean space with underlying Euclidean vector space  $(V, \langle \cdot, \cdot \rangle)$, cf.\ Example \ref{AffineForms}. Given subsets $A \subset \Sigma$ (or $A \subset V$) and $B \subset V$ we denote by $A + B \defeq \{a+b \mid a \in A, b \in B\}$ their Minkowski sum. Note that $A + B \subseteq \Sigma$ or $A + B \subseteq V$ depending on where $A$ was taken from.
We write $a+B \defeq \{a\}+B$.
Given a closed convex subset $P \subset \Sigma$, we  denote by $p_P \colon \Sigma \to P$ the closest point projection.

A \emph{polyhedron} $P$ in $V$ or $\Sigma$ is a finite intersection of closed affine halfspaces. A polyhedron $P$ in $V$ or $\Sigma$ is called a \emph{(generalized) cone} if it contains a unique minimal non-empty face; this face is then denoted $T(P)$ and called the \emph{tip} of $P$, and $P$ is called a \emph{proper cone} if its tip is a singleton.

A (generalized) \emph{linear cone} $C \subset V$ an intersection of finitely many closed linear halfspaces, i.e.\ halfspaces bounded by a hyperplane through the origin. Its tip $T(C)$ is the unique maximal vector subspace of $C$ and $C = T(C) + C_o$ where $C_o \defeq C \cap T(C)^{\perp}$ is a proper linear cone. Every cone $P \subset \Sigma$ can be written as $o+C$ for a linear cone $C \subset V$ and a basepoint $o \in \Sigma$; its tip is $T(P) = o+T(C)$ and we have $P = T(P) + C_o$ with $C_o \defeq C \cap T(C)^\perp$. This decomposition is unique up to the choice of $o$ in $T(C)$.

Now let $P$ be a polyhedron in $\Sigma$. We denote by $\mathcal F_P$ the poset of faces of $P$ and by $\mathcal F_P^*$ the set of \emph{non-empty} faces of $P$. Given $F \in \mathcal F_P^*$ we denote by $\mathring F$ the relative interior of $F$ in its affine span.  This is $F$ without its proper faces so if $F$ is minimal in $F_P^*$ we have $\mathring F = F$.

 \begin{definition} The \emph{normal cone} of $F \in \mathcal F_P^*$ is 
 \[
 \NC(F) \defeq  \{v \in V \mid \gen{v,f-p} \ge 0 \text{ for all }f \in F, p \in P\} \subseteq V
 \] 
 and its \emph{normal set} is
 \[
N_F \defeq \mathring{F} + \NC(F)
\]
(the dependence on $P$ is implicit).
 \end{definition}

The collection $\mathcal N_P \defeq \{\NC(F) \mid F \in \mathcal F_P^*\}$ is  a complete fan, called the \emph{normal fan} of the polyhedron $P$, and the map $\mathcal F_P^\times \to \mathcal N_P$, $F \mapsto \NC(F)$ is an inclusion-reversing bijection such that 
\begin{equation}\label{DimNormalCone}
\dim \NC(F) = \dim \Sigma - \dim F.
\end{equation}
In particular, if $P$ is a cone, then $\NC(T(P))$ is the unique cone of maximal dimension in $\mathcal N_P$.

\begin{remark}
A more geometric description of the normal cone is
\begin{align}
\NC(F)  &=  \{x - p_P(x) \mid x \in \Sigma, p_p(x) \in \mathring F\}
\end{align}
which shows that $N_F = p_P^{-1}(F)$.
\end{remark}

Now let $F \in \mathcal F_P$. Then $N_F$ is convex, and it follows from \eqref{DimNormalCone} that  $N_F$ has dimension $\dim \Sigma$ and hence non-empty interior. It is closed if and only if $F$ is a minimal non-empty face. Note that
\begin{equation}\label{NormalSetProject}
  \Sigma = \dot\bigcup_{F \in \mathcal{F}_P} N_F
\end{equation}
by the remark above, i.e.\ the normal sets partition $\Sigma$.

If $C$ is a generalized linear cone in $V$ and $(o_t)_{t\in \R}$ is a sequence of basepoints in $\Sigma$, then $(Y_t \defeq o_t+C)_{t\in \R}$ is called a \emph{properly nested family of generalized cones} modeled on $C$ if 
\[
\bigcup_{t \in \R} Y_t = \Sigma, \quad \bigcap_{t \in \R}Y_t = \emptyset \quad \text{and} \quad Y_s \subset \mathring{Y}_t \text{ for all }s < t.
\]
We observe that if $(Y_t)_{t \in \R}$ is a properly nested family of cones modeled on $C = T(C) +C_o$, then $(Y_t' \defeq T(Y_{-t})+(-C_o))$ is also a properly nested family of cones, modeled on $-C$.

\begin{example}\label{ex:busemann_nested_cones}
If $b_1, \dots, b_k$ are non-constant affine functions on $\Sigma$ (i.e.\ rescaled Busemann functions), then 
\[
Y_t \defeq \{x \in \Sigma \mid b_j(x) \leq t \text{ for all } 1\leq j\leq k \}
\]
is a properly nested family of generalized cones in $\Sigma$.
\end{example}

\begin{lemma}\label{ConeInclusion} Let $(N_t)_{t \in \R}$, $(Z_t)_{t \in \R}$ be two properly nested families of generalized cones modeled on the same generalized cone $C \subset V$. Then for every $t \in \R$ there exist $t_0, t_1 \in \R$ such that $Z_{t_0} \subset \mathring N_{t} \subset N_t \subset \mathring Z_{t_1}$.
\end{lemma}

\begin{proof} Picking an origin of $\Sigma$ in $T(N_t)$ we may assume $\Sigma = V$ and $N_t = C$. Taking the quotient by $T(C)$ we may further assume that $C$ is a proper cone. Since $\bigcup_t Z_t = V$ there is a $t_1'$ such that $0 \in Z_{t_1'}$. Thus for $t_1 < t_1'$ we have
\[
  Z_{t_1} \subseteq \mathring{Z}_{t_1'} = T(Z_{t_1'}) + C \subseteq 0 + C = \mathring{N}_t\text{.}
\]
Exchanging the roles of $(N_t)_t$ and $(Z_t)_t$ one obtains $t_0$.
\end{proof}

If $(Y_t)_{t \in \R}$ is a properly nested family of generalized cones modeled on $C \subset V$, then $(N_{T(Y_t)})_{t \in \R}$ is again a family of generalized cones (modeled on the normal cone $\NC(C)$). This new family need not be properly nested in general, see Figure~\ref{fig:desc_norm} below. However, if $(Y_t)_t$ is given by rescaled Busemann functions as in Example~\ref{ex:busemann_nested_cones}, the functions can be rescaled to get a new properly nested family of generalized cones such that the normal sets of the tips are properly nested as well:

\begin{lemma}\label{lem:rescale}
Let $\tilde{\beta}_1,\ldots,\tilde{\beta}_d \colon \Sigma \to \R$
be affine forms with linearly independent linear part. Then there exist constants $s_1,\ldots,s_d >0$ such that for $t > 0$ the normal set of the tip in $Y_t = \{x \in \Sigma \mid s_i\tilde{\beta}_i(x) \le t\}$ is contained in the interior of the normal set of the tip in $Y_0 = \{x \in \Sigma \mid \tilde{\beta}_i(x) \le 0\}$.
\end{lemma}
\begin{proof}
By choosing the origin to lie in the intersection of the zero-levels, we can assume that the $\tilde{\beta}_i$ are linear forms on a Euclidean vector space.
Now note that if a tuple of $s_i$ works for $t = 1$ then it works for any $t$. Since the normal cones are translates of each other, we only need to ensure that some/any vector $v$ with $s_i\tilde{\beta}_i(x) = 1$ lies in the interior of the normal set of $Y_0$. A solution is therefore to simply take $s_i \defeq 1/\tilde{\beta}_i(v)$, where $v$ is any point in the interior of the normal set of $Y_0$.
\end{proof}

\subsection{Generalized cones from renormalized Busemann functions}
 We now return to the setting of Convention~\ref{ConventionReductiveCAP} from the last section. In particular, $\Lambda = \bfG(\calO_S)$ is an $S$-arithmetic approximate subgroup of a reductive group $\bfG$ over a global field $k$ where $S$ is a finite set of places, which acts on the \cato{}-space $X_S = \prod_{s \in S} X_s$. We denote by $\Delta = \Delta_k$ the spherical building associated to $\bfG(k)$, which we consider as a subset of $\partial X_S$. In addition we assume that the invariant horofunction $\pi$ is normalized so that by \eqref{eq:horo_to_busemann} there exists a unique family $(\tilde{\beta}_{\bfQ,S})_{\bfQ \in \calV(\Delta_k)}$ of rescaled Busemann functions such that $\tilde{\beta}_{\bfQ, S}$ is centered at $\bfQ$ and satisfies
\[
 \tilde{\beta}_{\bfQ, S}|_{\calK_S} = \log \pi(\bfQ,\cdot)|_{\calK_S}\text{.}
\]
As in \eqref{Defmus} we then set $ \mu^{\bfP}_i =  \sum_{j=1}^r c_{ij} \tilde{\beta}_{\bfP_j,S}$ so that 
\[
 \mu_{i}^{\bfP}|_{\calK_S} = \log  \nu_i(\bfP, \cdot)|_{\calK_S}
\]
(the reason for decorating the rescaled Busemann functions with a tilde will become clear shortly).

In our geometric formulation of reduction theory, it seems more natural to us to employ geometric terminology when referring to elements of $\Delta$, i.e.\ we will speak of vertices, chambers, simplices rather than maximal parabolics, minimal parabolics, parabolics and containment will be reversed. We also denote by $\mathrm{typ}(v) \in \typ(\Delta) \defeq \{1, \dots, r\}$ the type of the (parabolic corresponding to) the vertex $v$, cf.\ Section~\ref{sec:pars_chars}. If $v$ is a vertex corresponding to a maximal parabolic $\bfQ$, then we will denote the corresponding Busemann function by \[\tilde{\beta}_v \defeq \beta_{\bfQ, S}.\] Similarly, if $c$ is a chamber, corresponding to a minimal parabolic $\bfP$, then we write  $\mu_i^c \defeq  \mu_{i}^{\bfP}$. If necessary we denote the stabilizer of a vertex $v$ or chamber $c$ by $\bfP_v$ or $\bfP_c$ respectively.

The rescaled Busemann functions $(\tilde{\beta}_v)_{v \in \Vt(\Delta)}$ are a main ingredient for a \emph{geometric reduction datum} as in \cite[Section~1]{bux13}. We will not work with these functions directly but rather rescale them as follows. For each type $i \in \typ(\Delta)$ we will choose a constant $s_i > 0$ (to be specified later) and define
\begin{equation}\label{eq:beta_scaling}
\beta_v \defeq s_{\typ(v)} \tilde{\beta}_v\text{.}
\end{equation}
We summarize the transformation behavior of the functions $\beta_v$, $\mu_i^c$ and the relationship between them by the following formulae that are additive versions of Definition~\ref{def:pi_trafo}\eqref{item:pi_trafo_parabolic}, Lemma~\ref{lem:nu_trafo}\eqref{item:nu_trafo_parabolic}, \eqref{eq:horofunction_to_dual} and \eqref{eq:dual_to_horofunction} respectively, taking into account the rescaling \eqref{eq:beta_scaling}:

\begin{align}
\beta_v(g.x) &= s_{\typ(v)}\log(\abs{\chi_{\bfP_v}(g)}) + \beta_v(x) &\text{for}\ g \in (\bfP_v)_S\text{,}\label{eq:busemann_trafo}\\
\mu_i^c(g.x) &= \log(\abs{\alpha^{\bfP_c}_i(g)}) + \mu_i^c(x) &\text{for}\ g \in (\bfP_c)_S\text{,}\label{eq:root_trafo}\\
\mu_i^c &= \sum_{j = 1}^r \frac{\log c_{ij}}{s_j} \beta_j & \text{for }v < c \text{ of type }j\text{,}\label{eq:beta_to_mu}\\
\beta_v &= s_j \cdot \sum_{i=1}^r \log n_{ji} \cdot \mu_i^c & \text{for }v < c \text{ of type }j\text{.}\label{eq:mu_to_beta}
\end{align}

An important consequence of all these formulae is that the situation described by the $\beta_v$ and $\mu^c_i$ looks the same around each chamber in the following sense:

\begin{lemma}\label{lem:all_chambers_equal}
Let $c,c' \in \Ch(\Delta) \subset \partial X_S$ be chambers and let $\Sigma,\Sigma' \subset X_S$ be apartments with $c \subseteq \partial \Sigma$ and $c' \subseteq \partial \Sigma'$. Let $v_i$ respectively $v_i'$ be the vertices of $c$ respectively $c'$ of type $i \in \typ(\Delta)$. There is an isometry $\iota \colon \Sigma \to \Sigma'$ such that
\begin{align}
\beta_{v_i}|_\Sigma &= \beta_{v_i'}|_{\Sigma'} \circ \iota &&\text{and}\label{eq:beta_same}\\
\mu^{c}_i|_\Sigma &= \mu^{c'}_i|_{\Sigma'} \circ \iota &&\text{for all }i\text{.}\label{eq:mu_same}
\end{align}
\end{lemma}

\begin{proof} Firstly, $\mathrm{Is}(X_S)$ acts transitively on triples $(\Sigma, c)$, where $\Sigma$ is an apartment and $c \subseteq\partial \Sigma$ is a chamber, hence there exists an isometry $\iota: \Sigma \to \Sigma'$ mapping $c$ to $c'$. Then for every given $y \in \typ(\Delta)$ the functions $\beta_{v_i}|_\Sigma$ and $\beta_{v_i'}|_{\Sigma'} \circ \iota$ are two affine functions on $\Sigma$ with the same center. Modifying $\iota$ if necessary we can ensure that the zero level sets of these two affine functions agree, but then these two affine functions differ only by a scaling. However, it follows from Equation \eqref{eq:busemann_trafo} that
\[
\max_{\substack{x,y \in X\\x\ne y}} \frac{\beta_{v_i}(x) - \beta_{v_i}(y)}{d(x,y)} = \max_{\substack{x',y' \in X\\x' \ne y'}} \frac{\beta_{v_i'}(x') - \beta_{v_i'}(y')}{d(x',y')}
\]
and this maximum is realized when $x$ lies on the ray $[y,v_i)$ and $x'$ lies on the ray $[y',v_i')$, so in particular in the apartments $\Sigma$ and $\Sigma'$. This implies that the scaling is the same, hence there is a map $\iota$ that takes $c$ to $c'$ and satisfies \eqref{eq:beta_same} for our fixed choice of $i \in \typ(\Delta)$. Since the functions $\beta_{v_i}|_\Sigma$ are linearly independent, \eqref{eq:beta_same} can be satisfied for all types simultaneously. Then \eqref{eq:mu_same} just follows from \eqref{eq:beta_to_mu}.
\end{proof}



Now let $\tau$ be a simplex in $\Delta \subset \partial X_S$ and let $\Sigma$ be an apartment in $X_S$ with $\tau \subset \partial \Sigma$. We denote by $(V, \langle \cdot, \cdot \rangle)$ the underlying Euclidean vector space of $\Sigma$. We may identify $\partial \Sigma$ with the $1$-sphere $S(V)$ in $V$ such that a unit vector $v \in S(V)$ corresponds to the class of $(t \mapsto o + tv)$ for some (hence any) basepoint $o \in \Sigma$. Under this identification, the vertices of $\tau$ correspond to points in $S(V)$ and we can define a generalized cone in $V$ by
\[
C_{\Sigma, \tau} = \{w \in V \mid \langle w, v \rangle \geq 0 \text{ for every vertex }v\text{ of }\tau\}.
\]
Then for $t \in \R$ the sets
\[
Y_{\Sigma,\tau}(t) \defeq \{x \in \Sigma \mid \beta_v(x) \le t\text{ for every vertex }v\text{ of }\tau\}
\]
form a properly nested family of generalized cones modeled on $C_{\Sigma, \tau}$. Note that the tip of $Y_{\Sigma,\tau}(t)$ is given by
\[
F_{\Sigma,\tau}(t) \defeq  \{x \in \Sigma \mid \beta_v(x) = t\text{ for every vertex }v\text{ of }\tau\}.
\]
\begin{remark}
If $\tau = v$ is a vertex, then $Y_{\Sigma,\tau}(t)$ is just a sublevel set of the affine function $\beta_v$, hence an affine halfspace in $\Sigma$ and equal to its tip. If $\tau = c$ is a chamber, then the dimension of the tip $F_{\Sigma, \tau}(t)$ is given by $\left(\sum_{s \in S}\mathrm{rk}_{k_s}(\bfG(k_s))\right) -\mathrm{rk}_k(\bfG(k))$, and hence the generalized cone $Y_{\Sigma,\tau}(t)$ is typically not proper.
\end{remark}

We now return to the \cato{}-space $X_S$. Given $t \in \R$, an apartment $\Sigma \subset X_S$ and a simplex $\tau$ in $\partial \Sigma$ we denote by $N_{\Sigma, \tau}(t)$ the normal set of the tip $F_{\Sigma,\tau}(t)$ in $Y_{\Sigma,\tau}(t)$. Explicitly, $N_{\Sigma, \tau}(t)$ is the generalized cone given by
\begin{equation}
N_{\Sigma, \tau}(t) = \{x \in \Sigma \mid \beta_{v}(\pr_{Y_{\Sigma,t}} x) = t\text{ for every vertex }v\text{ of }\tau\}\text{.}
\end{equation}

From the properly nested family $(Y_{\Sigma,c}(t))_{t \in \R}$ of generalized cones (modeled on $C_{\Sigma, \tau}$) we have thus produced another family of generalized cones $(N_{\Sigma, \tau}(t))_{t\in \R}$ (modeled on the normal cone of $C_{\Sigma, \tau}$). However, depending on our choice of the rescaling parameters $s_1, \dots, s_r$ this second family of generalized cones may or may not be properly nested (see Figure~\ref{fig:desc_norm} for an illustration).
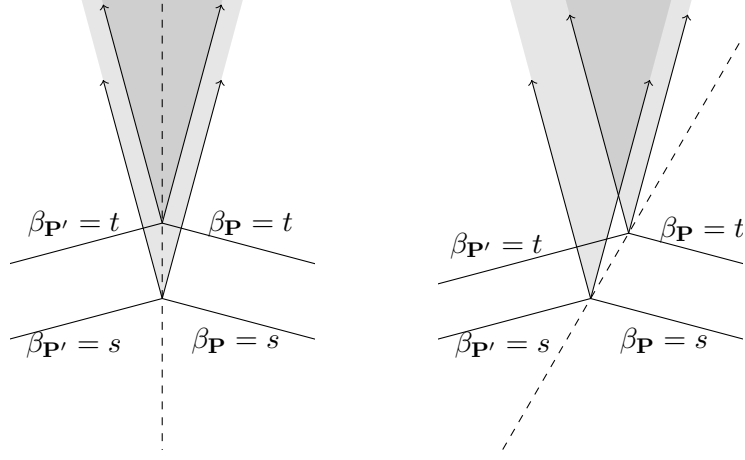
\begin{figure}
\hspace{\fill}
\begin{tikzpicture}
\newcommand{\coneang}{150}
\newcommand{\vecang}{0}
\clip (-2,-2) rectangle (2,4);
\draw (0,0) -- ({-90+\coneang/2}:5) (0,0) -- ({-90-\coneang/2}:5);
\fill[color = black, opacity=.1] (0,0) -- ({\coneang/2}:10) -- ({180-\coneang/2}:10) -- cycle;
\draw (0,0) edge[->] ({\coneang/2}:3) (0,0) edge[->] ({180-\coneang/2}:3);
\draw[dashed] ({90-\vecang}:-4) -- ({90-\vecang}:4);
\draw[shift={({90-\vecang}:1)}] (0,0) -- ({-90+\coneang/2}:5) (0,0) -- ({-90-\coneang/2}:5);
\fill[shift={({90-\vecang}:1)},color = black, opacity=.1] (0,0) -- ({\coneang/2}:10) -- ({180-\coneang/2}:10) -- cycle;
\draw[shift={({90-\vecang}:1)}] (0,0) edge[->] ({\coneang/2}:3) (0,0) edge[->] ({180-\coneang/2}:3);
\node[anchor=north] at ({-90+\coneang/2}:1) {$\beta_{\bfP} = s$};
\node[anchor=north] at ({-90-\coneang/2}:1.2) {$\beta_{\bfP'} = s$};
\node[shift={({90-\vecang}:1)},anchor=south] at ({-90+\coneang/2}:1.2) {$\beta_{\bfP} = t$};
\node[shift={({90-\vecang}:1)},anchor=south] at ({-90-\coneang/2}:1.2) {$\beta_{\bfP'} = t$};
\end{tikzpicture}
\hspace{\fill}
\begin{tikzpicture}
\newcommand{\coneang}{150}
\newcommand{\vecang}{30}
\clip (-2,-2) rectangle (2,4);
\draw (0,0) -- ({-90+\coneang/2}:5) (0,0) -- ({-90-\coneang/2}:5);
\fill[color = black, opacity=.1] (0,0) -- ({\coneang/2}:10) -- ({180-\coneang/2}:10) -- cycle;
\draw (0,0) edge[->] ({\coneang/2}:3) (0,0) edge[->] ({180-\coneang/2}:3);
\draw[dashed] ({90-\vecang}:-4) -- ({90-\vecang}:4);
\draw[shift={({90-\vecang}:1)}] (0,0) -- ({-90+\coneang/2}:5) (0,0) -- ({-90-\coneang/2}:5);
\fill[shift={({90-\vecang}:1)},color = black, opacity=.1] (0,0) -- ({\coneang/2}:10) -- ({180-\coneang/2}:10) -- cycle;
\draw[shift={({90-\vecang}:1)}] (0,0) edge[->] ({\coneang/2}:3) (0,0) edge[->] ({180-\coneang/2}:3);
\node[anchor=north] at ({-90+\coneang/2}:1) {$\beta_{\bfP} = s$};
\node[anchor=north] at ({-90-\coneang/2}:1.2) {$\beta_{\bfP'} = s$};
\node[shift={({90-\vecang}:1)},anchor=south] at ({-90+\coneang/2}:1) {$\beta_{\bfP} = t$};
\node[shift={({90-\vecang}:1)},anchor=south] at ({-90-\coneang/2}:1.8) {$\beta_{\bfP'} = t$};
\end{tikzpicture}
\hspace{\fill}
\caption{The rescaled Busemann function on the left satisfy the nested normal cones property, the ones on the right do not: the normal cone for the larger value $t$ is not contained in the one for the smaller value $s$.}
\label{fig:desc_norm}
\end{figure}

\begin{definition} We say that $(\beta_v)_{v \in \Vt(\Delta)}$ satisfies the \emph{nested normal set} property if for some (hence any) pair $(c,\Sigma)$ with $c \in \Ch(\Delta)$ and $c \subseteq \partial \Sigma$ the family $(N_{\Sigma,c}(t))_{t \in \R}$ is properly nested.
\end{definition}
\begin{proposition}\label{prop:nested_normal_cones}
The constants $(s_i)_{i \in \typ(\Delta)}$ can be chosen in such a way that the functions $(\beta_v)_v$ satisfy the nested normal set property.
\end{proposition}
\begin{proof}
We first fix a chamber $c$ in the boundary of some apartment $\Sigma$. By Lemma~\ref{lem:rescale} there exist constants $s_i > 0, i \in \typ(\Delta)$ such that the rescaled Busemann functions $s_{\typ(v)}\tilde{\beta}_v$ for $v \in \Vt(c)$ satisfy $N_{\Sigma,c}(s)$ contains $N_{\Sigma,c}(t)$ in its interior whenever $s < t$. It then follows from Lemma~\ref{lem:all_chambers_equal} that the same holds for any pair $(\Sigma', c')$ where $\Sigma'$ is an apartment and $c'$ is a chamber in its boundary. By definition, this means that $(\beta_v)_{v \in \Delta}$ has the nested normal set property.
\end{proof}
The reason that we are interested in the nested normal cone property is that it allows us to relate the normal sets  $N_{\Sigma,c}(t)$ with the sets
\[
Z_{\Sigma,c}(t) \defeq \{x \in \Sigma \mid \forall i \in \{1, \dots, r\}:\; \mu_i^c(x) \ge t\}
\]
induced by dual invariant horofunctions. 
\begin{lemma}\label{lem:roots_versus_normal_cones}
Assume that $(\beta_v)_{v \in \Vt(\Delta)}$ have the nested normal cones property. Then for every $t \in \R$ there exist $t_0, t_1 \in \R$ such that for every apartment $\Sigma$ and every chamber $c \subseteq \partial \Sigma$,
\[
Z_{\Sigma,c}(t_0) \subseteq N_{\Sigma,c}(t) \subseteq Z_{\Sigma,c}(t_1) \quad \text{and} \quad
N_{\Sigma,c}(t_0) \subseteq Z_{\Sigma,c}(t) \subseteq N_{\Sigma,c}(t_1)\text{.}
\]
\end{lemma}
\begin{proof} Firstly,  if we prove the statement for some pair $(\Sigma,c)$, then it will hold for all such pairs by Lemma~\ref{lem:all_chambers_equal}, hence we assume from now on that the pair $(\Sigma, c)$ is fixed.

Next we observe that $(Z_{\Sigma,c}(t))$ is a properly nested family of generalized cones, modeled on the generalized cone
\[
C'_{\Sigma, c} \defeq \{v \in V \mid \forall i \in \{1, \dots, r\}:\; \mathring \mu_i^c(v) \geq 0\},
\]
and by assumption $(N_{\Sigma,c}(t))_{t \in \R}$ is a properly nested family of generalized cones, modeled on the normal cone of $C_{\Sigma, c}$. In view of Lemma \ref{ConeInclusion} it thus suffices to show that these two cones coincide; we prove this by showing that $N_{\Sigma,c}(0) = Z_{\Sigma,c}(0)$.

Note that $\tilde{\beta}_v \ge 0$ if and only if $\beta_v \ge 0$ so as far as the cones with $t = 0$ are concerned, we may work with $\tilde{\beta}_v$ instead of $\beta_v$. We regard the Euclidean space $\Sigma$ as a vector space by picking an origin in the subspace $\{x \in \Sigma \mid \tilde{\beta}_v(x) = 0 \text{ for all } v \in \Vt(c)\}$. Then the restrictions of $\tilde{\beta}_v$ and $\mu_i^c$ can be regarded as elements of the dual vector space. The transformation matrix between them is the matrix $(c_{ij})_{ij}$ from \eqref{RootsVsCharacters}. Now $c_{ij} = 0$ for $i \ne j$ and $c_{ii} > 0$. Therefore, for $v \in \Vt(c)$ of type $i$ and $x \in \Sigma$ the condition $\mu_i^c(x) \ge 0$ is equivalent to the condition $\tilde{\beta}_v(\pr_{Y_{\Sigma,t}(0)}) \ge 0$. This shows that $N_{\Sigma,c}(0) = Z_{\Sigma,c}(0)$ and finishes the proof.
\end{proof}

\subsection{Reduction data}
We can now reformulate $S$-adic reduction theory in geometric terms. The key notion is that of a \emph{reduction datum} from \cite{bux13} (with some necessary modification, see the discussion in Subsection \ref{CompareBKW} below).

Let $\tau$ be a simplex in $\Delta$ and let $x \in X_S$. There then exists an apartment $\Sigma$ with $x \in \Sigma$ and $\tau \subset \partial \Sigma$. For every $t \in \R$ the set $Y_{\Sigma, \tau}(t)$ is a closed convex subset of $\Sigma$ and it follows from Lemma \ref{lem:all_chambers_equal} that for every vertex $v \in \tau$ the value $\beta_v(\pr_{Y_{\Sigma,\tau}(t)})$ is independent of the chosen apartment $\Sigma$. In particular, the set 
\[
\mathcal V(\sigma_t(x,\tau)) \defeq \{v\text{ vertex of } \tau \mid \beta_v(\pr_{Y_{\Sigma,\tau}(t)}) = t\}
\]
is independent of the choice of apartment $\Sigma$, and it is the set of vertices of a unique subsimplex $\sigma_t(x,\tau)$ of $\tau$. By constriction, $\sigma_t(x,\tau)$ is the largest face of $\tau$ such that $x$ is contained in $N_{\Sigma,\tau}(t)$ for some (hence any) apartment $\Sigma$ as above.
\begin{observation} Given a chamber $c \subset \Delta$ and a point $x \in X$ we say that an apartment $\Sigma \subset X$ is a \emph{$(c,x)$-apartment} if $x \in \Sigma$ and $c \subset \partial \Sigma$. Then
the following are equivalent:
\begin{enumerate}[(i)]
\item $\sigma_t(x, c) = c$.
\item $\beta_v(\pr_{Y_{\Sigma,\tau}(t)}) = t$ for every vertex $v \in c$ and some (hence any) $(c,x)$-apartment $\Sigma$.
\item $\pr_{Y_{\Sigma,\tau}(t)}$ is contained in the tip of $Y_{\Sigma, \tau}$ for some (hence any) $(c,x)$-apartment $\Sigma$.
\item $x \in N_{\Sigma, c}(t)$ for some (hence any) $(c,x)$-apartment $\Sigma$.
\end{enumerate}
In this case we say that the chamber $c$ \emph{$t$-reduces} $x$.
\end{observation}
We can now state the main definition of this section.
\begin{definition}\label{def:reduction_datum}
Let $R > r> 0$. Then $((\beta_v)_{v \in \Vt(\Delta)}, r, R)$ is called a \emph{reduction datum} if $\beta_v$ is a rescaled Busemann function centered at $v$ for every $v \in \Vt(\Delta)$ and 
 the following properties hold:
\begin{enumerate}
\item the family $(\beta_v)_{v \in \Vt(\Delta)}$ satisfies the nested normal set property;\label{item:reduction_datum_normal_cones}
\item for every point $x \in X$ there is a chamber $c \in \Ch(\Delta)$ that $r$-reduces $x$;\label{item:reduction_datum_existence}
\item for any chamber $c$ that $r$-reduces $x$, the simplex $\sigma_R(x,c)$ is contained in every chamber $c'$ that $r$-reduces $x$.\label{item:reduction_datum_uniqueness}
\end{enumerate}
We say that a reduction datum is \emph{$d$-uniform} if for every subset $B$ of $X$ of diameter at most $d$ there is a chamber $c$ that reduces every point in $B$. Given such a reduction datum and $t \in \R$ we then define
\[
Y_t \defeq \left\{x \in X \mathrel{\Big\vert} \beta_v(x) \le t\ \substack{\textstyle\text{ for some chamber }c\text{ that $r$-reduces }x\\[.3em]\textstyle\text{ and all vertices }v\text{ of }c}\right\}\text{.}
\]
\end{definition}
It follows from \eqref{item:reduction_datum_existence} every reduction datum is $0$-uniform and that the sets $Y_t$ cover $X$.

\begin{definition} Let  $\mathcal R=((\beta_v)_{v \in \Vt(\Delta)}, r, R)$ be a \emph{reduction datum} and let $\Lambda \subset \bfG_S$ be an approximate subgroup.  Then $\mathcal R$ is
\begin{itemize}
\item \emph{$\Lambda$-quasi invariant} if there is a constant $c \ge 0$ such that
\[
\beta_v(x) - c \le \beta_{\gamma.v}(\gamma.x) \le \beta_v(x) + c \quad \text{ for all }x\in X, \gamma \in \Lambda;
\]
\item \emph{$\Lambda$-cobounded} if $Y_t$ is at bounded distance from some (hence any) quasi-orbit $\Lambda.o$ of $\Lambda$ in $X$.
\end{itemize}
\end{definition}
In \cite{bux13} the notions of an \emph{invariant}, respectively \emph{cocompact} reduction datum was introduced for a subgroup $\Gamma < \bfG_S$. The above notions, which apply in the wider context of approximate subgroups, are more general, i.e.\ in the group case invariance implies quasi-invariance and cocompactness implies coboundedness. We also note:

\begin{observation}\label{obs:cobounded_orbit_distance}
Let $\Lambda \subseteq \bfG_S$ be an approximate subgroup. A reduction datum is $\Lambda$-quasi invariant and $\Lambda$-cobounded if and only if for sufficiently large $t$ the space $Y_t$ has finite Hausdorff distance from some (hence any) quasi-orbit $\Lambda.o$.
\end{observation}
We can now state the fundamental theorems of $S$-adic reduction theory (Theorem \ref{thm:lower_s}) and Mahler's compactness criterion (Theorem \ref{thm:mahler_s}) in the following geometric form; recall that $\Lambda \defeq \bfG(\calO_S) \subset \bfG_S$.
\begin{theorem}[Existence of good reduction data]\label{thm:reduction_datum}
In the setting of Convention~\ref{ConventionReductiveCAP}, there exists a $d$-uniform $\Lambda$-quasi invariant reduction datum for every $d>0$.
If $\X[k](\bfG) = 0$ then it can be chosen to be $\Lambda$-cobounded.
\end{theorem}

\begin{proof} As before we denote by $(\beta_v)_{v \in \Vt(\Delta)}$ the quasi-invariant rescaled Busemann functions obtained from Theorem~\ref{thm:busemann_family}, rescaled to satisfy the nested normal set property using Proposition~\ref{prop:nested_normal_cones}. We need to determine constants $R>r>0$ such that $((\beta_v)_{v \in \Vt(\Delta)}, r, R)$ is a reduction datum with the required properties. We will obtain the constants $r$ and $R$ by taking logarithms of the constants $c_1$ and $c_2$ from Theorem~\ref{thm:lower} and Theorem~\ref{thm:uniqueness} into the constants $r$ and $R$ by taking logarithms and applying Lemma~\ref{lem:roots_versus_normal_cones}. Throughout we fix a constant $d \geq 0$.

Thus let $c_1'$ be as in Theorem~\ref{thm:lower}. Let $x \in X$ be a vertex in the $\bfG_S$-orbit of $o$. Let $g \in \bfG_S$ be such that $g.o = x$. Using Lemma~\ref{lem:roots_versus_normal_cones} we can find a constant $r'$ such that $N_{\Sigma,c}(r') \supseteq Z_{\Sigma,c}(\log c_1')$ for every pair of an apartment $\Sigma$ and a chamber $c$ in its boundary. Let $r \le r'$ be such that $N_{\Sigma,c}(r)$ contains the $d$-ball around every closed star of every vertex in $N_{\Sigma,c}(r')$ for every $\Sigma$ and $c$.

Now let $x \in X$ be arbitrary and let $y \in B_d(x)$. Let $g \in \bfG_S$ be such that $x$ lies in the closed star of $g.o$. By Theorem~\ref{thm:lower} there exists a chamber $c \in \Ch(\Delta)$ such that $c_1' \le \nu_i(\bfP_c,{}^gC)$ for all $i$. Chasing definitions this translates into $\log c_1' \le \mu_i^c(g.o)$ and into $x \in Z_{\Sigma,c}(\log c_1')$ where $\Sigma$ is an apartment that contains $x$ and contains $c$ in its boundary. By the choice of constants we have $x \in N_{\Sigma,c}(r')$ and thus $y \in N_{\Sigma,c}(r)$.

Let $c_1$ be such that $Z_{\Sigma,c}(\log c_1)$ contains the closed star of every vertex in $N_{\Sigma,c}(r)$ for all $\Sigma$ and $c$, using Lemma~\ref{lem:roots_versus_normal_cones}. Let $c_2$ be as in Theorem~\ref{thm:uniqueness} for that $c_1$. Let $R'$ be such that $Z_{\Sigma,c}(\log c_2) \supseteq N_{\Sigma,c}(R')$ for all $\Sigma$ and $c$. Let $R \ge R'$ be such that every vertex of $N_{\Sigma,c}(R)$ is contained in the closed star of a vertex in $N_{\Sigma,c}(R')$, for all $\Sigma$ and $c$.

Now let $c \subseteq \partial \Sigma$ and $c' \subseteq \partial \Sigma'$ be chambers that share a vertex $v$, and assume that $x \in N_{\Sigma,c}(r) \cap N_{\Sigma,c'}(r)$ and $x \in N_{\Sigma,v}(R) \cap N_{\Sigma',v'}(R)$. Let $g \in \bfG_S$ be such that $x$ lies in the closed star of $g.o$ and $g.o \in Z_{\Sigma,c}(\log c_1) \cap N_{\Sigma',c'}(\log c_1)$ and $g.o \in N_{\Sigma,v}(R') \cap N_{\Sigma',v'}(R')$. Then $g \in \Omega_{c_1}^{\bfP_{c}} \cap \Omega_{c_1}^{\bfP_{c'}}$ and $\mu^c_i(g.o),\mu^{c'}_i(g.o) \ge \log c_2$ meaning $\nu_i(\bfP_c,{}^gC),\nu_i(\bfP_{c'},{}^gC) \ge c_2$. Theorem~\ref{thm:uniqueness} then ensures that $(\bfP_c)_i = (\bfP_c')_i$, i.e.\ that the vertices of type $i$ of $c$ and $c'$ are the same.

We have shown that $((\beta_v)_{v \in \Vt(\Delta)},r, R)$ is a $d$-uniform reduction datum, and by Theorem \ref{thm:busemann_family} this reduction datum is $\Lambda$-quasi-invariant. Finally we note that $Y_t$ is contained in a bounded neighborhood of $(\bfG_S \cap \bigcup_{\bfP} \Omega_{c_1,c_2}^\bfP).o$, and if $\X[k](\bfG) = 0$, then it is therefore contained in a bounded neighborhood of $\bfG(\calO_S).o$ by Corollary~\ref{cor:mahler_s} and thus $\Lambda$-cobounded.
\end{proof}

\begin{remark} We have assumed throughout this section that $\bfG$ is $k$-isotropic; this implies that  our definition of $Y_t$ coincides with the definition of $Y_t$ on \cite[p.~318]{bux13} by \cite[Remark~12.13]{bux13}. The case that $\bfG$ is $k$-anisotropic is of little interest to us here, since in this case $\Delta$ is empty. Nevertheless, Theorem \ref{thm:reduction_datum} also works in this setting, albeit trivially: The only chamber is the empty simplex and not having any vertices it trivially reduces every point. Thus $Y_t = X$ for all $t$ and $\bfG(\calO_S).o$ is cobounded in $X$.
\end{remark}

\subsection{Comparison with Bux--Köhl--Witzel}\label{CompareBKW}

Apart from inclusion of characteristic $0$ and approximate groups there is one fundamental difference between the notion of a reduction datum in \cite{bux13} and ours: while in \cite{bux13} the scaling constants $(s_i)_{i \in \typ(\Delta)}$ are chosen so as to make the functions $(\beta_v)_v$ actual Busemann functions, we chose them to make the rescaled Busemann functions $(\beta_v)_v$ satisfy the nested normal cones property.

The reason is that, as we have seen, the nested normal cones property is crucial and it is implicitly also used in \cite[Observation~12.7]{bux13}. However, actual Busemann functions do not generally satisfy the nested normal cones property as the following example shows.

\begin{example}\label{ex:counter_example}
The roots and principal weights of the root system of type $D_4$ are the columns of the matrices (see \cite[Planche~IV]{BourbakiLie46})
\[
A = \begin{pmatrix}
1 & -1 & &\\
 & 1 &-1&\\
& & 1 & -1\\
& & 1 & 1
\end{pmatrix}
\quad\text{and}\quad
W = \begin{pmatrix}
1 & 1& 1/2 & 1/2\\
 & 1 & 1/2 & 1/2\\
&  & 1/2 & 1/2\\
& & -1/2 & 1/2
\end{pmatrix}\text{.}
\]

The only principal weight that does not have unit length is the second. Therefore the matrix $\bar{W}$ of normalized weights is $W$ with the second column replaced by $1/\sqrt{2}(1,1,0,0)^T$. The vector $v$ satisfying $v.\bar{W} = (1,1,1,1)$ is $v = (1,\sqrt{2}-1,2-\sqrt{2},0)$. The roots evaluate on it to $v.A = (2-\sqrt{2},2\sqrt{2}-3,2-\sqrt{2},2-\sqrt{2})$ with the second entry being $<0$. So $v$ does not lie in the cone defined by $x.A > 0$. As a remark, note that using $W$ instead of $\bar{W}$ would result in the vector $v' = (1,0,1,0)$ with $v'.A = (1,-1,1,1)$ thus using the principal weights without scaling does not resolve the issue.
\end{example}

\begin{remark}
Calculations for Coxeter types of small rank suggest that $(1,\ldots,1)\bar{W}^{-1}A \ge 0$ for linear root systems ($A_n$, $B_n$, $C_n$, $F_4$, $G_2$) but not for non-linear ones ($D_n$, $E_6$, $E_7$, $E_8$). That is, Busemann functions should work in the former cases but not in the latter. This suggests also that rescaling constants could be described explicitly depending on the Coxeter diagram.
\end{remark}

Now we have seen how our choice of scaling resolves an issue of \cite{bux13}. Since we will eventually use the methods and results from \cite{bux13}, we are left with the task of verifying that our choice of scaling does not break anything in \cite{bux13}. Since the Morse function is based on measuring distance from one fixed set $Y_R$ it turns out that only one statement needs to be adjusted.
Namely \cite[Proposition~2.4]{bux13} remains true as stated but the proof needs to be changed slightly. To explain the change let $(r_i)_{i \in \typ(\Delta)}$ be such that $r_{\typ(v)}\beta_v$ is a Busemann function and define $r_{\text{min}} = \min_{i \in \typ(\Delta)} r_i$. 
Now the set $Y_{s+R}$ needs to be replaced by $Y_{s/r_{\text{min}}+R}$. The displayed chain of inequalities becomes
\begin{multline*}
\textrm{dist}(x_{\Sigma,c},x) \ge r_{\typ(v)}(\beta_v(x) - \beta_v(x_{\Sigma,c})) >\\
r_{\typ(v)}((s/r_{\text{min}}+R) - R) = (r_{\typ(v)}/r_{\text{min}})s \ge s\text{.}
\end{multline*}

%

\section{Proof of the main theorem and discussion}\label{sec:proof_main}

We are now ready to prove our main theorem (Theorem \ref{MainTheorem}). Recall that it is concerned with the following setup:

\begin{convention}\label{ConventionMainThm}
\begin{enumerate}
\item $S$ is a finite, non-empty set of finite places of $k$;
\item $\bfG$ is a non-commutative almost simple $k$-isotropic $k$-group;
\item $\Lambda = \bfG(\calO_S) = \Lambda(G, H, \Gamma, W \cap H)$ is the $S$-arithmetic approximate subgroup of $\bfG$ where $G$, $H$, and $W$ are chosen as before (Convention~\ref{ConventionReductiveCAP}).
\item $X_s$ is the Bruhat--Tits building associated to $\bfG(k_s)$ for $s \in S$ and $X \defeq \prod_{s \in S} X_s$;
\item $d \defeq \sum_{s \in S} \rk_{k_s} \bfG$ is the dimension of $X$.
\end{enumerate}
\end{convention}
Note that $X$ is a product of buildings because $S$ contains only finite places. With these conventions in place the main theorem says that the approximate group $\Lambda$ is of type $F_{d-1}$ but not of type $F_d$. In view of Proposition~\ref{CheckFn} we have to show that the filtration $(N_r(\Lambda.o))_{r \ge 0}$ is essentially $(d-2)$-connected but not essentially $(d-1)$-connected. In view of Lemma~\ref{lem:equivalent_filtrations} we may replace this filtration by an equivalent one, which does not involve $\Lambda$ directly.

To find a more geometric filtration in the equivalence class of $(N_r(\Lambda.o))_{r \ge 0}$ we first observe that the assumptions that $\bfG$ be non-commutative and almost simple imply that $\X[k](\bfG) = 0$, hence  by Theorem~\ref{thm:reduction_datum} there exists a $\Lambda$-quasi invariant and $\Lambda$-cobounded reduction datum $\mathcal R=((\beta_v)_{v \in \Vt(\Delta)}, r, R)$. We fix such a reduction datum once and for all. By Observation~\ref{obs:cobounded_orbit_distance}, the filtration $(Y_t)_{t \ge R}$ is then equivalent to the filtration $(N_r(\Lambda.o))_{r \ge 0}$ by $(Y_t)_{t \ge R}$, hence we see that our main theorem is equivalent to the following geometric statement, which no longer involves the approximate group $\Lambda$ directly, but only the reduction datum $\mathcal R$:
\begin{theorem}\label{MainThmConvenient} The filtration $(Y_t)_{t \ge R}$ is essentially $(d-2)$-connected but not essentially $(d-1)$-connected
\end{theorem}
The remainder of this section is devoted to the proof of Theorem \ref{MainThmConvenient} along the lines of \cite{bux13}. Note that Sections~1--10 of \cite{bux13} do not make any assumptions on the characteristic of the local fields underlying the Bruhat--Tits buildings, or in fact any assumption that the buildings be Bruhat--Tits. Our Sections~\ref{sec:adelic_reduction} and~\ref{sec:s-arithmetic_reduction} will act as replacements of Sections~11 and~12 of \cite{bux13}.

A large part of the proof of Theorem \ref{MainThmConvenient}, as given in \cite{bux13}, consists in replacing the filtration $(Y_t)_{t \ge R}$ by a better behaved equivalent filtration several times. We are going to use exactly the same filtrations as in \cite{bux13}, but we need a few additional arguments to see that these filtrations are still equivalent in our more general setting. In fact, in the setting of \cite{bux13} equivalence of the various filtrations is immediate from the fact that they are all invariant under and cocompact with respect to the same group action (and therefore not even mentioned explicitly), but this argument does not apply in our setting. We therefore trace the proof in \cite{bux13} arguing that the various filtrations considered there remain equivalent even in the absence of a cocompact group action.

We adapt \cite[Section~2]{bux13} using literally the same definitions. For a point $x \in X$, as apartment $\Sigma$ containing $x$ and a chamber $c$ in the boundary of $\Sigma$ that reduces $x$ we define $x_{\Sigma,c}$ to be the metric projection of $x$ to $Y_{\Sigma,c}$ and put $\hat{h}_{\Sigma,c}(x) = d(x,x_{\Sigma,c})$. Proposition~2.1 and Corollary~2.2 of \cite{bux13} apply verbatim to show that $\hat{h}_{\Sigma,c}$ does not depend on $\Sigma$ or $c$, hence we may define $\hat{h}(x) = \hat{h}_{\Sigma,c}(x)$ where $c$ is a chamber reducing $x$ and $\Sigma$ is an apartment containing $x$ and containing $c$ in its boundary.

The following is a metric version of \cite[Proposition~2.4]{bux13}:

\begin{lemma}
For every $s$ there exists a $t$ such that 
\[
\hat{h}^{-1}([0,s]) \subseteq Y_{R+t} \quad \text{and} \quad Y_{R+s} \subseteq \hat{h}^{-1}([0,t])\text{.}
\]
\end{lemma}

\begin{proof}
Let $(r_i)_{i \in \typ(\Delta)}$ be constants such that $r_{\typ(v)}\beta_v$ are actual Busemann functions and let $r_{\text{min}} = \min_{i \in \typ(\Delta)} r_i$. Then
\[
r_{\typ(v)}(\beta_v(x) - R) \le r_{\typ(v)}(\beta_v(x) - \beta_v(x_{\Sigma,c}))\le  d(x,x_{\Sigma,c}) = \hat{h}(x)
\]
for all $v \in \Vt(c)$, and hence $\beta_v(x) \le \frac{1}{r_\text{min}} \hat{h}(x) + R$. In particular $\hat{h}^{-1}([0,s]) \subseteq Y_{R + s/r_{\text{min}}}$.

For the converse direction let us fix an apartment $\Sigma$ and a chamber $c \in \partial \Sigma$. If $\rho$ is a geodesic ray in $\Sigma$ starting in a point $x \in Y_{R}$ and tending to a point $\xi$ in the closure of $c$ then for every vertex $v \in \Vt(c)$ the function $\beta_v \circ \rho$ is affine
\begin{equation}\label{eq:buse_ray}
\beta_v(\rho(t)) = a_{v,\xi} t+b_{v,x,\xi}
\end{equation}
with slope $a_{v,\xi} \ge 0$ because the diameter of $c$ is at most $\pi/2$. Moreover $\max_{v \in \Vt(c)} a_{v,\xi}$ is bounded away from zero because simplices in the barycentric subdivision of $c$ have diameter $< \pi/2$. By compactness of $\bar{c}$ it follows that
\[
u \defeq \min_{\xi \in \bar{c}} \max_{v \in \Vt(c)} a_{v,\xi}
\]
exists and is positive. Taking a ray with $\rho(0) = x_{\Sigma,c}$ and $\rho(d(x,x_{\Sigma,c})) = x$ it follows from \eqref{eq:buse_ray} that 
\[
\hat{h}(x) = d(x,x_{\Sigma,c}) \le \max_{v \in \Vt(c)} \frac{1}{a_{v,\xi}}(\beta_v(x) - \beta_v(x_{\Sigma,c})) \le \frac{1}{u}\max_{v \in \Vt(c)}(\beta_v(x) - \beta_v(x_{\Sigma,c}))\text{.}
\]
By Lemma~\ref{lem:all_chambers_equal} this whole discussion is independent of $\Sigma$ and $c$. It follows that $\hat{h}^{-1}([0,s]) \subseteq Y_{r+s/u}$.
\end{proof}

The lemma states that the filtrations $(Y_t)_{t \ge R}$ and $(\hat{h}^{-1}([0,t]))_{t \ge 0}$ are equivalent so we may work with the latter from now on. We claim that \cite{bux13} contains a proof that the filtration $(\hat{h}^{-1}([0,t]))_{t \ge 0}$ is essentially $(d-2)$-connected and not essentially $(d-1)$-connected.

In \cite[Section~5]{bux13} the height function $\hat{h}$ is replaced by a function $h$. It is clear from the definition that $h \le \hat{h}$ and \cite[Observation~5.5]{bux13} shows that there is a constant $C$ such that $\hat{h} \le h + C$. Hence the filtrations $(\hat{h}^{-1}([0,t]))_{t \ge 0}$ and $(h^{-1}([0,t]))_{t \ge 0}$ are equivalent and we may work with the latter. The ultimate Morse function $f$ on \cite[p.338]{bux13} has $h$ (maximized over a simplex) as its first component. Consequently $(h^{-1}([0,t]))_{t \ge 0}$ and $(f^{-1}(t,0,0))_{t \ge 0}$ are equivalent filtrations. In summary we have established an equivalence of filtrations
\begin{equation}\label{FiltEq}
(Y_t)_{t \ge R} \sim (\hat{h}^{-1}([0,t]))_{t \ge 0} \sim (h^{-1}([0,t]))_{t \ge 0} \sim (f^{-1}(t,0,0))_{t \ge 0}.
\end{equation}

\begin{proof}[Proof of Theorem \ref{MainThmConvenient}] In view of \eqref{FiltEq} it remains to show only  that the filtration of $X$ by sublevel sets of the function $f$ is essentially $(d-2)$-connected but not essentially $(d-1)$-connected. This is precisely what is established (in greater generality) on \cite[p.~345]{bux13}.
\end{proof}
At this point we have concluded the proof of our main theorem.

\section{Reductive, anistoropic, and commutative groups}\label{sec:reductive}

In this short section we discuss the necessity of the assumptions in our main theorem; we also establish a number of variants, including Corollary~\ref{cor:reductive}. 

\begin{remark}[The anisotropic case]\label{AnistropicCase}
In the previous section we have only considered the case in which $\bfG$ is $k$-isotropic. If $\bfG$ is $k$-anisotropic then $\bfG(k)$ is a uniform lattice in $\bfG(\A)$ by  \cite[Theorem~5.5]{PlatonovRapinchuk}. This statement can be seen as a tautological form of our Theorem~\ref{thm:mahler}. Now a direct application of Proposition~\ref{prop:descent} shows that $\Lambda$ is relatively dense in $\bfG_S$. Thus $\Lambda$ is coarsely equivalent to $\bfG_S$ and hence to $X_S$ which is contractible. This shows that $\Lambda$ is of type $F_\infty$ in this case.
\end{remark}

This shows that the assumption that $\bfG$ be isotropic is necessary.

\begin{remark}[The abelian case]
If $\bfG$ is absolutely isomorphic to $\GL_1$ (a $1$-dimensional torus) then it is either anisotropic ($k$-rank $0$) so that $\Lambda$ is of type $F_\infty$ by Remark \ref{AnistropicCase}, or it is $k$-isomorphic to $\GL_1$ ($k$-rank $1$). Writing $S^+ = S \cup V\inf$ and $\Lambda^+ = \bfG(\calO_{S^+})$, the group $\Lambda^+$ is virtually free abelian by the Dirichlet unit theorem \cite[Theorem p.72]{CasselsFroehlich}. It follows that $\Lambda$ is a discrete approximate subgroup of some $\R^d$ and therefore is relatively dense in some linear subspace by \cite[Proposition~II.2, p.~315]{Schreiber73}, \cite[Theorem~2.2]{Fish19}.
In particular, it is of type $F_\infty$.

If $\bfG$ is absolutely isomorphic to the additive group then there is $k$-epi\-mor\-phism to the additive group with finite kernel \cite[Theorem~2.1]{Russell70} (if $k$ is a number field, it is an isomorphism). It therefore suffices to consider the case where $\bfG$ is the additive group in which case we claim that $\Lambda = \calO_S$ is not coarsely connected.
Indeed, $\Lambda$ consists of elements $g \in k$ with $\abs{g}_s \le 1$ for $s \in V\fin \setminus S$ and $\abs{g}_s \le c_s$ for $s \in V\inf \setminus S$ where $c_s \ge 0$ is some constant. Pick a finite place $s \in S$. Now for $r \ge 0$ arbitrarily large we can use the Chinese remainder theorem (i.e.\ strong approximation for the additive group) to find an element $g \in \calO_S$ with $\abs{g}_{s} \ge r$, and hence $d_{k_s}(0,g) = \abs{g - 0}_s \ge r$. Since $k_s$ is ultrametric, it follows that there is no sequence of elements from $0=g_0,\ldots,g_k=g$ with $\abs{g_i - g_{i-1}}_s \le r$. This shows that $\Lambda$ is not coarsely connected thus not of type $F_1$.
\end{remark}

This shows that the assumption that $\bfG$ be non-abelian is necessary.

We now come to the proof of Corollary~\ref{cor:reductive}.

\begin{proof}[Proof of Corollary~\ref{cor:reductive}]
Let $\bfG^0$ denote the connected component of $\bfG$ and let $\tilde{\bfG}$ be the universal cover of $\bfG^0$, i.e.\ the simply connected $k$-group with central isogeny $\tilde{\bfG} \to \bfG^0$. Then $\tilde{\bfG}$ is an almost direct product of a central torus $\bfZ$ and its derived subgroup $\mathscr{D} \tilde{\bfG}$ which is semisimple \cite[Proposition~2.2]{BorelTits65}. There is a finite extension $\ell/k$ and a semisimple $\ell$-group $\bfH$ such that $\mathscr{D} \tilde{\bfG} \cong R_{\ell/k} \bfH$ as $k$-groups and such that $\bfH \cong \bfH_1 \times \ldots \times \bfH_k$ where the $\bfH_i$ are \emph{absolutely} almost simple, see \cite[6.21(ii)]{BorelTits65}. Let $T$ be the set of places of $\ell$ that lie above places in $S$. 

The Main~Theorem asserts that $\bfH_i(\calO_T)$ is of type $F_{d_i-1}$ but not of type $F_{d_i}$ where $d_i = \sum_{s \in T} \rk_{\ell_s} \bfH_i$. Hence $(\bfZ \times \bfH)(\calO_T)$ is of type $F_{d-1}$ but not of type $F_d$ where $d = \infimum_i d_i$ by Lemma~\ref{lem:coarse_product}. The infimum is understood to be $\infty$ if $k = 0$.

Now we claim that the composition
\begin{equation}\label{eq:isogeny_coarse_equiv}
(\bfZ \times \bfH)(\calO_T) \to (\bfZ \times \mathscr{D} \tilde{\bfG})(\calO_S) \to \tilde{\bfG}(\calO_S) \to \bfG^0(\calO_S) \to \bfG(\calO_S)
\end{equation}
is a coarse equivalence, which will conclude the proof in view of Corollary~\ref{cor:coarse_connected_invariant}.

Put $S^+ = S \cup V_k\inf$ and $T^+ = T \cup V_\ell\inf$ and note that $T^+$ is the set of places above $S^+$. Note that for every $k$-group $\bfL$ the $S$-arithmetic approximate group $\bfL(\calO_S)$ is a model set of the cut-and-project scheme $(\bfL_S,\bfL_{S^+\setminus S},\bfL(\calO_{S^+}))$ and similarly for $T$. Using Lemma~\ref{lem:coarse_equiv} it therefore suffices to show that the composition \eqref{eq:isogeny_coarse_equiv} has finite kernel and finite-index image when $S$ is replaced by $S^+$ and $T$ is replaced by $T^+$. We verify this from right to left: $\bfG^0(\calO_{S^+})$ has finite index in $\bfG(\calO_{S^+})$; the map $\tilde{\bfG}(\calO_{S^+}) \to \bfG^0(\calO_{S^+})$ has finite kernel and finite-index image by \cite[Satz~1]{Behr68}; the same argument applies to the map $\bfZ \times \mathscr{D}\tilde{\bfG} \to \tilde{\bfG}$; $(\bfZ \times \bfH)(\calO_{T^+})$ and $(\bfZ \times \mathscr{D} \tilde{\bfG})(\calO_{S^+})$ are commensurable by \cite[\S 1.3]{Weil82}, see also \cite[Lemma~I.3.1.4]{Margulis91}.
\end{proof}

\providecommand{\bysame}{\leavevmode\hbox to3em{\hrulefill}\thinspace}
\providecommand{\MR}{\relax\ifhmode\unskip\space\fi MR }
\providecommand{\MRhref}[2]{%
  \href{http://www.ams.org/mathscinet-getitem?mr=#1}{#2}
}
\providecommand{\href}[2]{#2}


\begin{thebibliography}{BKW13}

\bibitem[AA93]{AbelsAbramenko}
Herbert Abels and Peter Abramenko, \emph{On the homotopy type of subcomplexes
  of {T}its buildings}, Adv. Math. \textbf{101} (1993), no.~1, 78--86.

\bibitem[AB08]{abramenko08}
Peter Abramenko and Kenneth~S. Brown, \emph{Buildings: Theory and
  applications}, Graduate Texts in Mathematics, vol. 248, Springer, 2008.

\bibitem[Abe87]{Abels87}
Herbert Abels, \emph{Finite presentability of {$S$}-arithmetic groups.
  {C}ompact presentability of solvable groups}, Lecture Notes in Mathematics,
  vol. 1261, Springer-Verlag, Berlin, 1987.

\bibitem[Abr96]{Abramenko}
Peter Abramenko, \emph{Twin buildings and applications to {S}-arithmetic
  groups}, Lecture Notes in Mathematics, vol. 1641, Springer-Verlag, Berlin,
  1996.

\bibitem[Alo94]{alonso94}
Juan~M. Alonso, \emph{Finiteness conditions on groups and quasi-isometries}, J.
  Pure Appl. Algebra \textbf{95} (1994), no.~2, 121--129.

\bibitem[Beh68]{Behr68}
Helmut Behr, \emph{Zur starken {A}pproximation in algebraischen {G}ruppen
  \"{u}ber globalen {K}\"{o}rpern}, J. Reine Angew. Math. \textbf{229} (1968),
  107--116.

\bibitem[BH99]{BridsonHaefliger}
Martin~R. Bridson and Andr\'{e} Haefliger, \emph{Metric spaces of non-positive
  curvature}, Grundlehren der Mathematischen Wissenschaften, vol. 319,
  Springer-Verlag, Berlin, 1999.

\bibitem[BH18]{BH}
Michael Bj\"{o}rklund and Tobias Hartnick, \emph{Approximate lattices}, Duke
  Math. J. \textbf{167} (2018), no.~15, 2903--2964.

\bibitem[BHP18]{BHP1}
Michael Bj\"{o}rklund, Tobias Hartnick, and Felix Pogorzelski, \emph{Aperiodic
  order and spherical diffraction, {I}: auto-correlation of regular model
  sets}, Proc. Lond. Math. Soc. (3) \textbf{116} (2018), no.~4, 957--996.

\bibitem[BKW13]{bux13}
Kai-Uwe Bux, Ralf K\"ohl, and Stefan Witzel, \emph{Higher finiteness properties
  of reductive arithmetic groups in positive characteristic: the {R}ank
  {T}heorem}, Ann. of Math. (2) \textbf{177} (2013), no.~1, 311--366.

\bibitem[Bor63a]{Borel63}
Armand Borel, \emph{Some finiteness properties of adele groups over number
  fields}, Inst. Hautes \'{E}tudes Sci. Publ. Math. (1963), no.~16, 5--30.

\bibitem[Bor63b]{Borel}
\bysame, \emph{Some finiteness properties of adele groups over number fields},
  Inst. Hautes \'{E}tudes Sci. Publ. Math. (1963), no.~16, 5--30.

\bibitem[Bou07]{BourbakiLie46}
N.~Bourbaki, \emph{Groupes et algèbres de {L}ie. {C}hapitres 4 à 6},
  Éléments de Mathématique, Springer, 2007.

\bibitem[Bro87]{brown87}
Kenneth~S. Brown, \emph{Finiteness properties of groups}, Proceedings of the
  {N}orthwestern conference on cohomology of groups ({E}vanston, {I}ll., 1985),
  vol.~44, 1987, pp.~45--75.

\bibitem[Bro89]{Brown89}
\bysame, \emph{Buildings}, Springer-Verlag, New York, 1989.

\bibitem[BS76]{BorelSerre76}
A.~Borel and J.-P. Serre, \emph{Cohomologie d'immeubles et de groupes
  {$S$}-arithm\'{e}tiques}, Topology \textbf{15} (1976), no.~3, 211--232.

\bibitem[BT65]{BorelTits65}
Armand Borel and Jacques Tits, \emph{Groupes r\'{e}ductifs}, Inst. Hautes
  \'{E}tudes Sci. Publ. Math. (1965), no.~27, 55--150.

\bibitem[BT72]{BruhatTits1}
F.~Bruhat and J.~Tits, \emph{Groupes r\'{e}ductifs sur un corps local}, Inst.
  Hautes \'{E}tudes Sci. Publ. Math. (1972), no.~41, 5--251.

\bibitem[BT84]{BruhatTits2}
\bysame, \emph{Groupes r\'{e}ductifs sur un corps local. {II}. {S}ch\'{e}mas en
  groupes. {E}xistence d'une donn\'{e}e radicielle valu\'{e}e}, Inst. Hautes
  \'{E}tudes Sci. Publ. Math. (1984), no.~60, 197--376.

\bibitem[Bux04]{Bux04}
Kai-Uwe Bux, \emph{Finiteness properties of soluble arithmetic groups over
  global function fields}, Geom. Topol. \textbf{8} (2004), 611--644.

\bibitem[BW07]{BuxWortman07}
Kai-Uwe Bux and Kevin Wortman, \emph{Finiteness properties of arithmetic groups
  over function fields}, Invent. Math. \textbf{167} (2007), no.~2, 355--378.

\bibitem[BW11]{BuxWortman11}
\bysame, \emph{Connectivity properties of horospheres in {E}uclidean buildings
  and applications to finiteness properties of discrete groups}, Invent. Math.
  \textbf{185} (2011), no.~2, 395--419.

\bibitem[CdlH16]{CornulierDeLaHarpe16}
Yves Cornulier and Pierre de~la Harpe, \emph{Metric geometry of locally compact
  groups}, EMS Tracts in Mathematics, vol.~25, European Mathematical Society
  (EMS), Z\"{u}rich, 2016, Winner of the 2016 EMS Monograph Award.

\bibitem[CF10]{CasselsFroehlich}
J.~W.~S. {Cassels} and A.~{Fr\"ohlich} (eds.), \emph{Algebraic number theory},
  2nd edition ed., London: London Mathematical Society, 2010 (English).

\bibitem[CHT]{CHT}
Matthew Cordes, Tobias Hartnick, and Vera Tonić, \emph{Foundations of
  geometric approximate group theory}, arXiv:2012.15303v2.

\bibitem[Con]{Conrad}
Brian Conrad, \emph{Weil and {G}rothendieck approaches to adelic points},
  \url{http://math.stanford.edu/~conrad/papers/adelictop.pdf}.

\bibitem[Fis19]{Fish19}
Alexander Fish, \emph{Extensions of {S}chreiber's theorem on discrete
  approximate subgroups in {$\mathbb{R}^d$}}, J. \'{E}c. polytech. Math.
  \textbf{6} (2019), 149--162.

\bibitem[Gan12]{Gandini12}
Giovanni Gandini, \emph{Bounding the homological finiteness length}, Bull.
  Lond. Math. Soc. \textbf{44} (2012), no.~6, 1209--1214.

\bibitem[God64]{godement64}
Roger Godement, \emph{Domaines fondamentaux des groupes arithm\'etiques},
  S\'eminaire {B}ourbaki, 1962/63. {F}asc. 3, {N}o. 257, Secr\'etariat
  math\'ematique, Paris, 1964, p.~25.

\bibitem[Har67]{harder67}
G\"{u}nter Harder, \emph{Halbeinfache {G}ruppenschemata \"{u}ber
  {D}edekindringen}, Invent. Math. \textbf{4} (1967), 165--191.

\bibitem[Har68]{harder68}
\bysame, \emph{Halbeinfache {G}ruppenschemata \"{u}ber vollst\"{a}ndigen
  {K}urven}, Invent. Math. \textbf{6} (1968), 107--149.

\bibitem[Har69]{harder69}
G.~Harder, \emph{Minkowskische {R}eduktionstheorie \"{u}ber
  {F}unktionenk\"{o}rpern}, Invent. Math. \textbf{7} (1969), 33--54.

\bibitem[Har77]{harder77}
\bysame, \emph{Die {K}ohomologie {$S$}-arithmetischer {G}ruppen \"{u}ber
  {F}unktionenk\"{o}rpern}, Invent. Math. \textbf{42} (1977), 135--175.

\bibitem[Hat02]{hatcher}
Allen Hatcher, \emph{Algebraic topology}, Cambridge University Press,
  Cambridge, 2002.

\bibitem[Hel78]{Helgason}
Sigurdur Helgason, \emph{Differential geometry, {L}ie groups, and symmetric
  spaces}, Pure and Applied Mathematics, vol.~80, Academic Press, Inc., New
  York-London, 1978.

\bibitem[Hru]{Hrushovski}
Ehud Hrushovski, \emph{Beyond the {L}ascar {G}roup}, arXiv:2011.12009v3.

\bibitem[KP]{KalethaPrasad}
Tasho Kaletha and Gopal Prasad, \emph{Bruhat--tits theory: a new approach}.

\bibitem[LMR00]{LMR}
Alexander Lubotzky, Shahar Mozes, and M.~S. Raghunathan, \emph{The word and
  {R}iemannian metrics on lattices of semisimple groups}, Inst. Hautes
  \'{E}tudes Sci. Publ. Math. (2000), no.~91, 5--53 (2001).

\bibitem[LY21]{LeuzingerYoung21}
Enrico Leuzinger and Robert Young, \emph{Filling functions of arithmetic
  groups}, Ann. of Math. (2) \textbf{193} (2021), no.~3, 733--792.

\bibitem[Mac]{Machado}
Simon Machado, \emph{Approximate {L}attices in {H}igher-{R}ank {S}emi-{S}imple
  {G}roups}, arXiv:2011.01835v4.

\bibitem[Mar91]{Margulis91}
G.~A. Margulis, \emph{Discrete subgroups of semisimple {L}ie groups},
  Ergebnisse der Mathematik und ihrer Grenzgebiete (3), vol.~17,
  Springer-Verlag, Berlin, 1991.

\bibitem[Mey72]{Meyer}
Yves Meyer, \emph{Algebraic numbers and harmonic analysis}, North-Holland
  Mathematical Library, Vol. 2, North-Holland Publishing Co., Amsterdam-London;
  American Elsevier Publishing Co., Inc., New York, 1972.

\bibitem[Nie15]{niesdroy15}
Henning Niesdroy, \emph{Geometric reduction theory}, Ph.D. thesis, Bielefeld
  University, Bielefeld, 2015.

\bibitem[PR94]{PlatonovRapinchuk}
Vladimir Platonov and Andrei Rapinchuk, \emph{Algebraic groups and number
  theory}, Pure and Applied Mathematics, vol. 139, Academic Press, Inc.,
  Boston, MA, 1994, Translated from the 1991 Russian original by Rachel Rowen.

\bibitem[Pra77]{prasad77}
Gopal Prasad, \emph{Strong approximation for semi-simple groups over function
  fields}, Ann. of Math. (2) \textbf{105} (1977), no.~3, 553--572.

\bibitem[Roe03]{Roe}
John Roe, \emph{Lectures on coarse geometry}, University Lecture Series,
  vol.~31, American Mathematical Society, Providence, RI, 2003.

\bibitem[Ros22]{Rosendal}
Christian Rosendal, \emph{Coarse geometry of topological groups}, Cambridge
  Tracts in Mathematics, vol. 223, Cambridge University Press, Cambridge, 2022.

\bibitem[Rus70]{Russell70}
Peter Russell, \emph{Forms of the affine line and its additive group}, Pacific
  J. Math. \textbf{32} (1970), 527--539.

\bibitem[Sch]{Schesler}
Eduard Schesler, \emph{The ${\Sigma}$-invariants of ${S}$-arithmetic subgroups
  of {B}orel groups}, arXiv:2203.10132.

\bibitem[Sch73]{Schreiber73}
Jean-Pierre Schreiber, \emph{Approximations diophantiennes et probl\`emes
  additifs dans les groupes ab\'{e}liens localement compacts}, Bull. Soc. Math.
  France \textbf{101} (1973), 297--332.

\bibitem[Tao08]{Tao}
Terence Tao, \emph{Product set estimates for non-commutative groups},
  Combinatorica \textbf{28} (2008), no.~5, 547--594.

\bibitem[Tit66]{TitsClassSemiSimple}
J.~Tits, \emph{Classification of algebraic semisimple groups}, Algebraic
  {G}roups and {D}iscontinuous {S}ubgroups ({P}roc. {S}ympos. {P}ure {M}ath.,
  {B}oulder, {C}olo., 1965), Amer. Math. Soc., Providence, R.I., 1966, 1966,
  pp.~33--62.

\bibitem[Tit79]{tits79}
\bysame, \emph{Reductive groups over local fields}, Automorphic forms,
  representations and {$L$}-functions ({P}roc. {S}ympos. {P}ure {M}ath.,
  {O}regon {S}tate {U}niv., {C}orvallis, {O}re., 1977), {P}art 1, Proc. Sympos.
  Pure Math., XXXIII, Amer. Math. Soc., Providence, R.I., 1979, pp.~29--69.

\bibitem[Wei82]{Weil82}
Andr\'{e} Weil, \emph{Adeles and algebraic groups}, Progress in Mathematics,
  vol.~23, Birkh\"{a}user, Boston, Mass., 1982, With appendices by M. Demazure
  and Takashi Ono.

\bibitem[Wit13]{Witzel13}
Stefan Witzel, \emph{Abels's groups revisited}, Algebr. Geom. Topol.
  \textbf{13} (2013), no.~6, 3447--3467.

\bibitem[Wit14]{Witzel14}
\bysame, \emph{Finiteness properties of arithmetic groups acting on twin
  buildings}, Lecture Notes in Mathematics, vol. 2109, Springer, Cham, 2014.

\end{thebibliography}
\end{document}